\documentclass[a4paper,11pt,reqno]{amsart}

\usepackage[utf8]{inputenc}
\usepackage{pifont}
\usepackage[english]{babel}
\usepackage{amsmath}
\usepackage{amsthm}
\usepackage{amsfonts,mathrsfs,relsize,bigints,stmaryrd}
\usepackage{amssymb}
\usepackage{anysize}
\usepackage{hyperref}
\usepackage{appendix}
\usepackage{mathtools}
\usepackage{graphicx}
\usepackage{xcolor}
\usepackage[normalem]{ulem}
\makeatletter  \@addtoreset{equation}{section} \makeatother
\newtheorem{lem}{Lemma}[section]
\newtheorem{theo}{Theorem}[section]

\newtheorem{pro}{Proposition}[section]

\DeclareMathOperator{\N}{\mathbb{N}}
\DeclareMathOperator{\R}{\mathbb{R}}

\DeclareMathOperator{\T}{\mathbb{T}}
\DeclareMathOperator{\D}{\mathbb{D}}
\DeclareMathOperator{\Z}{\mathbb{Z}}

\newcommand\E{\varepsilon}

\allowdisplaybreaks 

{\thanks{C.G. has been supported by the ERC-StG-852741 (CAPA), the MINECO--Feder (Spain) research grant number RTI2018--098850--B--I00, the Junta de Andaluc\'ia (Spain) Project
FQM 954, the Severo Ochoa Programme for Centres of Excellence in R\&D(CEX2019-000904-S) and by the PID2021-124195NB-C32, the  ERC Advanced Grant 834728 and by the CAM under the multiannual Agreement with UAM in the line for the Excellent of the University Research Staff in the context of the V PRICIT. T. H. has been supported by Tamkeen under the NYU Abu Dhabi Research Institute grant.    { J. M.  has been partially supported by PID2020-112881GB-I00 and Severo Ochoa and Maria de Maeztu program for centers CEX2020-001084-MMTM2016--75390 (Mineco, Spain)}}


\begin{document}

\title[time periodic solutions around radial monotonic profiles]{Time periodic solutions close to    localized radial monotone profiles  for the 2D Euler equations }

\author[C. Garc\'ia]{Claudia Garc\'ia}
\address{ Departamento de Matem\'aticas, Universidad Aut\'onoma de Madrid, Ciudad Universitaria de Cantoblanco, 28049, Madrid, Spain \& Research Unit ``Modeling Nature'' (MNat), Universidad de Granada, 18071 Granada, Spain}
\email{ claudia.garcial@uam.es}
\author[T. Hmidi]{Taoufik Hmidi}
\address{NYUAD Research Institute, New York University Abu Dhabi, PO BOX 129188, Abu Dhabi, United Arab Emirates. Univ Rennes, CNRS, IRMAR-UMR 6625, F-35000 Rennes, France}
\email{thmidi@univ-rennes1.fr}
\author[J. Mateu]{Joan Mateu}
\address{Departament de Matem\`atiques, Universitat Aut\`onoma de Barcelona, 08193 Bellaterra; Barcelona, Catalonia}
\email{mateu@mat.uab.cat}


\begin{abstract}
In this paper, we address  for the 2D Euler equations  the existence of rigid time periodic solutions close to  stationary  radial vortices of type $f_0(|x|){\bf 1}_{\D}(x)$, with $\D$ the unit disc and $f_0$  being a strictly monotonic profile with constant sign. We distinguish two scenarios  according to  the sign of the  profile: {\it defocusing and focusing.} In the first regime, we have scarcity of the bifurcating curves associated with lower symmetry. However in the {\it focusing case} we get a countable family of bifurcating solutions associated with large symmetry. The approach  developed in this work is new and flexible, and  the explicit expression of the radial profile is no longer  required as in \cite{GHS:2020} with the quadratic shape. The alternative for that is a refined study of the associated  spectral problem based on  Sturm-Liouville differential equation with a variable potential  that changes the sign depending on the shape of the  profile  and the location of the time period. Deep hidden  structure on positive definiteness of some intermediate integral operators are also discovered and used in a crucial way. Notice that a special study will be performed  for  the linear problem associated with the first mode founded on  Pr\"{u}fer transformation and Kneser's  Theorem  on the non-oscillation phenomenon.  
\end{abstract}

\maketitle

\tableofcontents
\section{Introduction}
The motion of a two-dimensional ideal homogeneous incompressible fluid follows the 2D Euler equations, whose vorticity-velocity formulation reads as 
\begin{eqnarray}   \label{Eulereq}	           
       \left\{\begin{array}{ll}
          	\partial_t\omega+(v\cdot \nabla) \omega=0, &\text{ in $[0,+\infty)\times\mathbb{R}^2$}, \\
         	 v=K*\omega,& \\
         	 \omega(0,\cdot)=\omega_0,& \text{ in $\mathbb{R}^2$}.
       \end{array}\right.
\end{eqnarray}
 The second equation is known as the Biot-Savart law and links the velocity to the vorticity, with $K(x)=\frac{1}{2\pi}\frac{x^\perp}{|x|^2}$. In \cite{yudovich},  Yudovich proved the  global existence and uniqueness of solutions for integrable and bounded initial data. Moreover, these solutions are known to be Lagrangian and they can be recovered from their initial data and the mapping flow that describes the trajectories. A particular class is given by vortex patches where the vorticity is uniformly distributed in a bounded domain, that is, $\omega(t)={\bf{1}}_{D_t}$. In this case the dynamics reduces to the motion of one or multiple interfaces in the plane propelled by the self-induction and the interaction mechanisms. The global in time  persistence  of their boundary regularity in H\"older spaces $\mathscr{C}^{k,\alpha}$, with $k\geq 1$ and $\alpha\in(0,1)$,  was solved in  \cite{bertozzi-constantin, chemin, serfati}. Very recently, an ill-posedness result in  $\mathscr{C}^2$ has been obtained by Kiselev and Luo \cite{kiselev-luo}. For well-posedness and ill-posedness results with singular vortex patches, we refer to the recent papers of Elgindi and Jeong \cite{Elg2,Elg1}.

The 2D Euler equations can be seen as a Hamiltonian system, and thus it is quite natural from a dynamical system point of view to explore whether  time  periodic solutions around specific equilibrium states  may exist. This is a traditional subject in fluid dynamics with   a long history  and  a lot of contributions have been made over the past decades covering  several rich aspects on vortex motion. Explicit steady solutions in the patch form  are well known in the literature: the Rankine vortex (the circular patch) is stationary whereas the Kirchhoff ellipse \cite{kirchhoff:book} performs uniform rotation about its center with a constant  angular velocity related to its  aspect ratio. In 1978, Deem and Zabusky \cite{dz:vstates} gave some numerical evidences of the existence of non trivial rotating patches (also called V-states) living  {\it close} to the Rankine vortex. Some years later, this numerical  conjecture was analytically proved by Burbea \cite{burbea:motions}  using bifurcation theory \cite{kielhofer} and conformal maps. Unfortunately,  Burbea's work escaped the attention of PDE's community for long time and we have to wait  for around  thirty years  before the emergence of  intensive and  rich activity dealing with  the construction of   periodic solutions and the analysis of the  local structure of the bifurcation diagram associated with  different topological structures, see for instance \cite{ccg:reg, CCGS-2016-2,Hoz-Hassainia-Hmidi-Mateu:disc,Hmidi-Mateu:kirchhoff,Hmidi-Mateu:degenerate,hmv:reg,hmv:doubly, Hmidi-Mateu-Verdera:doubly}  and the references therein.

The Burbea patches rotate with an angular velocity $\Omega$ in $(0,\frac12)$, and then there have been many works studying the situation outside that range. First, Fraenkel \cite{Fraenkel} proved that the only simply-connected stationary patch ($\Omega=0$) is the Rankine vortex. Later, Hmidi \cite{hmidi:trivial} proved that if $\Omega<0$ (supplemented with a convexity assumption) and $\Omega=\frac12$ the only simply-connected rotating patch must be the disc. Finally, G\'omez-Serrano, Park, Shi and Yao \cite{GSPSY-rigidity} closed this question and proved that non trivial rotating patches cannot be found outside $[0,\frac12]$. Moreover, they  generalize Fraenkel's result to non constant vorticity and prove that any stationary compactly supported {\it smooth} solution with fixed sign of the 2D Euler equation must be radially symmetric up to a translation (here, they can include vorticities that are smooth in a bounded domain with a possible jump at the boundary). Their result is coherent with the vortex axi-symmetrization work by Bedrossian, Coti Zelati and Vicol \cite{Bedrossian-CotiZelati-Vicol:2019} who analyze the incompressible two dimensional Euler equations linearized
around a smooth radially symmetric, strictly monotone decreasing vorticity distribution. They show an inviscid damping phenomenon and get that the vorticity  converges weakly  to some 
radial symmetric profile  for large time.

In the literature, there are several works about the existence of nontrivial (non radial) stationary solutions, which are not patches. Those approach are based on the study of the characteristic trajectories of the system and are connected to the elliptic equation $\Delta \psi=\omega=F(\psi)$. In that context, Nadirashvilii \cite{Nadirashvili} studied the curvature of streamlines of smooth solutions. The local structure of stationary solutions in the nondegenerate case was explored by Choffrut and Sver\'ak in \cite{Sverak12}. 
Focusing on the elliptic equation, Ru\'iz \cite{Ruiz} proved some   symmetry results for compactly supported steady solutions, generalizing the work \cite{GSPSY-rigidity}. Recently, G\'omez-Serrano, Park and Shi \cite{GSPS:2021} constructed a family of nontrivial compactly supported stationary solutions with finite energy, using perturbative arguments and Nash-Moser scheme. Another interesting study has been conducted recently  by Coti Zelati, Elgindi and Widmayer  on    stationary solutions close to  shear flows of Kolmogorov and Poiseuille type in the periodic setting \cite{Elg3}. Flexibility and rigidity theorems of Liouville type for stationary flows  have been discovered recently by Constantin, Drivas and Ginsberg \mbox{in \cite{Constan21}.}

Notice that all the aforementioned  works on the V-states  concern connected patches. However, the situation is different when one looks for steady disconnected patches where the bifurcation is not well-adapted. Notice that  a few examples are known in the literature and one   of them was reported by  Lamb in \cite{lamb} who found a nontrivial example of touching counter-rotating pairs, where the vorticity inside the domain is not constant but given by a smooth function related to Bessel functions. The alternative for the bifurcation theory is the desingularization of steady point vortex system. This technique    was introduced by Marchioro and Pulvirenti \cite{mp:vortex} in another context to approximate in a weak sense Euler solutions by a vortex point system.  The equations governing  the point vortex model is  a collection of nonlinear  ODE's with singular potential characterizing the evolution of Dirac masses. This finite dimensional  system  admits a lot of steady states. For instance, two point vortices are steady: either they rotate with a constant angular velocity or they translate with a constant speed.
Using variational approach arguments, Turkington \cite{turkington:nfold} constructed pair of co-rotating patches. Similar construction with smooth profiles has been implemented in different contexts,  see  \cite{CLW:2014, CLZ:2021, CW:2022, DPMW:2020, SV:2010}. It seems that  the variational approach  does not give enough  information on  the topology and the geometry of the patch. To remedy to this defect,  Hmidi and Mateu \cite{hm:pairs} performed from the contour dynamics equation an {\it{ad hoc} } desingularization procedure in an infinite dimensional function space leading to the existence of co-rotating and counter-rotating convex and smooth vortex patches emanating from the vortex pairs. Their method is robust and flexible and it was used to cover more configurations such as the desingularization of  Thomson polygons \cite{garcia:choreography}, K\'arm\'an Vortex Street \cite{garcia:karman}, general patterns like nested polygons  \cite{Hassainia-Wheeler-points}. It was also successfully used to generate  asymmetric vortex pairs \cite{HH:asymmetric-pairs} or to achieve analogous construction   for more general active scalar equations \cite{hm:pairs}.

The common feature of the aforementioned long-lived structures is in their construction which is devised by perturbing  steady solutions. In general, we are able to view these solutions as one or more local branches emerging from the steady state. Thus, exploring the global behavior by tracking these branches is a question of great interest but  quite involved  leading to various open scenarios and questions.  
As an example, it is known from the numerical experiments \cite{Overman} that all the Burbea branches, excepted the ellipses, end with a singular patch with corners of angle $\frac\pi2.$ Recently,  Hassainia, Masmoudi and Wheeler \cite{hmw:global} proved through global bifurcation arguments, related to the work of Buffoni and Toland \cite{bt:analytic},  that at the end of the branches   the angular velocity of the patch must vanish  at some point located at the boundary, which is coherent with the formation of corners. The second result concerning this topic is the work of Garc\'ia and Haziot \cite{GH:2022} about the global bifurcation for the corotating vortex pairs found by Hmidi and Mateu in \cite{hm:pairs}. They obtained a self-intersection of the pair of patches together with the vanishing angular velocity at the end of the curve. The singularity given by the point vortices brings  an extra complexity to the problem which requires to adapt in a suitable way  the classical global analytic theorem in \cite{bt:analytic}.

Over the past few years, more development around time periodic vortex patches  has been implemented   to  other two-dimensional active scalar equations such as the generalized surface quasi-geostrophic equation or the quasi-geostrophic shallow water equations. In those systems, steady patches have been investigated  through the contour dynamics equations paired with  bifurcation theory or the implicit function theorem, see \cite{ CCGS-2016,  CCGS-2016-2,CCGS-2019, CCGS-2020, Hoz-Hassainia-Hmidi:doubly-gSQG, Hoz-Hassainia-Hmidi-Mateu:disc,  Hoz-Hmidi-Mateu-Verdera:doubly-euler, Dristchel-Hmidi-Renault, garcia:karman,garcia:choreography, GH:2022, GHM:doubly, GHM:2022, GHS:2020,GSPSY-sheets, Hassainia-Hmidi:vstates-gSQG, HH:asymmetric-pairs, Hassainia-Hmidi-Masmoudi:2021, hmw:global, Hassainia-Roulley:2022,  Hassainia-Wheeler-points, Hmidi-Mateu:kirchhoff, Hmidi-Mateu:degenerate, hm:pairs,Hmidi-Mateu-Verdera:doubly, Roulley:2022}. Variational tools in the spirit of  Turkington approach were also performed for active scalar equations as in \cite{ADPM:2021,CWWZ:2021,CQZZ:2021,CQZZ:2022, GCL:2021, GC:2021, GC:2022}. 

Very recently, some progress opening new perspectives  has been done on the existence of time quasi-periodic solutions close to Rankine vortices using KAM techniques and Nash Moser scheme. Indeed, Berti, Hassainia and Masmoudi confirm in \cite{Berti-Hassainia-Masmoudi:2022} these structures for Euler equations  near  any Kirchhoff ellipse provided that its eccentricity belongs to a Cantor like set. Similar results were obtained by Hmidi, Hassainia and Masmoudi \cite{Hassainia-Hmidi-Masmoudi:2021}  for the generalized $(\hbox{SQG})_\alpha$ equations, provided that the exponent of the fractional Laplacian lies in a massive Cantor set. In the same period, Hmidi and Roulley \cite{Hmidi-Roulley:2022} explored the emergence of quasi-periodic solutions for the quasi-geostrophic shallow water equation.  Similar results related to the boundary effects on the emergence of invariant tori have been discussed in \cite{Hassainia-Roulley:2022}. 

As to  the three dimensional cases, a lot of important results were obtained by using variational approach, see for instance \cite{Norbu72,Frank75}. We should also  point   the  3D quasi-geostrophic system which fits well with our discussion  since rotating simply and doubly connected volumes bifurcating  from generic revolution shapes have been proved very recently in  \cite{GHM:doubly,GHM:2022}.
 Other related topics concerning  compactly  steady solutions (not necessary patches) are discussed  by  Gravilov \cite{Gravilov} and Constantin, La and Vicol \cite{CLV}, where they obtained interesting examples of smooth compactly supported stationary solutions for the 3D Euler equations.  We refer also to \cite{DPMW:rings, Gallay-Sverak}  for  the leapfrogging vortex rings  and related subjects.

The main goal of the current  work is to explore some aspects  of the portrait phase around stationary solutions for Euler equations. From \cite{GSPSY-rigidity}, we know that any stationary vorticity with fixed sign must be radial. Hence, the general question is the following:
$$
\textnormal{{\it Do time periodic solutions still survive around  stationary radial  vortices with constant sign?}}
$$
The fact that the vorticity is not constant inside its support induces complex   spectral problem compared to the vortex patch problem. This subject turns out to be less  explored in the literature and only a few relevant results are known. The first result has been obtained recently by  Castro, C\'ordoba and G\'omez-Serrano \cite{CCGS-2019}, who managed to carefully mollify  a rotating vortex patch into a smooth rotating solution with non constant vorticity. The key point of their work is the use of the level sets of the vorticity in order to overcome some difficulties related to the degeneracy of  the spectral problem, which will be explained later. The second one refers to the work of Garc\'ia, Hmidi and Soler \cite{GHS:2020} on the perturbation of a quadratic profile supported in a { in the unit disc $\mathbb D$:} $\omega_0(x)=(A|x|^2+B){\bf 1}_{\D}(x)$. The equilibrium state is far away Rankine vortices and in this case we  highlight  new phenomena related to the scarcity and the abundance   of rigid time periodic  solutions with respect to the parameters of the quadratic profile. 

In this paper, we shall tackle the general problem and explore periodic solutions  in the vicinity of stationary solutions taking the form
\begin{equation}\label{intro-initial-cond}
\omega_0(x)=f_0(|x|){\bf 1}_{\D}(x),
\end{equation}
where $\mathbb D$ is the unit disc and $f_0$ is a monotonic smooth profile with a constant sign in the disc. Before stating our main result, we need to introduce some materials.
In the same spirit of  \cite{GHS:2020}, we will simultaneously perturb the density  $f_0$ by a non radial function and the boundary  using a conformal map $\Phi$ from the disc to a general bounded simply--connected domain $D$. More precisely, we shall look for rigid periodic solutions subject to the following ansatz,
\begin{align}\label{rotatingsol}
 \omega_0=(f\circ\Phi^{-1}) {\bf{1}}_D, \quad \omega(t,x)=\omega_0(e^{-it\Omega} x),  \quad  \forall x\in\R^2,
\end{align}
where $\Omega$ is the angular velocity, ${\bf{1}}_D$ is  the characteristic function of  a smooth simply connected domain $D$. Notice that in view of this ansatz  the solution rotates uniformly around the origin The real function $f:\overline{\D}\to \R$ denotes  the density profile  and $\Phi:\D\to D$  the conformal mapping  which are given by perturbing those of the stationary state \eqref{intro-initial-cond}, that is, 
$$
f=f_0+g,\quad \Phi=\mbox{Id}+\phi,
$$ 
where  $g$ and  $\phi$ are small enough in suitable function spaces.
By inserting this  ansatz  into Euler equations, we get the equivalent stationary equation on the vorticity  $\omega_0$ with velocity $v_0$ 
\begin{equation}\label{FirstEqB}
(v_0(x)-\Omega x^\perp)\cdot\nabla\omega_0(x)=0,\quad \forall x\in\R^2,
\end{equation}
{where $(x_1,x_2)^\perp=(-x_2,x_1)$. }Then, assuming that $f$ is not vanishing at the boundary of $D$ (a property  which is preserved by perturbation), we deduce the equivalent system 
\begin{align}
(v_0(x)-\Omega x^\perp)\cdot\nabla (f\circ\Phi^{-1})(x)&=0,\quad\hbox{in } D,\label{rotatingeq2d}\\
(v_0(x)-\Omega x^\perp)\cdot\vec{n}(x)&=0,\quad \hbox{on } \partial D,
\label{rotatingeq2b}
\end{align}
where $\vec{n}$ is the upward unit normal vector  to the boundary $\partial D$. We will refer to \eqref{rotatingeq2d} and \eqref{rotatingeq2b} as the density and boundary equations, respectively. Note that in the case where $f_0$ is a constant and $g=0$, the boundary equation agrees with the vortex patch problem.\\
Due to several constraints required along this paper from the spectral study to the stability of the function spaces, we need to impose to the initial profile  $f_0$  the following  conditions  dealing with its regularity and  monotonicity,
\begin{align}
&r\in[0,1]\mapsto f_0(r)=\tilde{f}_0(r^2)\mbox{ with }\tilde{f}_0\in\mathscr{C}^{2,\beta}([0,1]),\quad \beta\in(0,1).\hfill\label{H2}\tag{H1}\\
&\inf_{r\in(0,1]}\tfrac{{f}_0'(r)}{r}> 0.\hfill\label{H1}\tag{H2}
\end{align}
Notice that the last assumption implies in particular that $f_0$ is strictly increasing. However,  we can also deal with strictly decreasing profile by working with $-f_0$. 
We will distinguish two  cases for which we can achieve a complete answer on the emergence of periodic solutions depending on the sign of $f_0$ which is supposed to be constant. The two cases $f_0>0$ and $f_0<0$ offer two different scenarios is a similar way  to  the quadratic profile discussed in \cite{GHS:2020}. 
For the main statement we need to define  the {\it relative amplitude} of $f_0$ as
$$
A[f_0]:=\frac{f_0(1)}{f_0(0)},
$$
which is greater than $1$ according to the monotonicity assumption. The function spaces of H\"older type that will be used now  are discussed in  Section \ref{Funct-space}. Now, we are ready to formulate the main result of this paper. 
\begin{theo}\label{th-intro}
Let $0<\alpha<\beta<1$ and  $f_0$ satisfy \eqref{H2}--\eqref{H1}.
\begin{itemize}
\item[i)] {\it $($Scarcity  case$)$} Assume that $\,\inf_{r\in[0,1]}f_0(r)>0$, there exist $a>0$ and  $m_1\in\N$ with 
$$
m_1\geqslant \frac{1}{10(A[f_0]-1)},
$$
such that for any $m\in[3,m_1]\cap\N$, there exists a continuous curve  $\xi\in(-a,a)\mapsto (\Omega_\xi, f_\xi, \phi_\xi)\in\R\times \mathscr{C}^{1,\alpha}_{s,m}(\D)\times \mathscr{H}\mathscr{C}^{2,\alpha}_{m}(\D)$ such that the initial datum 
$$
\omega_0=(f\circ\Phi^{-1}){\bf 1}_{\Phi(\D)}, \quad f=f_0+f_\xi,\quad \Phi=\textnormal{Id}+\phi_\xi,
$$
generates  for $\xi\neq0$ a non radial m-fold solution for  Euler equations that rotates at constant angular velocity $\Omega_\xi$.
\item[ii)] {\it $($Abundance  case$)$} Assume that $\,\sup_{r\in[0,1]}f_0(r)<0$, there exist $m_2\in\N, a>0$, such that for any $m\geqslant m_2$ there exists  a continuous curve  $\xi\in(-a,a)\mapsto (\Omega_\xi, f_\xi, \phi_\xi)\in\R\times \mathscr{C}^{1,\alpha}_{s,m}(\D)\times \mathscr{H}\mathscr{C}^{2,\alpha}_{m}(\D)$ such that
$$
\omega_0=(f\circ\Phi^{-1}){\bf 1}_{\Phi(\D)}, \quad f=f_0+f_\xi,\quad \Phi=\textnormal{Id}+\phi_\xi,
$$
generates  for $\xi\neq0$ a non radial m-fold solution for  Euler equations that rotates at constant angular velocity $\Omega_\xi$.
\end{itemize}
\end{theo}
From this statement, we see two different regimes related to the sign of $f_0.$ When the initial profile  is nonnegative, then we are in the scarcity case and we find only a finite number of bifurcating curves. Their number depends on the relative amplitude $A[f_0]\geqslant1$: we get more and more when this number is close to $1$ meaning that  the profile is varying slowly.   
 On the other hand, for the nonpositive and increasing profiles, we get an infinite  family of nontrivial bifurcating curves associated with  large symmetry.
 
Next, we shall   briefly outline the main ideas  of the proof of Theorem \ref{th-intro}. First, we write the system  \eqref{rotatingeq2d}-\eqref{rotatingeq2b} using the change of coordinates, and then perform  the Implicit Function Theorem with  the boundary equation \eqref{rotatingeq2b}. This is achieved provided that $\Omega$ is excluded from the following   singular set 
\begin{equation*}
\mathcal{S}^m_{\textnormal{sing}}=\bigg\{\widehat{\Omega}_n:=\int_0^1sf_0(s)ds-\frac{n+1}{n}\int_0^1s^{2n+1}f_0(s)ds,\quad    \forall n\in m\N^\star\cup\{\infty\} \bigg\},
\end{equation*}
and we find  $\phi$  as an implicit  function depending on $(\Omega,g)$,
$$
\phi=\mathcal{N}(\Omega,g).
$$
Second, we should deal with the density equation which has the defect to be degenerating at the radial direction when we linearize at the equilibrium state. As a consequence, we loose  the Fredholm structure which is required in the bifurcation techniques. Then we proceed as in  \cite{GHS:2020} where we should modify the equation  \eqref{rotatingeq2d} by imposing some rigidity in the resolution of the nonlinear elliptic equation. This scheme will be   explained in Section \ref{Sec-Rigidity}. Hence, the outcome of this step is to be able to  reformulate the density equation into
\begin{equation}\label{intro-density-2}
\widehat{G}(\Omega,g):=G(\Omega,g,\mathcal{N}(\Omega,g))=0,
\end{equation}
where
$$
G(\Omega,g,\phi)(z):=\mathcal{M}(\Omega,f(z))+\frac{1}{2\pi}\int_{\D}\log|\Phi(z)-\Phi(y)|f(y)|\Phi'(y)|^2 dA(y)-\frac12\Omega|\Phi(z)|^2-\lambda,
$$
and $\lambda$ is a constant such that
\begin{align*}
\lambda=&\mathcal{M}(\Omega,f_0(r))-\int_r^1 \frac{1}{\tau}\int_0^\tau sf_0(s)dsd\tau-\frac12 \Omega r^2.
\end{align*}
The role of the function $\mathcal{M}$ is to guarantee that the radial case $g=0=\phi$ is still a trivial solution to this model for any $\Omega$ belonging to a suitable open interval (implying in turn that $\lambda$ is a constant). By this way, we find a natural scheme to generate $\mathcal{M}$, which is described below,

$$
\mathcal{M}(\Omega,t)=\int_{a}^t \frac{ds}{\mu(\Omega,s)},\quad \mu(\Omega,t)=\frac{2(\tilde{f}_0'\circ \tilde{f}_0^{-1})({t})}{\Omega-\frac12\int_0^1 \tilde{f}_0\big(s\tilde{f}_0^{-1}({t})\big)ds}.
$$
Remark that
applying $\mu$ to the trivial solution, we arrive to 
$$
\mu(\Omega,f_0(r))=\frac{f_0'(r)}{r\left(\Omega-\int_0^1 sf_0(sr)ds\right)}:=\mu_\Omega^0(r).
$$
We point out that $\mu(\Omega,\cdot)$ is only defined in the interval $[a,b]:=\textnormal{Range }(\tilde{f}_0)$ and an extension procedure outside this segment is required. 
More details  can be found in Section \ref{Sec-Rigidity}. To implement Crandall-Rabinowitz theorem in bifurcation theory \cite{rabinowitz:simple,kielhofer} and generate a nontrivial solution to the equation \eqref{intro-density-2}, one needs to check the required spectral properties for the linearized operator at the equilibrium.  From Section \ref{Fourier-expans1}, we obtain for any test function
$$r e^{i\theta}\in\D\mapsto h(re^{i\theta})=\sum_{n\in\N}h_n(r)\cos(n\theta)
,$$
the following structure for the linearized operator, see \eqref{lin-op},
\begin{align*}
D_g \widehat{G}(\Omega,0)[h](re^{i\theta})=&\mathbb{L}_0^\Omega [h_0](r)+\sum_{n\geqslant 1} \cos(n\theta)\mathbb{L}_n^\Omega [h_n](r).
\end{align*}
where the operators $\mathbb{L}_n^\Omega$ are of integral type. As we shall see in Section \ref{Sec-KerneLn}, the kernel equation reduces to solving the collection of the  one dimensional linear problems. For $n=0,$
$$
\big(\hbox{Id}-\sigma_\Omega \mathcal{L}^\Omega_0\big)[h_0](r)=0\quad\hbox{with}\quad K_0(r,s)= \nu_\Omega(r)\nu_\Omega(s)\left[\ln (\tfrac{1}{r}){\bf 1}_{[0,r]}(s)+\ln (\tfrac{1}{s}){\bf 1}_{[r,1]}(s)\right]
$$
and for $n\geqslant 1,$
\begin{equation}\label{KK-EE}
\big(\hbox{Id}-\sigma_\Omega \mathcal{L}^\Omega_n\big)[h_n](r)=-\sigma_{\Omega}\frac{\nu_\Omega(r)}{\nu_\Omega(1)}\frac{r G_n(r)}{G_n(1)} \mathcal{L}^\Omega_n(1), n\geqslant 1
\end{equation}
with
\begin{align*}
\mathcal{L}^\Omega_n[h](r)=\int_0^1 {K}_n(r,s)h(s) d\lambda_\Omega(s),\quad d\lambda_\Omega(s):=\frac{s}{\nu_\Omega(s)}ds
\end{align*}
where  the involved  kernel is symmetric  and takes  the form
\begin{align*}
{K}_n(r,s)=\frac{\nu_\Omega(r)\nu_\Omega(s)}{2n}\left[\left(\frac{r}{s}\right)^n {\bf 1}_{\{r\leqslant s\leqslant 1\}}+\left(\frac{s}{r}\right)^n {\bf 1}_{\{0\leqslant s\leqslant r\}}\right].
\end{align*}
The function  $G_n$ is defined in \eqref{Gn-def} and   $\nu_\Omega$ is positive with
\begin{align*}
\nu_\Omega(r):=\sigma _\Omega\,\mu_\Omega^0(r)=|\mu_\Omega^0(r)|,\quad\hbox{with}\quad 
 \sigma_\Omega:=\begin{cases}
-1,\quad \Omega\in(-\infty,\kappa_1),\\
1,\quad \Omega\in(\kappa_2,+\infty).
\end{cases}
\end{align*}
In this formulation we have assumed that $\Omega\notin[\kappa_1,\kappa_2]$ in order to guarantee that $\mu_\Omega^0$ is not vanishing inside $[0,1]$, where
$$
\kappa_1=\inf_{r\in[0,1]} \int_0^1 sf_0(sr)ds, \quad \kappa_2=\sup_{r\in[0,1]} \int_0^1sf_0(sr)ds.
$$
This gives that $\mu_\Omega^0$ admits a constant sign depending on the location of $\Omega.$ It is negative for $\Omega<\kappa_1$ and this regime is called {\it defocusing}. However, when  $\Omega>\kappa_2$, the function $\mu_\Omega^0$ is positive and this regime is called {\it focusing }. This terminology is borrowed from Schr\"odinger equations and it  is justified by the fact that the resolution of the multiple integral equations is related to the following singular Sturm Liouville problems, see \eqref{diff-Fn},
\begin{align}\label{Fn-DEEE}
F_{n,\Omega}''+\frac{2n+1}{r}F'_{n,\Omega}+\mu_\Omega^0 F_{n,\Omega}=0,\quad  F_{n,\Omega}(0)=1,\quad n\geqslant1.
\end{align}
Then the potential $\mu_\Omega^0$ is repelling ({\it defocusing regime}) for $\Omega<\kappa_1$  and attracting ({\it focusing regime}) for $\Omega>\kappa_2.$
At this stage, we are led to solve these equations with respect to the parameters $n$ and $\Omega$. For the kernel equation at the level $n=0$, we show in Proposition \ref{radialfunctions}, that zero is the only solution provided that $\Omega\in (-\infty,\kappa_1)\cup((\kappa_2,\infty)\backslash\mathbb{S})$, where the set $\mathbb{S}$ is finite. The proof in the {\it defocusing} case ($\sigma_\Omega=-1$) is related to the following nice property which states that the operator   $\mathcal{L}^\Omega_0:L^2(d\lambda_\Omega)\to L^2(d\lambda_\Omega)$ is self-adjoint positive definite compact operator, implying in particular  that all the spectrum is contained in the positive region $(0,\infty).$ Consequently, the operator $\hbox{Id}+ \mathcal{L}^\Omega_0$ is invertible and therefore its kernel is trivial.  In  the {\it focusing} case ($\sigma_\Omega=1$),  the proof turns out to be more tricky and subtle. First we transform the kernel equation into a second order differential equation of Sturm-Liouville type
$$
(rF^\prime)^\prime+r \nu_\Omega(r)F=0,\quad\hbox{with}\quad F(1)=0.
$$
Due to the singular structure around zero, we expect that smooth solutions in $[0,1]$ form a one-dimensional vector space parametrized by $\Omega$ and generated by a nontrivial element  denoted by $F_\Omega$ and normalized as $F_\Omega(0)=1$. Then we have to find the set $\mathbb{S}$  covering all the  $\Omega>\kappa_2$ such that $F_\Omega$ matches with the boundary condition $F_{\Omega}(1)=0$. Here, due to the singular structure, we are not able to use the shooting method and we find another elegant way to tackle this control problem using Pr\"{u}fer transformation together with Kneser's  Theorem \cite{Kneser} on the non-oscillation phenomenon.\\
As to the kernel equation \eqref{KK-EE} for $n\geqslant1$, the  resolution is more complex due to the variable coefficients in $r$, which contrasts with the vortex patch problem where the operator is simply given by a Fourier multiplier in the angular variable and constant in $r$.  The first step is to characterize analytically the set of $\Omega$ associated with a nontrivial kernel, leading to what we call the {\it dispersion equation.} This will be performed  along  Section \ref{Sec-Disper-eq} using the generator $F_{n,\Omega}$ defined by \eqref{Fn-DEEE}. We find there, see Lemma \ref{Lem-disper},  that the dispersion equation takes the form
\begin{equation}\label{intro-disp}
\zeta_n(\Omega)=0,
\end{equation}
where
$$
\zeta_n(\Omega)=F_{n,\Omega}(1)\left(\Omega-\frac{n}{n+1}\int_0^1 sf_0(s)ds\right)+\int_0^1F_{n,\Omega}(s) s^{2n+1}\big(f_0(s)-2\Omega\big)ds,
$$
The structure of $\zeta_n$ is connected to the profile $f_0$ and exploring the zeroes in $\Omega$  for  this highly nonlinear function is not trivial. Notice that the dependence of the generator $F_{n,\Omega}$ in $\Omega$ is not explicit. This situation is quite  different from the quadratic case analyzed in  \cite{GHS:2020}, where $F_{n,\Omega}$ was  computed explicitly through  {\it Gauss hypergeometric function}. In this latter case,  the analysis of the zeroes for $\zeta$ takes into account of several specific  algebraic and analytic properties  of $F_{n.\Omega}.$ In the current situation, one should find the alternative to  the explicit form by developing qualitative properties around $F_{n,\Omega}$, see Section \ref{Sec-STE}. Then the resolution of \eqref{intro-disp} depends on the region where $\Omega$ is taken and the sign of $f_0$. Actually, when $f_0 >0$ ({\it{scarcity case}}) we find that the dispersion equation admits at most a finite number of solution in $(-\infty,\kappa_1)$ depending on the relative amplitude of $f_0$. However, when $f_0<0$ ({\it{abundance case}}), we find a countable family of solutions located in $(\kappa_2,\infty)$. As a matter of fact, these solutions live in the vicinity of the singular set $\mathcal{S}^m_{\textnormal{sing}}$ introduced before. The formal intuition about the  transition regimes between the scarcity and the abundance can be made more clear using the following image. When $f_0>0$, the elements  of the  singular set form an increasing sequence converging to $\kappa_2$. So, the set $\{n\geqslant 1, \widehat{\Omega}_n\notin[\kappa_1,\kappa_2]\}$ is finite and their corresponding values $\widehat{\Omega}_n$ are necessary located below $\kappa_1.$ Around most of these points, we are able to find solutions to the dispersion equation. As to the case $f_0<0$, the singular sequence becomes decreasing and lives in the region $(\kappa_2,\infty)$ and we show that for large symmetry each element  $\widehat{\Omega}_n$ is paired with at least one solution ${\Omega}_n$ to \eqref{intro-disp}. 

Another difficulty stems from  the  transversality condition in Crandall-Rabinowitz theorem, which is not at all  obvious in our setting. It will be discussed in Section \ref{Sec-transv}. The  goal is  to encode it in an analytical form using the kernel structure, see Proposition \ref{prop-element-kernel} together with  the  range characterization as the kernel of a linear form, see Proposition \ref{Prop-Range22}. Then by making refined estimates we are able to check the transversality in the two different configurations $f_0>0,$ and  $f_0<0$.

The paper is organized as follows. In Section \ref{Sec-formulation}, we review the formulation for a rotating solution to the 2D Euler equations together with the derivation of the density equation \eqref{intro-density-2}, and we also provide there the function spaces used through all the work. Later in Section \ref{Sec-lin-op}, we compute the linear operator of $\widehat{G}$ around the trivial solution $(\Omega,0)$ using Fourier expansion. After that, we explore its Fredholm structure. In Section \ref{Sec-kernel}, we shall  focus on the kernel study. First, we characterize the kernel with the dispersion relation and  later solve it for the different cases $f_0>0$ and $f_0<0$. At the end of the section, we provide the kernel generators and discuss their regularity. Section \ref{Sec-transv} is devoted to the proof of  the transversal condition needed to apply Crandall-Rabinowitz theorem. Finally, in Section \ref{Sec-main-result} we collect all the previous analysis in order  to prove Theorem \ref{th-intro}.

\section{Rigid time periodic solutions}\label{Sec-formulation}
In this section, we intend to  carefully reformulate the equations governing  rotating solutions to the 2D Euler equations. To be precise about this terminology, given an initial data $\omega_0$ we say that  $\omega$   is a rigid time periodic solution to  the system \eqref{Eulereq} or equivalently (up to a spatial translation)  a rotating solution
 with constant angular velocity $\Omega\in\R$ if 
$$
\omega(t,x)=\omega_0(e^{-i\Omega t}x).
$$
Then, inserting  this  ansatz into \eqref{Eulereq} yields to the stationary equation on the profile $\omega_0$
\begin{equation}\label{rotating-eq}
(v_0(x)-\Omega x^\perp)\cdot \nabla\omega_0(x)=0, \quad x\in\R^2,
\end{equation}
where $(x_1,x_2)^\perp=(-x_2,x_1)$. The goal now is to derive from this general form the equations associated to localized rotating solutions. In which case, $\omega_0(x)=\omega_0(x){\bf{1}}_D$ for some smooth bounded domain $D$. Notice that we have identified the vorticity $\omega_0$ defined on $\R^2$ with its density still denoted by $\omega_0$ and only defined in the closed set $\overline{D}$. We assume that the density function is smooth in  $\overline{D}$.  Then from straightforward arguments, the equation \eqref{rotating-eq} can be written in the weak sense in the form of two coupled equations
\begin{align}
(v_0(x)-\Omega x^\perp)\cdot \nabla\omega_0(x)=0,\quad x\in D,\label{rotating-system-1}\\
\omega_0(x)\big[(v_0(x)-\Omega\, x^\perp)\cdot n(x)\big]=0,\quad x\in \partial D,\label{rotating-system-2}
\end{align}
where $n$ is a normal vector to $\partial D$. A general class of solutions to the system \eqref{rotating-system-1}-\eqref{rotating-system-2} is given by radial functions corresponding to $D=\D$ the unit disc  and the density $\omega_0$ a radial smooth function.  Actually, any function taking the form
\begin{equation}\label{trivial-sol}
\omega_0^{\textnormal{trivial}}(x)=f_0(|x|){\bf 1}_{\D},
\end{equation}
with $f_0$ being a smooth function, is a solution to \eqref{rotating-system-1}-\eqref{rotating-system-2} for any arbitrary value of $\Omega\in\R.$ In \eqref{trivial-sol}, the density $f_0$ is a generic radial function that we aim to perturb. Proceeding as in  \cite{GHS:2020}, we plan to construct nontrivial solutions to the above system by perturbing in a suitable way the trivial solution \eqref{trivial-sol}.  For this aim,  we make use of a conformal map $\Phi:\D\rightarrow D$ and look for rotating solution in the form
\begin{equation}\label{initial-data-ansatz}
\omega_0=(f\circ \Phi^{-1}){\bf 1}_D,
\end{equation}
with $f=f_0+g$ and $g:\D\to\R$ small enough {in a suitable norm.} To avoid the degeneracy of the boundary equation \eqref{rotating-system-2} we have to assume that  $f_0(1)\neq 0$ that will remain true for $f$. As the domain $D$ is a small smooth perturbation of $\mathbb{D}$ then this can be encoded through the conformal mapping by simply imposing  that is
$$\Phi=\textnormal{Id}+\phi,$$
with $\phi$ being small enough in a strong topolgy.
In what follows, we will refer to \eqref{rotating-system-1} or \eqref{rotating-system-2} as the {\it density equation} or the {\it boundary equation}, respectively.

\subsection{Boundary equation}
Following  \cite[Section 2]{GHS:2020}, inserting \eqref{initial-data-ansatz} into \eqref{rotating-system-2} and  using the Biot-Savart law, change of variables and the complex notations, one gets the boundary equation in terms of the new unknowns $(g,\phi)$,
\begin{align}\label{Bound-eq1}
F(\Omega,g,\phi)(w)=0,\quad w\in\T,
\end{align}
where the nonlinear functional $F$ is defined by 
$$
F(\Omega,g,\phi)(w):=\textnormal{Im}\left[\left(\Omega\overline{\Phi(w)}-\frac{1}{2\pi}\int_{\D} \frac{f(y)}{\Phi(w)-\Phi(y)}|\Phi'(y)|^2 dA(y)\right)\Phi'(w)w\right].
$$
According to  \cite[Proposition B.5]{GHS:2020}, one may easily check that $F(\Omega,0,0)=0$ for any $\Omega\in\R$, which is compatible  with the stationary solution \eqref{trivial-sol}.
We remark that by taking $g\equiv 0$ and $f_0\equiv 1$, the foregoing boundary equation reduces to the vortex patch problem  introduced in \cite{burbea:motions} and well-explored later  in various directions  \cite{ Berti-Hassainia-Masmoudi:2022, CCGS-2016,  CCGS-2016-2,CCGS-2019, CCGS-2020, Hoz-Hassainia-Hmidi:doubly-gSQG, Hoz-Hassainia-Hmidi-Mateu:disc,  Hoz-Hmidi-Mateu-Verdera:doubly-euler, Dristchel-Hmidi-Renault, garcia:karman,garcia:choreography, GH:2022, GHM:doubly, GHM:2022, GHS:2020,GSPSY-sheets, Hassainia-Hmidi:vstates-gSQG, HH:asymmetric-pairs, Hassainia-Hmidi-Masmoudi:2021, hmw:global, Hassainia-Roulley:2022,  Hassainia-Wheeler-points, Hmidi-Mateu:kirchhoff, Hmidi-Mateu:degenerate, hm:pairs,Hmidi-Mateu-Verdera:doubly, Hmidi-Roulley:2022, Roulley:2022}, and references therein.

\subsection{Transformation of the density equation}\label{Sec-Rigidity}
As we discussed in the introduction of \cite{GHS:2020}, some delicate problems are connected to the density  equation in its form \eqref{rotating-system-2} making hard to implement bifurcation tools. The difficulties are related first to the fact that any radial perturbation is still a solution and this will generate a big kernel for the linearized operator. Second, the linearized operator at the equilibrium state  is not of Fredholm type.   
To be more precise, notice that the density equation \eqref{rotating-system-1} can be written as
$$
H(\Omega,\omega_0)(re^{i\theta}):=\left(\frac{v^0_\theta}{r}-\Omega\right)\partial_\theta \omega_0+v_r^0\partial_r \omega_0=0,
$$
where we use the decomposition $v_0(re^{i\theta})=v_\theta^0(r,\theta) x^\perp+v_r^0(r,\theta) x$ and $x=re^{i\theta}$. By fixing  $D=\D$, the problem reduces to find some non radial density $g$ such that
$$
H(\Omega,f_0+g)\equiv 0.
$$
However, one gets
$$
H(\Omega,f_0+g_{\textnormal{radial}})\equiv 0,
$$
for any radial function $g_{\textnormal{radial}}$. That implies that any radial function belongs to the kernel of $D_g H(\Omega,0)$ which amounts to be with an infinite dimensional kernel. Moreover,  the linearized operator around the trivial solution takes the form
$$
D_g H(\Omega,0)h(re^{i\theta})=\left(\frac{v^0_\theta}{r}-\Omega\right)\partial_\theta h_0+ K[h](r,\theta)f_0'(r),
$$
where $K$ is a compact integral  operator. It follows that the leading transport part of this operator  depends only on the angular derivative $\partial_\theta,$ which turns out to be in contrast with  the nonlinear operator $H$ which is built upon  both derivatives. This implies that the linearized operator is not of Fredholm type with the function spaces used   to stabilize the nonlinear functional. 
These problems were analyzed in  \cite{CCGS-2019, GHS:2020}, where the authors avoid to use the formulation \eqref{rotating-system-2} by making some rigidities using different ways.  In the first work \cite{CCGS-2019}, the authors focuses on the desingularization of a patch through smooth profiles using level sets reformulation.   However, in \cite{GHS:2020} the authors proceeds in a different way and reformulate the density equation by taking into account that the density is constant along the level sets of the relative stream function. Notice that  in \cite{GHS:2020} the authors work with a radial quadratic profile $f_0$  because of  the spectral analysis of the linear operator which is much more tractable in this special case, compared to the general one.\\
In this work and in order  to filter radial solutions from the equation \eqref{rotating-system-2}  we follow the same strategy of \cite[Section 4]{GHS:2020}.

Actually, by taking any scalar function $\mu$ and imposing the structure 
\begin{equation}\label{density-eq-1}
\nabla (f\circ\Phi^{-1})(x)=\mu(\Omega, (f\circ\Phi^{-1})(x))(v_0(x)-\Omega x^\perp)^\perp,
\end{equation}
we find from \eqref{initial-data-ansatz} that the equation  \eqref{rotating-system-1} is automatically satisfied. We emphasize that this is only a sufficient condition for  \eqref{rotating-system-1}.  Since $\mu$ is arbitrary, then one needs to fix it  in order to get with $\Phi=\hbox{Id}$ that the radial profile $f_0$  is a solution to \eqref{density-eq-1} for any $\Omega$ belonging in some nontrivial subsets of $\R$. According to \cite[Section 4.1]{GHS:2020}, this assumption turns out from an explicit computation to be equivalent to the constraint
\begin{equation}\label{mutrivial}
\mu(\Omega,f_0(r))=\frac{f_0'(r)}{r\left(\Omega-\int_0^1 sf_0(sr)ds\right)}, \quad r\in(0,1].
\end{equation}
{It follows from  \eqref{H2}  that
\begin{equation}\label{mu-r2}
\mu(\Omega,f_0(r))=\mu(\Omega,\tilde{f}_0(r^2))=\frac{2\tilde{f}_0'(r^2)}{\Omega-\frac12\int_0^1 \tilde{f}_0(sr^2)ds},
\end{equation}
or equivalently
$$
{\mu}(\Omega,\tilde{f}_0(r))=\frac{2\tilde{f}_0'(r)}{\Omega-\frac12\int_0^1 \tilde{f}_0(sr)ds},\quad r\in(0,1].
$$
This  makes sense for strictly monotonic profiles $\tilde{f}_0$, as assumed in \eqref{H1},  and when $\Omega\notin [\kappa_1,\kappa_2]$ with
\begin{align}
\kappa_1:=\inf_{r\in[0,1]}\int_0^1 s {f_0}(rs)ds=\frac12 f_0(0),\label{kappa1}\\
\kappa_2:=\sup_{r\in[0,1]}\int_0^1 s{f_0}(rs)ds=\int_0^1 s{f_0}(s)ds.\label{kappa2}
\end{align}
We observe that by  $\Omega\notin [\kappa_1,\kappa_2]$ together with the monotonicity assumption \eqref{H1}, the function $\mu(\Omega,{f}_0)$ is well defined and   keeps a constant sign. Indeed,  it is positive for  $\Omega>\kappa_2$ and negative for  $\Omega<\kappa_1.$  
Therefore,  we should get 
\begin{equation}\label{Restr-mu}
{\mu}(\Omega,t)=\frac{2(\tilde{f}_0'\circ \tilde{f}_0^{-1})(t)}{\left(\Omega-\frac12\int_0^1 \tilde{f}_0(s\tilde{f}_0^{-1}(t))ds\right)},
\end{equation}
for any $t\in\mbox{Range}(\tilde{f}_0)$. From \eqref{H2}-\eqref{H1}  we infer that the range of $\tilde{f}_0$ is a segment $[a,b]$. Moreover, from the composition laws we get that ${\mu}(\Omega,\cdot)$ is of class $\mathscr{C}^{1,\alpha}([a,b])$. Now, we construct   ${\mu}(\Omega,\cdot)$ as any  extension of the right hand side of \eqref{Restr-mu} in a small neighborhood $[a-\epsilon,b+\epsilon]$ of $[a,b]$ and of class { $\mathscr{C}^{1,\alpha}([a-\epsilon,b+\epsilon])$}. By this way, we can   define ${\mu}(\Omega,f(z))$ for any $z\in\mathbb{\D}$ and with  
 $f$ a {\it small} perturbation of $f_0$.} Now, we  define
\begin{align}\label{def-MM}
\mathcal{M}(\Omega,t):=\int_{a}^t \frac{ds}{{\mu}(\Omega,s)},\quad \forall t\in[a-\epsilon,b+\epsilon]
\end{align}
 and integrate \eqref{density-eq-1} obtaining
\begin{equation}\label{density-2}
G(\Omega,g,\phi)(z):=\mathcal{M}(\Omega,f(z))+\frac{1}{2\pi}\int_{\D}\log|\Phi(z)-\Phi(y)|f(y)|\Phi'(y)|^2 dA(y)-\frac12\Omega|\Phi(z)|^2-\lambda=0,
\end{equation}
for any $z\in\D$, and where
$$
\lambda:=\mathcal{M}(\Omega,f_0(r))-\int_r^1 \frac{1}{\tau}\int_0^\tau sf_0(s)dsd\tau-\frac12 \Omega r^2.
$$
Let us point out that $\lambda$ does not depend on $r$ by using the relation between $\mu$ and $f_0$, and also using the estimates given in \cite[Proposition B.5]{GHS:2020}. From the choice of $\lambda$ we get $G(\Omega,0,0)\equiv 0$. For more details about the expression of $G$ we refer to \cite[Section 4.1]{GHS:2020}. 
We shall end this section with the following result dealing with the regularity of the function $\mathcal{M}(\Omega,\cdot)$ that will be used later.
\begin{lem}\label{lem-regM}
For any small $\epsilon>0$, there exists $\delta>0$ small enough such that if $\Omega\notin [\kappa_1-\delta,\kappa_2+\delta]$, the function $t\in [a-\epsilon,b+\epsilon]\mapsto \mathcal{M}(\Omega,t)$ is of class on $\mathscr{C}^{2,\alpha}$ and
$$
\forall t\in[a,b],\quad  \partial_t\mathcal{M}(\Omega,t)=\frac{\Omega-\frac12\int_0^1 \tilde{f}_0(s\tilde{f}_0^{-1}(t))ds}{2(\tilde{f}_0'\circ \tilde{f}_0^{-1})(t)}\cdot
$$
\end{lem}
\begin{proof}
Coming back to \eqref{def-MM} we see that $\mathcal{M}(\Omega,\cdot)$ is differentiable with
\begin{align*}
\partial_t\mathcal{M}(\Omega,t)=\frac{1}{{\mu}(\Omega,t)}, \forall t\in[a-\epsilon,b+\epsilon].
\end{align*}
To get the formula stated in the lemma, it is enough to use \eqref{Restr-mu}. As to the regularity,  
we have seen from the preceding discussion that $t\in[a-\epsilon,b+\epsilon]\mapsto {\mu}(\Omega,t)$ is of class $\mathscr{C}^{1,\alpha}$. Since this function is not vanishing then by the classical composition law we infer that  $t\in[a-\epsilon,b+\epsilon]\mapsto \frac{1}{{\mu}(\Omega,t)}$ is of class $\mathscr{C}^{1,\alpha}$. This achieves that $\mathcal{M}(\Omega,\cdot)$ is of class $\mathscr{C}^{2,\alpha}([a-\epsilon,b+\epsilon]).$
\end{proof}
\subsection{Function spaces}\label{Funct-space}
Here, we aim to fix the function spaces that will be used to perform the bifurcation arguments. First, for   $\alpha\in(0,1)$ we denote by $\mathscr{C}^{0,\alpha}({\D})$   the set of continuous functions such that
\begin{align*}
\|f\|_{\mathscr{C}^{0,\alpha}({\D})}:= \|f\|_{L^\infty({\D})}+\sup_{z_1\neq z_2\in{\D}}\frac{|f(z_1)-f(z_2)|}{|z_1-z_2|^\alpha}<\infty.
\end{align*}
 Second, for $k\in\N$ the space   $\mathscr{C}^{k,\alpha}({\D})$ stands for the set of   $\mathscr{C}^k$ functions whose partial $k$-th derivatives lie \mbox{in $\mathscr{C}^{0,\alpha}(\D)$.} Let us point out that $\mathscr{C}^{k,\alpha}(\overline{\D})$ coincides with   $\mathscr{C}^{k,\alpha}(\D)$, for $\alpha\in(0,1)$, in the sense that any element of this latter space admits a unique continuous extension up to the boundary which lies to $\mathscr{C}^{k,\alpha}(\overline{\D})$. Similarly, we define the H\"older spaces $\mathscr{C}^{k,\alpha}(\T)$ in the unit circle $\T$. Let us supplement these spaces with the additional symmetry structures. Consider $m\in\N$ and define 
\begin{equation}\label{space-1-m}
\mathscr{C}^{k,\alpha}_{s,m}(\D):=\bigg\{ g: {\D}\to \R\in \mathscr{C}^{k,\alpha}({\D}),\quad g(re^{i\theta})=\sum_{{n\in m\N}} g_n(r)\cos(n\theta),\, g_n\in\R,\,  \forall \,r\in[0,1], \theta\in\R \bigg\}
\end{equation}
and 
\begin{equation}\label{space-3-m}
 \mathscr{C}^{k,\alpha}_{a,m}(\T):=\bigg\{ \rho: \T\to \R\in \mathscr{C}^{k,\alpha}(\T),\quad \rho(e^{i\theta})=\sum_{n\in m\N^\star}\rho_n\sin(n\theta), \, \,\rho_n\in\R, \,  \forall \,\theta\in \R\bigg\}.
\end{equation}
These spaces  are equipped with the usual  norm $\|\cdot\|_{\mathscr{C}^{k,\alpha}}$. One can easily check that if the functions  $g\in \mathscr{C}^{k,\alpha}_{s,m}(\D)$ and $\rho\in \mathscr{C}^{k,\alpha}_{a,m}(\T)$, then they satisfy the following properties
\begin{equation}\label{Persis}
g(\overline{z})=g(z), \quad \rho(\overline{w})=-\rho(w), \quad \forall z\in {\D},\,\forall\,  w\in \T.
\end{equation}
Moreover, the parameter $m$ means the $m$-fold symmetry of the solution. We point out that the space $\mathscr{C}_{s,m}^{k,\alpha}(\D)$ will contain the perturbations of the initial radial density. The condition on $g$ means that this perturbation  is invariant by reflexion on the real axis. \\
The second kind of function spaces is $\mathscr{H}\mathscr{C}^{k,\alpha}_m(\D)$, which is the set  of holomorphic functions $\phi$ in $\D$  belonging to $ \mathscr{C}^{k,\alpha}({\D})$ and satisfying  
\begin{equation*}
  \phi(0)=0,\quad \phi'(0)=0\quad \hbox{and}\quad \overline{\phi(z)}=\phi(\overline{z}), \quad \forall z\in {\D}.
\end{equation*}
With these properties, the function $\phi$ admits the following Taylor expansion
\begin{align}\label{phi}
\phi(z)= z\sum_{{n\in m\N^\star}} a_n z^{n}, \quad a_n\in\R.
\end{align}
Next, we define
$$
\mathscr{H}\mathscr{C}^{k,\alpha}_m(\T):=\Big\{\phi\in \mathscr{C}^{k,\alpha}(\T),\quad\phi(w)=w\sum_{n\in m\N^\star}a_nw^{n},\, a_n\in\R,\, \forall w\in\T\Big\}.
$$
It is important to remark that the perturbed  domains that we shall construct for rotating solutions are parametrized by   the conformal map $\Phi=\textnormal{Id}+\phi$ taking  the form
$$
\Phi(z)=z+z\sum_{n\in m\N^\star}a_n z^{n}, \quad a_n\in\R,
$$
Thus by  imposing this structure we can remove the dilation, rotation  and translation invariances in order to reduce later the kernel size of the linearized operator. Finally, let us define the balls
\begin{align}\label{BAL}
 \left\{ \begin{array}{lll}B_{ \mathscr{C}_{s,m}^{k,\alpha}}(g_0,\varepsilon)=&\Big\{g\in   \mathscr{C}_{s,m}^{k,\alpha}({\D})\quad \textnormal{s.t.}\quad \|g-g_0\|_{k,\alpha}< \E  \Big\},\\
 B_{\mathscr{H} \mathscr{C}^{k,\alpha}_m}(\phi_0,\E)=&\Big\{\phi\in  \mathscr{H} \mathscr{C}^{k,\alpha}_m({\D})\quad \textnormal{s.t.}\quad \|\phi-\phi_0\|_{k,\alpha}< \E \Big\},
\end{array} \right.
\end{align}
for $\E>0$, {$k,m\in\N$, $\alpha\in(0,1)$}, $g_0\in  \mathscr{C}_{s,m}^{k,\alpha}({\D})$ and  $\phi_0\in  \mathscr{C}^{k,\alpha}({\D})$.
It is a classical fact that  if $\phi\in B_{\mathscr{H}\mathscr{C}^{k,\alpha}_m}(0,\varepsilon)$ with $\varepsilon\in(0,1)$ then $\Phi=\textnormal{Id}+\phi$ is a bi-Lipschitz function.

Moreover, from now on and by virtue of \eqref{H2}, we will assume through all the work that $\alpha$ is fixed for the function spaces and $\beta>\alpha$.

\subsection{Resolution of the boundary equation and consequence}
We have already seen in the foregoing sections that the equations \eqref{rotating-system-1}-\eqref{rotating-system-2} can be transformed into the form \eqref{Bound-eq1}-\eqref{density-2}. Thus to solve the new coupled system, we start first with solving  the boundary equation \eqref{Bound-eq1} by means of the Implicit Function Theorem, getting that the boundary perturbation  $\phi$ is a smooth functional on $(\Omega,g)$. Later, we will come back to the density equation \eqref{density-2} and insert such dependence of $\phi$  leading to a new equation  governing  only  the variables $(\Omega,g)$.\\
Now, we shall focus on the resolution of the boundary equation \eqref{Bound-eq1} where $\Omega$ and $g$ are viewed as parameters.
Following the same proof of \cite[Proposition 3.1]{GHS:2020}, we deduce that for a given profile  $f_0$ satisfying \eqref{H2},  for any $m\in\N^\star$ and  $\varepsilon\in(0,1)$, we  find that the functional 
$$
F:\R\times B_{\mathscr{C}_{s,{m}}^{1,\alpha}}(0,\varepsilon)\times B_{\mathscr{H}\mathscr{C}^{2,\alpha}_{{m}}}(0,\varepsilon)\rightarrow \mathscr{C}_{a,{m}}^{1,\alpha}(\T),
$$
is well-defined and of class $\mathscr{C}^1$. Moreover,  according to \cite[Proposition 3.3]{GHS:2020}, the partial Gateaux derivative $D_\phi F(\Omega,0,0)$ takes the form
\begin{align*}
D_\phi F(\Omega,0,0)[k](w)=\sum_{n\in {m\N^\star}}n\,a_n \left\{\Omega-\int_0^1 sf_0(s)ds+\frac{n+1}{n}\int_0^1 s^{2n+1}f_0(s)ds\right\}\sin(n\theta),
\end{align*}
with $\displaystyle{k(z)= \sum_{n\in m\N^\star}a_n z^{n+1}},$ and $a_n\in\R.$ Introduce the following   singular set 
\begin{equation}\label{Interv1}
\mathcal{S}^m_{\textnormal{sing}}:=\bigg\{\widehat{\Omega}_{n}:=\int_0^1sf_0(s)ds-\frac{n+1}{n}\int_0^1s^{2n+1}f_0(s)ds,\quad    \forall {n\in m\N^\star\cup\{\infty\}} \bigg\},
\end{equation}
which  corresponds to the location of  the points $\Omega$ where the partial  linearized operator $D_\phi F(\Omega,0,0)$ is not one-to-one.  Therefore, we obtain by virtue of  \cite[Proposition 3.3]{GHS:2020} that when $\Omega\notin \mathcal{S}^m_{\textnormal{sing}}$ one can parametrize locally  the zeros of  the nonlinear boundary equation. More precisely, we have the following result.

\begin{pro}\label{propImpl}
Let  $m\geqslant1,$ $f_0$ satisfy \eqref{H2} and let $I$ be an open interval such that  $\overline{I}\subset \R\backslash \mathcal{S}^m_{\textnormal{sing}}$.
Then, there exists $\varepsilon>0$ and  a $\mathscr{C}^1$ function  
$$\mathcal{N}: I\times B_{\mathscr{C}_{ s,{m}}^{1,\alpha}}(0,\varepsilon)\longrightarrow B_{\mathscr{H}\mathscr{C}^{2,\alpha}_{{m}}}(0,\varepsilon),
$$
 with the following property,
$$
 F(\Omega,g,\phi)=0\Longleftrightarrow \phi=\mathcal{N}(\Omega,g),
$$
 for any $ (\Omega,g,\phi)\in I\times B_{\mathscr{C}_{s,{m}}^{1,\alpha}}(0,\varepsilon)\times B_{\mathscr{H}\mathscr{C}^{2,\alpha}_{{m}}}(0,\varepsilon)$.
In addition,  we obtain the identity 
$$
D_g\mathcal{N}(\Omega, 0)[h](z)=z\sum_{n\in {m\N^\star}} A_n[h_n] z^{n},
$$
for any $h\in \mathscr{C}_{s,m}^{1,\alpha}({\D})$, with  $h(re^{i\theta})=\displaystyle{\sum_{n\in{m\N}}h_n(r)\cos(n\theta)}$ and 
\begin{equation}\label{An}
A_n[h_n]:={\frac{\displaystyle{\int_0^1s^{n+1}h_n(s)ds}}{2n\big(\widehat{\Omega}_n-\Omega\big)}},
\end{equation}
where $\widehat{\Omega}_n$ is defined in \eqref{Interv1}.  Moreover, we have
\begin{equation}\label{ContinQ1}
\|\mathcal{N}(\Omega, 0)[h]\|_{\mathscr{H}\mathscr{C}^{2,\alpha}_{{m}}(\D)}\leq  C\|h\|_{\mathscr{C}^{1,\alpha}_{{s,m}}(\D)}.
\end{equation}
\end{pro}
The next task is to  come back to the density equation \eqref{density-2} and define
\begin{equation}\label{densityEq}
\widehat{G}(\Omega,g):=G(\Omega,g,\mathcal{N}(\Omega,g)),
\end{equation}
where $\mathcal{N}$ is defined via Proposition \ref{propImpl}. Therefore, the main purpose is to find non trivial roots for the functional $\widehat{G}$, which turns out to be more subtle and requires bifurcation tools. In the following proposition we discuss the regularity of $\widehat{G}$.
\begin{pro}\label{Gwelldefined}\label{prop-G-welldef}
Let  $m\geqslant 1, f_0$ satisfy \eqref{H2}-\eqref{H1}, with $\beta>\alpha$,  and $I$ be an open interval with  $\overline{I}\subset\R\backslash( \mathcal{S}^m_{\textnormal{sing}}\cup[\kappa_1,\kappa_2])$. Then, there exists $\varepsilon>0$ such that
$$\widehat{G}:I\times B_{\mathscr{C}_{s,m}^{1,\alpha}(\D)}(0,\E)\to \mathscr{C}_{s,m}^{1,\alpha}(\D),
$$ 
is well--defined and of class $\mathscr{C}^1$. Recall that the singular set $\mathcal{S}^m_{\textnormal{sing}}$, $\kappa_1$ and $\kappa_2$ were defined in \eqref{Interv1}, \eqref{kappa1} and \eqref{kappa2}, respectively. 
\end{pro}
\begin{proof}
Notice that the only difference (in terms of regularity) in  the new density equation \eqref{densityEq} between the  restricted case of the quadratic profile explored in \cite[Section 4.2]{GHS:2020} and the generic monotonic  profile discussed here  is the first term of $\widehat{G}$, that is, 
$$
\mathcal{M}\big(\Omega, f(z)\big)=\mathcal{M}\big(\Omega, f_0(z)+g(z)\big).
$$
By virtue of Lemma \ref{lem-regM} we get under the assumptions  \eqref{H2}-\eqref{H1} that  $t\in [a-\epsilon,b+\epsilon]\mapsto \mathcal{M}(\Omega,t)$ is of class on $\mathscr{C}^{2,\alpha}$ and in particular of class $\mathscr{C}^2$. Therefore, by taking $g$ small enough in $\mathscr{C}_{s,m}^{1,\alpha}(\D)$  and using the composition rules  we deduce that get that $z\mapsto \mathcal{M}\big(\Omega, f_0(z)+g(z)\big)\in\mathscr{C}_m^{1,\alpha}(\D).$ In addition, the Frechet derivative with respect to $g$ takes in view of \eqref{def-MM} the form
$$
\partial_g\mathcal{M}\big(\Omega, f_0+g\big)[h](z)=\frac{h(z)}{{\mu}\big(\Omega, f_0(z)+g(z)\big)}, 
$$ 
with $g\in\mathscr{C}^{1,\alpha}_{s,m}(\D)$ small enough and $h\in\mathscr{C}^{1,\alpha}_{s,m}(\D)$. Notice that by the classical composition and product laws we infer that $\partial_g\mathcal{M}\big(\Omega, f_0+g\big)[h]\in\mathscr{C}^{1,\alpha}_{s,m}(\D)$. Let us now move the continuity of the differential $g\mapsto \partial_g\mathcal{M}\big(\Omega, f_0+g\big)$. From direct computations and  law products we  get for $g_1,g_2\in\mathscr{C}^{1,\alpha}_s(\D)$ and small enough
\begin{align*}
\big\|\partial_g\mathcal{M}\big(\Omega, f_0+g_1\big)[h]-\partial_g\mathcal{M}\big(\Omega, f_0+g_2\big)[h]\big\|_{\mathscr{C}^{1,\alpha}(\D)}\lesssim \|h\big\|_{\mathscr{C}^{1,\alpha}(\D)} \|\tfrac{1}{{\mu}(\Omega, f_0+g_1)} -\tfrac{1}{{\mu}(\Omega, f_0+g_2)} \|_{\mathscr{C}^{1,\alpha}(\D)}.
\end{align*} 
Denote $\widehat{\mu}(t)=\frac{1}{\mu(t)}$, then $\widehat\mu\in \mathscr{C}^{1,\beta}([a-\epsilon, b+\epsilon])$. Then straightforward computations yield
\begin{align*}
\| \partial_x\big({\widehat{\mu}( f_0+g_1)} -{\widehat{\mu}(f_0+ g_2)}\big) \|_{\mathscr{C}^{\alpha}(\D)} &\leqslant \| {\widehat{\mu}^\prime( f_0+g_1)} -{\widehat{\mu}^\prime(f_0+ g_2)}\big) \|_{\mathscr{C}^{\alpha}(\D)}\|f_0+g_1\|_{\mathscr{C}^{1,\alpha}(\D)}\\
&+\| {\widehat{\mu}^\prime( f_0+g_2)} \|_{\mathscr{C}^{\alpha}(\D)}\|g_1-g_2\|_{\mathscr{C}^{1,\alpha}(\D)}\\
&\lesssim \| {\widehat{\mu}^\prime( f_0+g_1)} -{\widehat{\mu}^\prime(f_0+ g_2)}\big) \|_{\mathscr{C}^{\alpha}(\D)}+\|g_1-g_2\|_{\mathscr{C}^{1,\alpha}(\D)}.
\end{align*} 
By an interpolation inequality we deduce that
\begin{align*}
\| {\widehat{\mu}^\prime( f_0+g_1)} -{\widehat{\mu}^\prime(f_0+ g_2)} \|_{\mathscr{C}^{\alpha}(\D)}&\lesssim \| {\widehat{\mu}^\prime( f_0+g_1)} -{\widehat{\mu}^\prime(f_0+ g_2)} \|_{L^\infty(\D)}^{\frac{\beta-\alpha}{\beta}}\\
&\times\left(  \| {\widehat{\mu}^\prime( f_0+g_1)}\|_{\mathscr{C}^{\beta}(\D)}^{\frac{\alpha}{\beta}}+ \|{\widehat{\mu}^\prime(f_0+ g_2)} \|_{\mathscr{C}^{\beta}(\D)}^{\frac{\alpha}{\beta}} \right)\\
&\lesssim \| {\widehat{\mu}^\prime( f_0+g_1)} -{\widehat{\mu}^\prime(f_0+ g_2)} \|_{L^\infty(\D)}^{\frac{\beta-\alpha}{\beta}}.
\end{align*} 
Using the definition of the H\"older norm implies 
\begin{align*}
\|{\widehat{\mu}^\prime( f_0+g_1)} -{\widehat{\mu}^\prime(f_0+ g_2)} \|_{L^\infty(\D)}&\lesssim  \| {\widehat{\mu}^\prime}\|_{\mathscr{C}^{\beta}(\D)}\|g_1-g_2\|_{L^\infty(\D)}^\beta\\
&\lesssim  \| {\widehat{\mu}^\prime}\|_{\mathscr{C}^{\beta}(\D)}\|g_1-g_2\|_{\mathscr{C}^{1,\alpha}(\D)}^\beta.
\end{align*} 
Hence  the map $g\in\mathscr{C}^{1,\alpha}_s(\D)\mapsto \partial_g\mathcal{M}\big(\Omega, f_0+g\big)$ is continuous. It follows that $g\in B_{\mathscr{C}_{s,m}^{1,\alpha}(\D)}(0,\E)\mapsto \mathcal{M}\big(\Omega, f_0+g\big)\in \mathscr{C}_{s,m}^{1,\alpha}(\D)$ is of class $C^1.$ This achieves the proof of the desired result.
\end{proof}
Thanks to the choice of $\mu$ as \eqref{mutrivial}, we achieve that
$$
\widehat{G}(\Omega,0)=0, \quad \forall\Omega\notin[\kappa_1,\kappa_2].
$$
Hence, the proof of Theorem \ref{th-intro} reduces to finding nontrivial roots of the nonlinear functional $\widehat{G}$, and the approach that we will follow is based on  the classical Crandall-Rabinowitz theorem in bifurcation theory, whose statement can be found in \cite{rabinowitz:simple,kielhofer}. In order to use such theorem, one needs to study some spectral properties of $D_g \widehat{G}(\Omega_0,0)$ for some suitable values  $\Omega_0$. In particular, we should show that $D_g \widehat{G}(\Omega_0,0)$ is a Fredholm operator with zero index and has one-dimensional kernel, together with the so called   transversal condition. In what follows, the main task is to check  the spectral properties of the linear operator  $D_g \widehat{G}(\Omega_0,0)$ in order to apply  later  in \mbox{Section \ref{Sec-main-result}} the Crandall-Rabinowitz theorem.

\section{General  structure of the linearized operator}\label{Sec-lin-op}
In this section, we plan to  compute the G\^ateaux derivative at the trivial solution $(\Omega,0)$ of the functional $\widehat{G}$ defined through \eqref{densityEq}. We will show that this operator is Fredholm of zero index since it can be described by a  compact perturbation of an isomorphism. Later, we  will give the expression of the linear operator $D_g \widehat{G}(\Omega,0)$ using  Fourier series expansion. We show in particular that this operator  can  be described through a countable family of one-dimensional operators with variable coefficients. This will be helpful later  in the kernel study that will be developed in Section \ref{Sec-kernel} and turns out to be tricky and very involved. 

\subsection{Fourier expansion }\label{Fourier-expans1}
The G\^ateaux derivative at the trivial solution $(\Omega,0)$ of the functional $\widehat{G}$ can be done in a straightforward way following the same computations as in \cite{GHS:2020}.  Then according to \eqref{def-MM} together with  the computations performed in   \cite[Section 5.1]{GHS:2020} we  find
\begin{align}\label{FormDif}
\nonumber D_g \widehat{G}(\Omega,0)[h](z)=&\frac{h(z)}{\mu(\Omega,f_0(r))}+\frac{1}{2\pi}\int_{\D} \log |z-y|h(y)dA(y)-\Omega\textnormal{Re}[\overline{z}k(z)]\\
\nonumber&+\frac{1}{2\pi}\int_{\D}\textnormal{Re}\left[\frac{k(z)-k(y)}{z-y}\right]f_0(y)dA(y)\\
&+\frac{1}{\pi}\int_{\D} \log|z-y|f_0(y)\textnormal{Re}[k'(y)]dA(y),
\end{align}
where  $\mu$ is given by the compatibility condition \eqref{mutrivial} and $k(z)=D_g\mathcal{N}(\Omega,0)[h]$ is detailed  in Proposition \ref{propImpl}. From now on we will work with the linear operator and we will denote 
\begin{align}\label{mu0-L}
\mu_\Omega^0(r):=\mu(\Omega,f_0(r)).
\end{align}
Next, we shall  give a more explicit form of the linearized operator using Fourier expansion.  Given a function $h\in\mathscr{C}^{k,\alpha}_s(\D)$ for some  $s\in\R$
$$r e^{i\theta}\in\D\mapsto h(re^{i\theta})=\sum_{n\in\N}h_n(r)\cos(n\theta)
,$$
then using \eqref{FormDif} together with  \cite[Proposition B.5]{GHS:2020}  we get that
\begin{align}\label{lin-op}
D_g \widehat{G}(\Omega,0)[h](re^{i\theta})=&\mathbb{L}_0^\Omega [h_0](r)+\sum_{n\geqslant 1} \cos(n\theta)\mathbb{L}_n^\Omega [h_n](r).
\end{align}
where
\begin{align*}
\quad \forall r\in(0,1),\quad\forall\, n\geqslant 1,\quad \mathbb{L}_n^\Omega [h](r):=&\frac{h(r)}{\mu_\Omega^0(r)}-\frac{r}{n}\left(G_n(r)A_n[h]+\frac{1}{2r^{n+1}}H[h](r)\right)\\
\mathbb{L}_0^\Omega [h](r):=&\frac{h(r)}{\mu_\Omega^0(r)}-\int_r^1\frac{1}{\tau}\int_0^\tau sh(s)dsd\tau,\nonumber
\end{align*}
with
\begin{align}\label{Hn-def}
 H_n[h](r)&:= r^{2n}\int_r^1s^{1-n}h(s)ds+\int_0^rs^{n+1}h(s)ds,
 \end{align}
 and
 \begin{align}\label{Gn-def}
 G_n(r):=& n\Omega r^{n+1}+r^{n-1}\int_0^1sf_0(s)ds-(n+1)r^{n-1}\int_0^rsf_0(s)ds\\
\nonumber  &
+\frac{n+1}{r^{n+1}}\int_0^rs^{2n+1}f_0(s)ds.
\end{align}
The value of $A_n[h]$  is given by \eqref{An} and agrees with
\begin{equation*}
A_n[h]={\frac{\displaystyle{\int_0^1s^{n+1}h(s)ds}}{2n\big(\widehat{\Omega}_n-\Omega\big)}}.
\end{equation*}
Moreover, there is another useful expression for $A_n[h]$ coming from the value of $G_n(1)$
\begin{align}\label{Gn1-1}
G_n(1)=n\left[\Omega-\int_0^1sf_0(s)ds+\frac{n+1}{n}\int_0^1s^{2n+1}f_0(s)ds\right]
=n\left(\Omega-\widehat{\Omega}_n\right).
\end{align}
Then we find
\begin{equation}\label{An-HHH}
A_n[h]=-\frac{H_n[h](1)}{2G_n(1)}, \quad \forall n\geqslant 1.
\end{equation}

\subsection{Fredholm structure}

Here, we intend to explore the Fredholm structure of the linearized operator described by  \eqref{FormDif}. 
\begin{pro}\label{compactop} 
Let $m\geqslant1, f_0$ satisfy \eqref{H2}-\eqref{H1} and $\alpha\in(0,1)$.
Then, for $\Omega\notin[\kappa_1,\kappa_2]\cup\mathcal{S}Pm_{\textnormal{sing}}$,  the linearized operator  
$ D_g \widehat{G}(\Omega,0):\mathscr{C}^{1,\alpha}_{s,m}(\D)\to \mathscr{C}^{1,\alpha}_{s,m}(\D)$ is a Fredholm operator with zero index.  
\end{pro}
\begin{proof}
We follow  exactly the same ideas   of \cite[Proposition 5.1]{GHS:2020}. 
We emphasize that compared to this reference, the only difference concerns the first term of the right hand side in \eqref{FormDif}. Thus to get the desired result it is enough  to check that
$$
\frac{1}{\mu_\Omega^0}\textnormal{Id}:\mathscr{C}^{1,\alpha}_{s,m}(\D)\to \mathscr{C}^{1,\alpha}_{s,m}(\D),
$$
is an isomorphism, where $\mu_\Omega^0$ is defined through \eqref{mu0-L} and \eqref{mutrivial}. This can be easily ensured first by the fact that $\frac{1}{\mu_\Omega^0}$ is non vanishing in $[0,1]$ for any   $\Omega\notin[\kappa_1,\kappa_2]$ and it belongs  together with its inverse to the  space $\mathscr{C}^{1,\alpha}_{s,m}(\D)$.  Second, we use the products law in H\"{o}lder spaces $\mathscr{C}^{1,\alpha}_{s,m}(\D)$.
\end{proof}
\section{Kernel study}\label{Sec-kernel}
The main concern  of this section is to find  the values   $\Omega$, sometimes called {\it eigenvalues}, such that the kernel of $D_g \widehat{G}(\Omega,0)$ is one dimensional, which is  required for Crandall-Rabinowitz theorem.  As we shall explore, the kernel structure is related in view of the Fourier expansion stated in \eqref{lin-op} to the study of a countable family   of one-dimensional    Sturm-Liouville problems $\mathbb{L}_m^{\Omega}$ indexed with the  parameters $m$ and $\Omega$. Then, we reduce in this way the study to exploring a nice integral dispersion equation depending on the structure of the stationary profile $f_0$. Notice that surprisingly we are able to transform  the dispersion equation into a tractable form which  is quite similar  to the special quadratic case analyzed before in \cite[Section 6]{GHS:2020}.  Its study is  quite involved and  reveals  two main regimes where we can achieve a complete study depending on the sign of the profile $f_0$. Actually, we show that  for $f_0>0$ we have a scarcity of the eigenvalues living in the region $(-\infty,\kappa_1)$ and associated to lower symmetry  $m\leqslant\frac{\kappa_1}{\kappa_2-\kappa_1},$ as stated in \mbox{Proposition \ref{prop-omegan-low}-$(4)$.} However, the case  $f_0<0$ is more rich and  corresponds to the abundance regime where we find an infinite countable family of eigenvalues  associated to large symmetry $m\geqslant m_0$. These eigenvalues live in the region $(\kappa_2,\infty)$ and decrease to $\kappa_2$. For more details, see Proposition \ref{prop-omegam-asymp}. It is worthy to point out that we retrieve here in this general study   the main  feature  with similar regimes of the particular case of quadratic shapes analyzed in \cite{GHS:2020}.
\subsection{Kernel description}
In what follows, we shall write down the constraints on  the kernel elements of the linearized operator $D_g \widehat{G}(\Omega,0)$ and try along the next sections to encode them  analytically through what is called a {\it dispersion equation}. Coming back to  \eqref{lin-op} we deduce that the kernel is described by
\begin{align}\label{ker-op}
\mbox{Ker} D_g \widehat{G}(\Omega,0)=\left\{h=\sum_{{n\in m\N}}h_n(r)\cos(n\theta)\in \mathscr{C}^{1,\alpha}_{s,m}(\D),\quad\exists n\in m\N\quad \hbox{s.t.}\quad \mathbb{L}_n [h_n]=0\right\}.
\end{align}
Therefore the kernel study reduces to the analysis of the kernel of the stratified one dimensional operator $\mathbb{L}_n$ whose coefficients are not constant but depend on the variable $r$ through the profile $f_0.$  This makes the study more tricky compared to   the vortex patch problem where the involved  operator is described by  a Fourier multiplier in the angular variable. \\
 By the definition of $\mathbb{L}_n^\Omega$ stated after \eqref{lin-op}, we find that
$$
\mathbb{L}^\Omega_n [h_n]=0
$$
is equivalent to 
\begin{equation}\label{kernel-eq}
h_n(r)-\tfrac{\mu_\Omega^0(r)}{2nr^{n}}H_n[h_n](r)=-\tfrac{H_n[h_n](1)}{2nG_n(1)}\mu_\Omega^0(r)r G_n(r),
\end{equation} 
for any {$n\in m\N^\star,$} and for $n=0$ to
\begin{align}\label{kernel-eqL0}
\mathbb{L}_0^\Omega [h](r)=\frac{h(r)}{\mu_\Omega^0(r)}-\int_r^1\frac{1}{\tau}\int_0^\tau sh(s)dsd\tau =0.
\end{align}
Let us recall that the function $\mu_\Omega^0$,  defined according to  \eqref{mu0-L} and \eqref{mutrivial}, keeps a constant sign   depending on the location of $\Omega\notin [\kappa_1,\kappa_2]$. 
For this reason, we shall devise from it a new positive function $\nu_\Omega$ needed  to construct a Hilbert space with positive measure. Set,
\begin{align}\label{nu-mu}
\nu_\Omega(r):=\sigma _\Omega\,\mu_\Omega^0(r)=|\mu_\Omega^0(r)|,\quad\hbox{with}\quad 
 \sigma_\Omega:=\begin{cases}
-1,\quad \Omega\in(-\infty,\kappa_1),\\
1,\quad \Omega\in(\kappa_2,+\infty).
\end{cases}
\end{align}
Then, one has by virtue of \eqref{H2}-\eqref{H1} that  $\nu_\Omega(r)> 0$, for any $r\in[0,1]$  and $\Omega\notin[\kappa_1,\kappa_2].$
Consider  the positive Borel measure
\begin{align}\label{Eq-lambda}
d\lambda_\Omega(s):=\frac{s}{\nu_\Omega(s)}ds,
\end{align}
and define the Hilbert space $L^2(\lambda_{\Omega})$   as the set of measurable functions $f:[0,1]\rightarrow \R$ such that
\begin{align}\label{Norm-L2}
\|f\|^2_{\Omega}=\int_0^1 |f(s)|^2 d\lambda_\Omega(s)<\infty,
\end{align}
equipped with the standard  inner product:
\begin{align}\label{scalar-prod}
\langle f,g\rangle_\Omega=\int_0^1 f(r)g(r)d\lambda_\Omega(r).
\end{align}
For the sake of simple notation, we  denote $L^2_{\Omega}$ instead of  $L^2(\lambda_{\Omega})$.\\
In the following result we state a lower and upper bound estimate for  the function $\nu_\Omega$ that  will be very useful later.
\begin{lem}\label{prop-InvertT1}
Let $f_0$ satisfy \eqref{H2}-\eqref{H1}.  There exists $C_0>0$ depending only on $f_0$ such that for any  $\Omega\in(\kappa_2,\infty)$ and for any $ \theta\in[0,1]$. 
\begin{align*} 
\forall \, r\in[0,1],\quad \frac{\frac{f_0^\prime(r)}{r}}{\Omega-\kappa_1}\leqslant {\nu}_\Omega(r)
&\leqslant \frac{C_0 }{(\Omega-\kappa_2)^\theta}\frac{\frac{f_0^\prime(r)}{r}}{(1-r)^{1-\theta}}\cdot
\end{align*}

\end{lem}
\begin{proof}
The lower bound is trivial using the definition of $\kappa_1$ introduced in \eqref{kappa1} and the fact that $\Omega>\kappa_2$. For the upper one, recall from \eqref{nu-mu} and \eqref{mutrivial} that  for $\Omega>\kappa_2$,
\begin{align*}
{\nu}_\Omega(r)=\frac{\frac{f_0'(r)}{r}}{\Omega-\int_0^1sf_0(rs)ds},\quad \forall\,  r\in  (0,1].
\end{align*}
On the one hand, since  $f_0$ is increasing  we get in view of \eqref{kappa2} that
\begin{align}\label{nu11}
\frac{1}{\Omega-\int_0^1sf_0(rs)ds}&\leqslant\frac{1}{\Omega-\int_0^1sf_0(s)ds}\leqslant \frac{1}{\Omega-\kappa_2}\cdot
\end{align}
On the other hand, it is obvious that
\begin{align*}
\frac{1}{\Omega-\int_0^1sf_0(rs)ds}&\leqslant\frac{1}{\kappa_2-\int_0^1sf_0(rs)ds}\cdot
\end{align*}
Moreover, 
$$
\kappa_2-\int_0^1sf_0(rs)ds=\int_0^1s\big[f_0(s)-f_0(rs)\big]ds.
$$
Hence,   by using the monotonicity of $f_0$ we achieve  for $r\in[0,\frac12]$ 
$$
\kappa_2-\int_0^1sf_0(rs)ds\geqslant \int_0^1s\big[f_0(s)-f_0(\tfrac12s)\big]ds:=C_1.
$$
As to  the case  $r\in[\frac12,1)$, we can use Taylor formula leading to
\begin{align*}
\kappa_2-\int_0^1sf_0(rs)ds&=\int_0^1s\big[f_0(s)-f_0(rs)\big]ds\\
&=(1-r)\int_0^1\int_0^1s^2f_0^\prime\big(\tau s+(1-\tau)rs\big)d\tau ds.
\end{align*}
Then by the continuity of the double integral term and the positivity of $f_0^\prime$ we find $\overline{r}\in[\frac12,1]$ such that for any $ r\in[\tfrac12,1],$
$$ \int_0^1\int_0^1s^2f_0^\prime\big(\tau s+(1-\tau)rs\big)d\tau ds\geqslant \int_0^1\int_0^1s^2f_0^\prime\big(\tau s+(1-\tau)\overline{r}s\big)d\tau ds:=C_2>0.
$$
It  follows that
\begin{align*}
\kappa_2-\int_0^1sf_0(rs)ds
&\geqslant C_2(1-r).
\end{align*}
Consequently, a constant $C_3>0$ exists such that for any $r\in[0,1]$
\begin{align*}
\kappa_2-\int_0^1sf_0(rs)ds\geqslant C_3(1-r).
\end{align*}
Interpolating this inequality with \eqref{nu11} we get for any $\theta\in[0,1]$  
\begin{align}\label{nu102}
 \frac{1}{\Omega-\int_0^1 s f_0(rs)ds}
&\leqslant \frac{C_0}{(\Omega-\kappa_2)^\theta}\frac{1}{(1-r)^{1-\theta}},
\end{align}
ensuring the result of the lemma.
\end{proof}

\subsection{Kernel of $\mathbb{L}_0^\Omega$}
In this part, we intend to solve the equation \eqref{kernel-eqL0} and  it appears that the difficulty  varies according  to the sign of $\mu_\Omega^0$  or equivalently to the location of $\Omega$ with respect to the interval $[\kappa_1,\kappa_2]$. The case $\Omega\in(-\infty,\kappa_1)$ corresponding to the positivity of $\mu_\Omega^0$ is slightly  easy to tackle in view of the positivity of the operator  $\mathbb{L}_0^\Omega$ implying in turn that its  kernel is trivial. However the case $\Omega\in(\kappa_2,+\infty)$   turns out  to be more complicated and  the kernel is shown to be trivial only when $\Omega$ belongs to a finite set $\mathbb{S}.$ The proof of this last point is based on a refined analysis on a Sturm-Liouville equation governing  the kernel equation, and  a key ingredient is a non oscillation theorem discovered by Kneser  \cite{Kneser} together with Pr\"{u}fer transformation.\\
Coming back to \eqref{kernel-eqL0}  and set 
\begin{align}\label{kernel-eqL00}
\widehat{\mathbb{L}}_0^\Omega [h](r)&:=\mu_\Omega^0(r)\mathbb{L}_0^\Omega [h](r)\\
\nonumber &={h(r)}-
\sigma_\Omega {\nu_\Omega(r)}\int_r^1\frac{1}{\tau}\int_0^\tau sh(s)dsd\tau \\
\nonumber &:={h(r)}-
\sigma_\Omega\mathcal{L}^\Omega_0[h](r).
\end{align}
Since  $\mu_\Omega^0$ is not vanishing in $[0,1]$ then the operators $\widehat{\mathbb{L}}_0^\Omega$ and $\mathbb{L}_0^\Omega$ admit the same kernel.
The main result of this section reads as follows.
\begin{pro}\label{radialfunctions}
Let $f_0$ satisfy \eqref{H2}-\eqref{H1}. Then the following results hold true.
\begin{enumerate}
\item For any $\Omega\notin [\kappa_1,\kappa_2]$, the operator $\mathcal{L}^\Omega_0:L^2_{\Omega}\to L^2_{\Omega}$ is a self-adjoint positive-definite compact operator.
\item
If $\Omega\in (-\infty,\kappa_1)$, then 
$$
\mbox{Ker }\mathbb{L}_0^\Omega=\{0\}.
$$
\item There exists a finite set $\mathbb{S}\subset (\kappa_2,\infty)$ such that 
$$
\forall \Omega\in (\kappa_2,\infty)\backslash \mathbb{S},\quad \mbox{Ker }\mathbb{L}_0^\Omega=\{0\}.
$$

\end{enumerate}
\end{pro}
\begin{proof}
\medskip \noindent
{\bf{(1)}} First, we write according to \eqref{kernel-eqL00} combined with integration by parts and the definition \eqref{Eq-lambda}
$$
\mathcal{L}^\Omega_0[h](r)= \int_0^1 K_0(r,s)h(s)d\lambda_\Omega(s),
$$
with 
$$
K_0(r,s):= \nu_\Omega(r)\nu_\Omega(s)\left[\ln (\tfrac{1}{r}){\bf 1}_{[0,r]}(s)+\ln (\tfrac{1}{s}){\bf 1}_{[r,1]}(s)\right].
$$
This kernel is  obviously  positive and symmetric  and  from direct computations we infer that
$$
\int_0^1\int_0^1 K_0^2(r,s)h(s)d\lambda_\Omega(r)d\lambda_\Omega(s)<\infty,
$$
 implying that $\mathcal{L}^\Omega_0:L^2_\Omega\rightarrow L^2_\Omega$  is a Hilbert-Schmidt integral operator. Consequently, $\mathcal{L}^\Omega_0$ is a self-adjoint compact operator. It remains to check that this operator is definite positive, which is equivalent to show
$$
\langle \mathcal{L}^\Omega_0[h],h\rangle _\Omega>0,
$$
for any $h\in L^2_\Omega$ non identically zero. To proceed, we write in view of \eqref{scalar-prod}
 \begin{align*}
\langle \mathcal{L}^\Omega_0[h],h\rangle_\Omega=&\int_0^1\int_0^r \ln(\tfrac1r) h(s)h(r)srdsdr+\int_0^1\int_r^1 \ln(\tfrac1s) h(s) h(r) rsdsdr\\
=&\int_0^1 \ln(\tfrac1r) h(r) r\int_0^r sh(s) dsdr+\int_0^1 rh(r)\int_r^1 \ln(\tfrac1s) h(s)sdsdr.
\end{align*}
Note that
$$
\int_r^1 \ln(s) h(s) sds=-\ln(r)\int_0^r sh(s)ds-\int_r^1 \frac{1}{s}\int_0^s \tau h(\tau) d\tau ds,
$$
and therefore we deduce by integration by parts that
\begin{align}\label{posit-imp}
\nonumber \langle \mathcal{L}^\Omega_0[h],h\rangle _\Omega
=&\int_0^1 rh(r)\int_r^1 \frac{1}{s}\int_0^s\tau h(\tau)d\tau ds dr\\
=&\int_0^1\left(\int_0^r s h(s)d s\right)^2   \frac{dr}{r} >0,
\end{align}
which ensures the announced result.

\medskip \noindent
{\bf{(2)}}
 The kernel equation \eqref{kernel-eqL00} reduces to the following integral problem,
\begin{align}\label{Eq-lin-rad}
\sigma_\Omega \mathcal{L}^\Omega_0[h](r)=h(r).
\end{align}
By virtue of \eqref{nu-mu}, if  $\Omega\in(-\infty,\kappa_1)$ then $\sigma_\Omega=-1$ and therefore \eqref{Eq-lin-rad} becomes
\begin{align}\label{aaa}
 \mathcal{L}^\Omega_0[h](r)+h(r)=0.
\end{align}
Taking the scalar product in  $L^2_\Omega$  (which refers to $L^2(\lambda_{\Omega})$),  of the function in \eqref{aaa} with $h$, together with \eqref{posit-imp}  yields
\begin{align*}
 \int_0^1\left(\int_0^r s h(s)d s\right)^2   \frac{dr}{r}+\|h\|_{\Omega}^2=0.
\end{align*}
This implies that $h$ is identically zero and thus the kernel of $\mathbb{L}^\Omega_0$  is trivial.

\medskip \noindent
{\bf{(3)}} According to  \eqref{nu-mu}, if  $\Omega\in(\kappa_2,\infty)$ then $\sigma_\Omega=1$ and therefore \eqref{kernel-eqL00} becomes
$$
{h(r)}=
 {\nu_\Omega(r)}\int_r^1\frac{1}{\tau}\int_0^\tau sh(s)dsd\tau,$$
with $h\in L^2_\Omega$. Assume now that this Volterra integral equation admits a nonzero smooth solution $h$  and let us   define $F:=\frac{h}{\nu_\Omega}$. Then this   nonzero  function satisfies 
\begin{align}\label{Eq-ty1}
\forall r\in[0,1],\quad \int_r^1\frac{1}{\tau}\int_0^\tau s \nu_\Omega (s) F(s)dsd\tau=F(r),
\end{align}
with
\begin{align}\label{tuni-12}
\int_0^1 F^2(r) r\nu_\Omega(r)dr<\infty.
\end{align}
We shall show that $F$ is actually continuous on $[0,1]$.
Applying Lemma \ref{prop-InvertT1} with $\theta=1$ yields  
\begin{align}\label{mu-omPL1}
 0\leqslant r\,{\nu}_\Omega(r)
&\leqslant \frac{\|f_0^\prime\|_{L^\infty}}{\Omega-\kappa_2},\quad \forall\,  r\in (0,1].
\end{align}
Combining \eqref{mu-omPL1} with Cauchy-Schwarz inequality and \eqref{tuni-12} allows to get
\begin{align*}
\nonumber \left(\int_0^\tau |F(s)|\, s \nu_\Omega (s) ds\right)^2&\leqslant \int_0^\tau |F(s)|^2 s \nu_\Omega (s)  ds \int_0^\tau s \nu_\Omega (s) ds \leqslant C_0 \tau.
\end{align*}
Consequently, the left hand side in \eqref{Eq-ty1} is well-defined uniformly in $r\in[0,1]$ and therefore from this identity we deduce that $F$ is continuous on $[0,1].$
Using this property together with the continuity of $\nu_\Omega$ we deduce that the left hand side in \eqref{Eq-ty1} is $\mathscr{C}^1$ on $[0,1]$. Actually, we get that $F$ is of class $\mathscr{C}^2$ on $[0,1]$, note that using a change of variable in \eqref{Eq-ty1} we get
$$
F''(r)=\int_0^1 s\nu_\Omega(sr)F(sr)ds-\nu_\Omega(r)F(r).
$$
 Now, differentiating  \eqref{Eq-ty1} we find the Sturm-Liouville equation
$$
(rF^\prime)^\prime+r \nu_\Omega(r)F=0,
$$
supplemented with the condition (which follows from  \eqref{Eq-ty1})
\begin{align}\label{Bc_Mon}
F(1)=0.
\end{align}
Introduce the auxilary function $G_\Omega:[1,\infty)\to \R$ through 
$$
\forall r\in[0,1)\quad F(r)=G_\Omega(\tfrac{1}{1-r}).
$$
Notice that $G_\Omega$ is of class $\mathscr{C}^2$ on $[1,\infty)$. 
Then by straightforward computations and making use of  the notations
 $$y=\frac{1}{1-r},\quad {\mu}_\Omega(y):={\nu}_\Omega(1-\tfrac{1}{y}),
$$
we deduce that
\begin{align}\label{Gomega1}
\left((y^2-y)G_\Omega^{\prime}(y)\right)^\prime+\tfrac{y-1}{y^3}{\mu}_\Omega(y)G_\Omega(y)=0, y\in(1,\infty).
\end{align}
In addition, the boundary condition \eqref{Bc_Mon} writes
\begin{equation}\label{Bc-H1}
\lim_{y\to\infty}G_{\Omega}(y)=0.
\end{equation}
If $G_\Omega(1)=0$ then from the equation \eqref{Gomega1} we get that $G_\Omega^\prime(1)=0$ and by a slight adaptation of the uniqueness result for the Cauchy problem with singular coefficients we should get $G_\Omega$ is identically zero. This contradicts the fact that $F$ is a nonzero function. Notice that this argument  shows that the set of smooth solutions on $[1,\infty)$ to this singular ODE is of dimension one generated by one nonzero element  still denoted by $G_\Omega$ and satisfies  by normalization 
\begin{align}\label{Initial1}G_\Omega(1)=1.
\end{align}
In this way we uniquely  construct a mapping $\Omega\in(\kappa_2,\infty)\mapsto G_\Omega$ and each element $G_\Omega$ is of   class $\mathscr{C}^2$ on $[1,\infty)$  and satisfies \eqref{Gomega1}. The next task is to  explore  the values $\Omega$ for which $G_\Omega$ matches with the boundary condition  \eqref{Bc-H1}.
First, we shall show that  $G_\Omega$ does not oscillate close to $\infty$ for any $\Omega\geqslant\kappa_2$. For this aim we introduce the auxiliary function 
$$
v(y)=\sqrt{y^2-y}\, G_\Omega(y).
$$
Then from straightforward computations we get
\begin{align}\label{second-orde}
v^{\prime\prime}(y)+V_\Omega(y) v(y)=0,
\end{align}
with
$$
V_\Omega(y)=\frac{1}{4}\frac{1}{(y^2-y)^2}+\frac{1}{y^4} {\mu}_\Omega(y).
$$
Kneser's  Theorem  \cite[pp. $414-418$]{Kneser} states that a sufficient condition of  the non-oscillation for any nonzero solution  (that is, it has a finite number of zeroes in the region $[y_0,\infty)$ for some $y_0>1$) to \eqref{second-orde} or \eqref{Gomega1}  at infinity is
\begin{align}\label{Knese1}
\limsup_{y\to \infty} y^2V_\Omega(y)<\frac14\cdot
\end{align}
Using Lemma \ref{prop-InvertT1} with $\theta=0$ we get a constant $C_0$  such that for any $\Omega\geqslant \kappa_2$,
$$
0\leqslant V_\Omega(y)\leqslant \frac{1}{4}\frac{1}{(y^2-y)^2}+\frac{C_0}{y^3}\cdot
$$
Thus
\begin{align*}
\limsup_{y\to \infty} y^2V_\Omega(y)=0<\frac14\cdot
\end{align*}
Consequently, there is no oscillation at $\infty$ for any $\Omega\geqslant \kappa_2$. This implies  that for $\Omega\geqslant \kappa_2$,  there exists a real number $y_\Omega>1$ such that
\begin{align}\label{Knese2}
\forall y\geqslant y_\Omega, \quad |G_\Omega(y)|>0.
\end{align}
Next, we come back to \eqref{Gomega1} and use  Pr\"{u}fer transformation through the relations  \begin{align}\label{pourq1}
G_\Omega(y)=\rho(y)\sin(\theta_\Omega(y)),\quad (y^2-y)G_\Omega^\prime(y)=\rho(y)\cos(\theta_\Omega(y)), \forall y\geqslant 1.
\end{align}
Then, we  find the  weakly coupled ODE's
\begin{align}\label{pourT8}
\theta_\Omega^\prime(y)=\frac{1}{y^2-y}\cos^2(\theta_\Omega(y))+\frac{y-1}{y^3} {\mu}_\Omega(y)\sin^2(\theta_\Omega(y))
\end{align}
and
\begin{align}\label{pour8}
\rho_\Omega^\prime(y)=\left(\frac{1}{y^2-y}-\frac{y-1}{y^3} {\mu}_\Omega(y)\right)\rho_\Omega(y)\sin(\theta_\Omega(y))\cos(\theta_\Omega(y)).
\end{align}
The boundary  condition \eqref{Initial1} matches with (note that $\theta_\Omega(1)$ is not uniquely determined)
\begin{align}\label{Ini-conw}
\rho_\Omega(1)=1,\quad \theta_\Omega(1)=\tfrac\pi2.
\end{align}
The system \eqref{pourT8} and \eqref{pour8} is globally well-defined  for any  $y\in[1,\infty).$ Notice that by uniqueness of the Cauchy problem for \eqref{pour8} we should get 
 $$\rho_\Omega(y)> 0,\forall y\in [1,\infty).
 $$
Consequently, for any $\Omega\geqslant\kappa_2,$ 
 $$
\mathcal{Z}(G_\Omega):= \Big\{y>1, G_\Omega(y)=0 \Big\}=\Big\{ y>1, \theta_\Omega(y)\in \pi \mathbb{Z}\Big\}.
 $$
We remark from \eqref{pourT8} that $y\mapsto \theta_\Omega(y)$ is strictly increasing. Since the ODE is not oscillating then   the set $\mathcal{Z}(G_\Omega)$ is finite implying that $\theta_\Omega$ is bounded for each $\Omega\geqslant\kappa_2$. Hence $\lim_{y\to\infty}  \theta_\Omega(y)$ exists and is finite. Denote by
\begin{align}\label{Conv-ppp}
\overline\theta(\Omega)=\lim_{y\to\infty}  \theta_\Omega(y),
\end{align}
then for any $\Omega\geqslant \kappa_2,$
\begin{align}\label{Es-ui}
\forall\, y>1,\quad \tfrac\pi2<\theta_\Omega(y)<\overline\theta(\Omega).
\end{align}
We intend to show that $\Omega\in[\kappa_2,\infty)\mapsto \overline\theta(\Omega)$ is strictly decreasing. Before that we shall show that for any $y>1$, $\Omega\in[\kappa_2,\infty)\mapsto \theta_\Omega(y)$ is strictly decreasing. Set $\dot{\theta}_\Omega(y)=\partial_\Omega  \theta_\Omega(y),$ then differentiating \eqref{pourT8} yields
\begin{align}\label{pourT008}
\nonumber \dot\theta_\Omega^\prime(y)&=\left(-\frac{1}{y^2-y}+\frac{y-1}{y^3}{\mu}_\Omega(y)\right)\sin(2\theta_\Omega(y))\,\dot\theta_\Omega(y)-\frac{(y-1) {\mu}_\Omega(y)}{y^3(\Omega-\int_0^1sf_0(ys)ds)}\sin^2(\theta_\Omega(y))\\
&:=A_\Omega(y)\dot\theta_\Omega(y)-B_\Omega(y).
\end{align}
Remark that from \eqref{Ini-conw} one gets $\dot\theta_\Omega(1)=0$.
{Notice $A_\Omega$ is continuous on $[1,\infty)$ since
$$
\lim_{y\rightarrow 1^-} A_\Omega=2\theta_\Omega'(1)=0,
$$
where we find $\theta_\Omega'(1)=0$ using the differential equation \eqref{pourT8} together with \eqref{Ini-conw}.
}
 Moreover $B_\Omega\geqslant0$ and not identically zero for any open set in $(1,\infty)$. By Duhamel formula we get
\begin{align}\label{Pira1}
\dot\theta_\Omega(y)=-\int_1^y e^{\int_s^y A_\Omega(s) ds}B_\Omega(s)ds.
\end{align}
Therefore
$$
\forall\, y>1, \quad \dot\theta_\Omega(y)<0.
$$
This shows that $\Omega\in[\kappa_2,\infty)\mapsto \theta_\Omega(y)$ is strictly decreasing, implying in turn  that
$$
\forall y>1,\forall \,\Omega>\kappa_2,\quad \tfrac\pi 2<\theta_\Omega(y)<\theta_{\kappa_2}(y).
$$
Passing to the limit as $y\to \infty$ and using \eqref{Conv-ppp} we infer 
\begin{align}\label{Up0SS1}
\frac\pi 2< \overline \theta_\Omega\leqslant \overline\theta_{\kappa_2}.
\end{align}
Coming back to \eqref{Pira1} we deduce from the positivity of $B_\Omega$ that for $y\geqslant 2$,
\begin{align}\label{Pira2}
\dot\theta_\Omega(y)\leqslant -\int_2^y e^{\int_s^y A_\Omega(\tau) d\tau}B_\Omega(s)ds.
\end{align}
On the other hand, using  Lemma  \ref{prop-InvertT1} with $\theta=1$ we find
\begin{align}
A_\Omega(y)\geqslant -\frac{1}{y^2-y}-\frac{C_0}{y^2(\Omega-\kappa_2)},
\end{align}
implying that for any $2\leqslant s\leqslant y$
\begin{align*}
\int_s^y A_\Omega(\tau)d\tau&\geqslant -\int_{2}^{\infty}\left(\frac{1}{\tau^2-\tau}+\frac{C_0}{\tau^2(\Omega-\kappa_2)}\right)d\tau\geqslant -\ln(2)-\frac{C_0}{\Omega-\kappa_2}.
\end{align*}
The constant $C_0$ may vary from line to line.  Moreover,
\begin{align}\label{minmin}
e^{\int_s^y A_\Omega(\tau)d\tau}
&\geqslant \tfrac{1}{2}e^{-\frac{C_0}{\Omega-\kappa_2}},
\end{align}
and hence  we get from \eqref{Pira2}
\begin{align}\label{Pira3}
\forall\, y\geqslant 2,\quad \dot\theta_\Omega(y)\leqslant -\tfrac{1}{2}e^{-\frac{C_0}{\Omega-\kappa_2}} \int_2^y B_\Omega(s)ds.
\end{align}
Consider $\kappa_2<\Omega_1<\Omega_2$, then Taylor formula implies that
\begin{align}\label{Douy}
\nonumber\overline \theta_{\Omega_2}-\overline \theta_{\Omega_1}=& \big(\overline\theta_{\Omega_2}-\theta_{\Omega_2}(y)\big)+\big(\theta_{\Omega_1}(y)-\overline \theta_{\Omega_1}\big)+\big(\theta_{\Omega_2}(y)-\theta_{\Omega_1}(y)\big)\\
&=\big(\overline\theta_{\Omega_2}-\theta_{\Omega_2}(y)\big)+\big(\theta_{\Omega_1}(y)-\overline \theta_{\Omega_1}\big)+(\Omega_2-\Omega_1)\int_0^1\dot\theta_{\tau\Omega_2+(1-\tau)\Omega_1}(y)d\tau.
\end{align}
This yields in view of \eqref{Pira3}
\begin{align}\label{apri9}
\nonumber\forall y\geqslant 2,\quad \overline \theta_{\Omega_2}-\overline \theta_{\Omega_1}&\leqslant \big(\overline\theta_{\Omega_2}-\theta_{\Omega_2}(y)\big)+\big(\theta_{\Omega_1}(y)-\overline \theta_{\Omega_1}\big)\\
&-\frac{1}{2}e^{-\frac{C_0}{\Omega_1-\kappa_2}}(\Omega_2-\Omega_1)\int_0^1\int_2^y B_{\tau\Omega_2+(1-\tau)\Omega_1}(s)ds d\tau.
\end{align}
By the monotone convergence theorem
$$
\lim_{y\to\infty}\int_0^1\int_2^y B_{\tau\Omega_2+(1-\tau)\Omega_1}(s)ds d\tau=\int_0^1\int_2^\infty B_{\tau\Omega_2+(1-\tau)\Omega_1}(s)ds d\tau.
$$
Taking the limit as $y\to\infty$ in \eqref{apri9} and using \eqref{Conv-ppp} we obtain
\begin{align*}
\overline \theta_{\Omega_2}-\overline \theta_{\Omega_1}&\leqslant 
-\frac{1}{2}e^{-\frac{C_0}{\Omega_1-\kappa_2}}(\Omega_2-\Omega_1)\int_0^1\int_2^\infty B_{\tau\Omega_2+(1-\tau)\Omega_1}(s)ds d\tau.
\end{align*}
Since $B_{\Omega}$ is positive and not identically zero in any open set then
\begin{align}\label{apri10}
\overline \theta_{\Omega_2}-\overline \theta_{\Omega_1}&<0.
\end{align}
This shows that the map $\Omega\in(\kappa_2,\infty)\mapsto \overline \theta_{\Omega}$ is strictly decreasing. By slight modification of the preceding argument we intend to   show that this map is locally  Lipschitz. Indeed, By virtue  of \eqref{Ini-conw}, \eqref{pourT008} and Lemma \ref{prop-InvertT1}
\begin{align*}
|A_\Omega(y)|&\leqslant \frac{\left|\sin(2\theta_\Omega(y))-\sin(2\theta_\Omega(1))\right|}{y^2-y}+\frac{C_0}{y^2(\Omega-\kappa_2)}\\
&\leqslant \frac{C_1}{y^2}+\frac{C_0}{y^2(\Omega-\kappa_2)}\cdot
\end{align*}
In the same way we get
\begin{align*}
0\leqslant B_\Omega(y)&= \frac{(y-1){\mu}_\Omega(y)}{y^3(\Omega-\int_0^1sf_0(rs)ds)}\sin^2(\theta_\Omega)\leqslant  \frac{C_0}{y^2(\Omega-\kappa_2)^2}.
\end{align*}
Plugging these estimates into \eqref{Pira1} we find for any $y\geqslant1$
\begin{align*}
|\dot\theta_\Omega(y)|&\leqslant  e^{\int_1^\infty A_\Omega(s) ds} \int_1^\infty B_\Omega(s)ds\\
&\leqslant C_0 e^{\frac{C_0}{(\Omega-\kappa_2)^2}}.
\end{align*}
Combining this estimate with  \eqref{Douy} we infer
 \begin{align*}
|\overline \theta_{\Omega_2}-\overline \theta_{\Omega_1}|
&\leqslant \big|\overline\theta_{\Omega_2}-\theta_{\Omega_2}(y)\big|+\big|\theta_{\Omega_1}(y)-\overline \theta_{\Omega_1}\big|+(\Omega_2-\Omega_1)\int_0^1\big|\dot\theta_{\tau\Omega_2+(1-\tau)\Omega_1}(y)\big|d\tau\\
&\leqslant \big|\overline\theta_{\Omega_2}-\theta_{\Omega_2}(y)\big|+\big|\theta_{\Omega_1}(y)-\overline \theta_{\Omega_1}\big|+C_0(\Omega_2-\Omega_1)e^{\frac{C_0}{(\Omega_1-\kappa_2)^2}}.
\end{align*}
By taking $y\to \infty$ we get
\begin{align*}
|\overline \theta_{\Omega_2}-\overline \theta_{\Omega_1}|\leqslant C_0(\Omega_2-\Omega_1)e^{\frac{C_0}{(\Omega_1-\kappa_2)^2}}.
\end{align*}
This shows that  $\Omega\in(\kappa_2,\infty)\mapsto \overline \theta_{\Omega}$  is locally Lipschitz and in particular it is continuous on $(\kappa_2,\infty)$.
Finally, we have proved that  $\Omega\in(\kappa_2,\infty)\mapsto \overline \theta_{\Omega}$ is continuous,  strictly decreasing and bounded in view of \eqref{Up0SS1}.  Consequently the set 
\begin{align}\label{Pttt1}
\mathbb{S}:=\big\{\Omega>\kappa_2, \overline \theta_{\Omega}\in \pi \mathbb{Z}\big\},
\end{align}
is finite. Now, 
by virtue of  \eqref{pour8}, \eqref{Ini-conw} and \eqref{pourT008} we write in view of Gronwall equality and \eqref{minmin} ( which remains true with $-A_\Omega$)
\begin{align*}
\forall\, y\geqslant 1,\quad \rho_\Omega(y)&=e^{-\frac12\int_{1}^yA_\Omega(s)ds}\geqslant \tfrac{1}{2}e^{-\frac{C_0}{\Omega-\kappa_2}}\cdot
\end{align*}
Hence, we obtain from \eqref{pourq1} that for any $\Omega>\kappa_2$
\begin{align*}
\forall\, y\geqslant 1,\quad |G_\Omega(y)|\geqslant \tfrac{1}{2}e^{-\frac{C_0}{\Omega-\kappa_2}}|\sin(\theta_\Omega(y))|,
\end{align*}
implying  in view  of \eqref{Bc-H1} and \eqref{Conv-ppp}
\begin{align*}
0= |\lim_{y\to\infty}G_\Omega(y)|&\geqslant \tfrac{1}{2}e^{-\frac{C_0}{\Omega-\kappa_2}}|\sin(\overline{\theta}_\Omega)|.
\end{align*}
Therefore we deduce that necessarily 
\begin{align*}
\overline{\theta}_\Omega\in\pi \Z.
\end{align*}
Consequently,  if the equation \eqref{Eq-ty1} admits a smooth nonzero  solution in $[0,1]$ for some $\Omega>\kappa_2$  then necessarily we should have $\Omega\in \mathbb{S}$ defined in \eqref{Pttt1}, which is proved to be finite. Consequently if $\Omega\in(\kappa_2,\infty)\backslash \mathbb{S}$ then  the equation \eqref{Eq-ty1} admits only the trivial solution in the smooth class. This concludes the proof of the desired result.
\end{proof}

\subsection{Kernel of $\mathbb{L}_n^\Omega, \,n\geqslant1$.}\label{Sec-KerneLn} The main goal of this section  is to explore  some qualitative properties   on the kernel equation \eqref{kernel-eq} that will be used later to derive  the dispersion equation. One of them is related to the positive definite structure of an intermediate Volterra integral operator.   Let  $n\geqslant1$ and  define the operator
\begin{equation}\label{Op-L} 
\mathcal{L}^\Omega_n[h](r):=\frac{\nu_\Omega(r)}{2nr^{n}}H_n[h](r),
\end{equation}
where $H_n$ is given by \eqref{Hn-def}. Then the equation \eqref{kernel-eq} is equivalent to
\begin{equation}\label{kernel-eqbis}
\big(\hbox{Id}-\sigma_\Omega \mathcal{L}^\Omega_n\big)[h_n](r)=-\frac{H_n[h_n](1)}{2nG_n(1)}\mu_\Omega^0(r)r G_n(r),
\end{equation}
where we use the notation in  \eqref{nu-mu}.
 Next, we intend to analyze some spectral properties of the operator $\mathcal{L}^\Omega_n$. For this purpose, we  write it in  the following integral form depending on the positive measure  $\lambda_\Omega$ introduced in \eqref{Eq-lambda},
\begin{align*}
\mathcal{L}^\Omega_n[h](r)=\int_0^1 {K}_n(r,s)h(s) d\lambda_\Omega(s),
\end{align*}
where  its kernel is symmetric  and it has the following expression
\begin{align*}
{K}_n(r,s):=\frac{\nu_\Omega(r)\nu_\Omega(s)}{2n}\left[\left(\frac{r}{s}\right)^n {\bf 1}_{\{r\leqslant s\leqslant 1\}}+\left(\frac{s}{r}\right)^n {\bf 1}_{\{0\leqslant s\leqslant r\}}\right].
\end{align*}

The first main result of this section reads as follows. 

\begin{pro}\label{pro-positivity}
Let $f_0$ satisfy \eqref{H2}-\eqref{H1} and $\Omega\notin[\kappa_1,\kappa_2]$, with the notations \eqref{kappa1} and \eqref{kappa2}. Then,  the operator $\mathcal{L}^\Omega_n: L^2_{\Omega}\rightarrow L^2_{\Omega}$ is a self-adjoint  Hilbert-Schmidt  operator.  In addition, it is positive-definite and   all its eigenvalues are strictly positive.
\end{pro}
\begin{proof}
The operator is symmetric which follows from  ${K}_n(r,s)=K_n(s,r)$. Moreover, we can easily check by using \eqref{mu-omPL1}  that
\begin{align*}
\|K_n\|_{\Omega}^2=&\int_0^1\int_0^1 \nu_\Omega(r)\nu_\Omega(s)rs\left\{\left(\frac{r}{s}\right)^{2n} {\bf 1}_{\{r\leqslant s\leqslant 1\}}+\left(\frac{s}{r}\right)^{2n} {\bf 1}_{\{0\leqslant s\leqslant r\}}\right\}drds\\
\leqslant & C\int_0^1\int_r^1 \left(\frac{r}{s}\right)^{2n}dsdr+C\int_0^1\int_0^r \left(\frac{s}{r}\right)^{2n}dsdr\leqslant  C(\Omega,f_0),
\end{align*}
and therefore  the operator is self-adjoint Hilbert-Schmidt  operator. Let us emphasize that the norm of the Hilbert-Schmidt operator $\mathcal{L}_n^\Omega$ depends on $\Omega$ and $f_0$.

Let us now move to the positivity of $\mathcal{L}^\Omega_n$.
Take a nonzero element $h\in L^2_\Omega$  and let us compute $\langle\mathcal{L}^\Omega_n[h], h\rangle_\Omega$ according to  \ref{scalar-prod}. By the definition  one has
\begin{align*}
\langle\mathcal{L}^\Omega_n[h], h\rangle_\Omega=&\frac{1}{2n}\int_0^1\int_0^1 h(s) h(r) sr\left\{\left(\frac{r}{s}\right)^n{\bf 1}_{[r,1]}(s)+\left(\frac{s}{r}\right)^n{\bf 1}_{[0,r]}(s)\right\}dsdr\\
=&\frac{1}{2n}\int_0^1 h(r) r^{1+n}\int_r^1 h(s) s^{1-n}ds dr+\frac{1}{2n}\int_0^1 r^{1-n}h(r)\int_0^r h(s) s^{1+n} ds dr.
\end{align*}
Integration by parts allows to get
\begin{align*}
\int_r^1 h(s) s^{1-n}ds=&\int_r^1 h(s) s^{n+1} s^{-2n}ds\\
=&\int_0^1 h(\tau)\tau^{n+1}d\tau-r^{-2n}\int_0^r h(\tau) \tau^{n+1}d\tau\\
&+2n\int_r^1 s^{-1-2n}\int_0^{{s}} h(\tau) \tau^{n+1}d\tau ds.
\end{align*}
Hence we get through straightforward computations 
\begin{align*}
\langle\mathcal{L}^\Omega_n[h], h\rangle_\Omega
=&\frac{1}{2n}\int_0^1 h(r) r^{1+n}\left[\int_0^1 h(\tau)\tau^{n+1}d\tau-r^{-2n}\int_0^r h(\tau) \tau^{n+1}d\tau\right]dr\\
&+\int_0^1 h(r) r^{1+n}\left[\int_r^1 s^{-1-2n}\int_0^{{s}} h(\tau) \tau^{n+1}d\tau ds\right] dr\\
&+\frac{1}{2n}\int_0^1 r^{1-n}h(r)\int_0^r h(s) s^{1+n} ds dr.
\end{align*}
Using integration by parts yields 
$$
\int_0^1 h(r) r^{1+n}\int_r^1 s^{-1-2n}\int_0^{{s}} h(\tau) \tau^{n+1}d\tau ds=\int_0^1 r^{-2n-1}\left(\int_0^r s^{n+1}h(s)ds\right)^2 dr.
$$
Therefore  we find the reduced form 
\begin{align}\label{N-iden}
\langle\mathcal{L}^\Omega_n[h], h\rangle_\Omega=&\frac{1}{2n}\left(\int_0^1 h(r) r^{1+n} dr\right)^2+\int_0^1 r^{-2n-1}\left(\int_0^r s^{n+1}h(s)ds\right)^2 dr> 0.
\end{align}
This shows that  $\mathcal{L}^\Omega_n$ is positive-definite. As a consequence, we get that all the eigenvalues are strictly positive.
\end{proof}

Next, we will discuss  the invertibility of the linear operator $(\mbox{Id}-\sigma_\Omega\mathcal{L}_n^\Omega)$ is invertible which  depends on the choice of $\Omega$.

\begin{pro}\label{prop-Invert1}
Let  $f_0$ satisfy \eqref{H2}-\eqref{H1} and $n\geqslant 1$. The following assertions hold true.
\begin{enumerate}
\item If  $\Omega\in(-\infty,\kappa_1)$, then $\textnormal{Id}+\mathcal{L}^\Omega_n: L^2_\Omega\to  L^2_\Omega $ is invertible. 
\item  Let $\Omega\in(\kappa_2,\infty).$ There exists $C_0>0$ depending only on $f_0$ such that if for some $\theta\in(0,1]$
$$
\frac{C_0}{n \theta (\Omega-\kappa_2)^\theta}<1,
$$
then $\textnormal{Id}-\mathcal{L}^\Omega_n$ is invertible.
\end{enumerate}
\end{pro}
\begin{proof}
\medskip\noindent
{\bf(1)} Since $\mathcal{L}^\Omega_n$ is compact then $\textnormal{Id}+\mathcal{L}^\Omega_n$ is Fredholm of zero index. To show that it is invertible it is enough to check that its  kernel is trivial. Assume that we have a nontrivial function $h$ such that
$$h+\mathcal{L}^\Omega_n[h]=0$$
then this implies that $-1$ is a negative eigenvalue of $\mathcal{L}^\Omega_n$ which contradicts the fact this operator is positive-definite seen in Proposition \ref{pro-positivity},  and all its eigenvalues are positive.

\medskip\noindent
{\bf (2)} To show the invertibility of $\textnormal{Id}-\mathcal{L}^\Omega_n$ in the space  $L^2_\Omega$,  it is enough to verify that
$$
\|\mathcal{L}^\Omega_n\|=\sup_{\|h\|_\Omega=1}\langle \mathcal{L}^\Omega_n[h],h\rangle_\Omega<1.
$$
Using \eqref{N-iden} together with Cauchy-Schwarz inequality we get for $\|h\|_\Omega=1$
\begin{align}\label{Tim11}
\langle \mathcal{L}^\Omega_n[h], h\rangle_\Omega\leqslant &\frac{1}{2n}\int_0^1 \nu_\Omega(r) r^{1+2n} dr+\int_0^1 r^{-2n-1}\int_0^r s^{2n+1}\nu_\Omega(s)ds dr.
\end{align}
Therefore, applying Lemma \ref{prop-InvertT1} we deduce for any  $\theta\in(0,1]$
\begin{align*}
\int_0^1 \nu_\Omega(r) r^{1+2n} dr&\leqslant  \frac{C_0}{(\Omega-\kappa_2)^\theta} \int_0^1 \frac{r^{1+2n}}{(1-r)^{1-\theta}}  dr\leqslant   \frac{C_0}{\theta (\Omega-\kappa_2)^\theta}\cdot
\end{align*}
Similarly, we find  
\begin{align*}
\int_0^1 r^{-2n-1}\int_0^r s^{2n+1}\nu_\Omega(s)ds dr&\leqslant  \frac{C_0}{(\Omega-\kappa_2)^\theta}\int_0^1 r^{-2n-1}\int_0^r \frac{s^{2n+1}}{(1-s)^{1-\theta}}ds dr\\
&\leqslant  \frac{C_0}{(\Omega-\kappa_2)^\theta}\int_0^1 \frac{r^{-2n-1}}{(1-r)^{1-\theta}}\int_0^r {s^{2n+1}}ds dr\\
&\leqslant  \frac{C_0}{(\Omega-\kappa_2)^\theta(2n+2)}\int_0^1 \frac{1}{(1-r)^{1-\theta}}dr,
\end{align*}
implying
\begin{align*}
\int_0^1 r^{-2n-1}\int_0^r s^{2n+1}\nu_\Omega(s)ds dr&\leqslant   \frac{C_0}{n \theta (\Omega-\kappa_2)^\theta}\cdot
\end{align*}
Inserting these estimates into \eqref{Tim11} allows to get
\begin{align*}
\langle \mathcal{L}^\Omega_n[h], h\rangle_\Omega\leqslant &\frac{C_0}{n \theta (\Omega-\kappa_2)^\theta}\cdot
\end{align*}
Therefore under the assumption
$$
\frac{C_0}{n \theta (\Omega-\kappa_2)^\theta}<1,
$$
we deduce that $\| \mathcal{L}^\Omega_n\|_\Omega<1$ and then $\textnormal{Id}- \mathcal{L}^\Omega_n:L^2_\Omega\to L^2_\Omega $ is invertible.
This ends the proof.
\end{proof}

Let us remark that Proposition \ref{prop-Invert1}-(2) suggests that we need $n$ to be large enough  in order to invert  $\mbox{Id}-\mathcal{L}_n^\Omega$  for $\Omega\in(\kappa_2,\infty)$.

\subsection{Dispersion equation. }\label{Sec-Disper-eq}

Recall that our main target is to find a suitable subset in $n$ and $\Omega$ such that the kernel equation \eqref{kernel-eq} admits a nonzero smooth solution. As we shall see, this  will be characterized by the zeroes  of an analytical equation in $\Omega$ and $n$, called the  {\it dispersion equation}.  The connection will be achieved  through  the construction of suitable generators to  Sturm-Liouville differential equations, whose properties are deeply related to the structure of the profile $f_0$. We shall distinguish two regimes in the  solvability of the dispersion equation  depending on the location of $\Omega $ and the sign of the function $f_0$: \\

\hspace{0.1cm}\quad $\bullet$ \mbox{\it Scarcity of eigenvalues (defocusing case):} $f_0> 0$ {and} $\Omega\in(-\infty,\kappa_1)$. In this regime,  we shall see that the dispersion equation is only solved for lower modes $n$ and we have at most a finite number of solutions \mbox{in $\Omega$.}

\hspace{0.1cm}\quad $\bullet$ \mbox{\it Abundance of eigenvalues (focusing case):} $ f_0< 0$ {and} $ \Omega\in(\kappa_2,+\infty)$ In this regime, we prove  that the dispersion equation can be solved for any $n$ large enough and we have an infinite  number of solutions \mbox{in $\Omega$ accumulating at $\kappa_2.$}
\
\subsubsection{Sturm-Liouville equation} \label{Sec-STE}

Here, we will analyze some aspects of the following second order differential equation which will appear later in the dispersion equation:
\begin{align}\label{diff-FnML1}
F^{\prime\prime}+\tfrac{2n+1}{r}F^\prime+\sigma_\Omega\nu_\Omega(r) F=0,
\end{align}
where the potential $\nu_\Omega:[0,1]\to\R$ is defined via \eqref{mutrivial} and \eqref{nu-mu} and is always strictly positive. The behavior of this free Sturm-Liouville equation is intimately related to the sign $\sigma_\Omega$ acting  in front of the potential $\nu_\Omega.$   We shall say that the equation \eqref{diff-FnML1} is {\it defocusing} if $\sigma_\Omega=-1$, corresponding to $\Omega<\kappa_1$. However, this equation  is said {\it focusing} if $\sigma_\Omega=1$, corresponding to $\Omega>\kappa_2.$  For quite similar reasons, we are borrowing  this  terminology from  nonlinear Schr\"{o}dinger equation.   Our goal is to discuss some qualitative properties on the solutions to \eqref{diff-FnML1} that will be used later.
 \begin{lem}\label{lem-pot1}
Let $f_0$ satisfy \eqref{H2}-\eqref{H1}, $n\geqslant1$ and $\Omega\notin[\kappa_1,\kappa_2]$. There exist two functions $y_{\pm}:[0,1]\to\R$ of class $\mathscr{C}^2$ with $y_{\pm}(0)=1$ such that any solution to \eqref{diff-FnML1} in $(0,1)$ takes the form
\begin{align*}
\forall r\in(0,1),\quad F(r)=\alpha r^{-2n}y_{-}(r)+\beta y_+(r), \quad \alpha,\beta\in\R.
\end{align*}
In addition, the following assertions hold true.
\begin{enumerate}
\item
{\it Defocusing case}: For any $\Omega<\kappa_1$, 
 the function  $y_{+}$ is strictly increasing with 
$$
\forall r\in[0,1],\quad 1\leqslant y_{+}(r)\leqslant e^{\frac{1}{2n}\int_0^r s{\nu_\Omega}(s)ds}.
$$
\item {\it Focusing case}:  For any $\Omega>\kappa_2$, if 
 there exists   some $\theta\in(0,1]$ such that 
\begin{align}\label{pouisa}
\frac{C_0}{n \theta (\Omega-\kappa_2)^\theta}<1,
\end{align}
 then the function  $y_{+}$ is strictly decreasing  with 
\begin{align*}
\forall\, r\in[0,1],\quad \tfrac12\leqslant y_+(r)\leqslant 1.
\end{align*}
Moreover,
$$
\forall r\in [0,1],\quad |y_+(r)-1|\leqslant \frac{C_0}{\theta n(\Omega-\kappa_2)^\theta}\cdot
$$

\end{enumerate}
\end{lem}

\begin{proof}
The structure of the solutions is a consequence of general ODE  with regular singular points according to Fuchs theorem. The indicial equation is given by
$$
x(x-1)+(2n+1)x=0,
$$
which is equivalent to
$$
x^2+2n x=0.
$$
It admits two different solutions
$$
x_-=-2n,\quad  x_+=0,
$$
and hence, we know from Frobenius method that all the solutions are  in the form
$$
F(r)=\alpha r^{-2n}y_{-}(r)+\beta  y_{+}(r), \quad \alpha,\beta\in\R,
$$
with $y_{\pm}$ being  analytic on $[0,1]$ and  $y_{\pm}(0)=1$.

\medskip\noindent
{\bf{(1)}}  Let $\Omega<\kappa_1$ then by virtue of \eqref{nu-mu} the equation \eqref{diff-FnML1} writes
\begin{align}\label{diff-FnMLL1}
F^{\prime\prime}+\tfrac{2n+1}{r}F^\prime-\nu_\Omega(r) F=0,
\end{align}
 with $\nu_\Omega$ a positive continuous function. We shall show that $y_{+}$ is  strictly increasing. Integrating this  differential equation yields 
 \begin{align}\label{Tip1}
y_+^\prime(r)=\frac{1}{r^{2n+1}}\int_0^r s^{2n+1}\nu_\Omega (s) y_+(s) ds.
\end{align}
Let us check that $y_+$ is positive, which implies in turn  hat $y_+'$ is positive in view of \eqref{Tip1}. For that, define the set
$$
I:=\big\{r\in[0,1], \forall s\in[0,r],\quad y_+(s)>0\big\}.
$$
Note that $y_+(0)=1$, hence by construction and continuity of $y_+$ the set  $I$ is a nonempty open interval of $[0,1]$ taking the form $[0,\overline{r}]$. We shall show that $\overline{r}=1$. We argue by contradiction,  and assume that $\overline{r}<1$.  Consider an increasing  sequence $(r_m)$ of $I$ converging to $\overline{r}$. We write 
$$
y_+^\prime(\overline{r})=\lim_{m\to \infty}y_+^\prime(r_m)=\lim_{m\to \infty}\frac{1}{r_m^{2n+1}}\int_0^{r_m} s^{2n+1} \nu_\Omega (s) y_+(s) ds=\frac{1}{{\overline{r}}^{2n+1}}\int_0^{\overline{r}} s^{2n+1} \nu_\Omega (s) y_+(s) ds,
$$
with the property
$$
\forall s\in[0,\overline{r}), \quad y_+(s)>0.
$$
Then
$$
y_+^\prime(\overline{r})>0,
$$
 implying that
 $$
 y_+(\overline{r})>0.
 $$
Thus by continuity we can find $r_\star>\overline{r}$ such that 
$$\forall s\in[0,r_\star],\quad y_+({r_\star})>0,
$$
 which gives that $r_\star\in I$ and this  contradicts that $\overline{r}$ is maximal. Consequently, we get that 
 $$
 \forall r\in[0,1],\quad  y_+(r)>0.
 $$
 Combined with \eqref{Tip1} we deduce that $y_+$ is strictly increasing and then
 $$
  \forall r\in[0,1],\quad y_+(r)\geqslant y_+(0)=1.
 $$
 Integrating \eqref{Tip1} and using integration by parts 
 \begin{align}\label{Tippp1}
y_+(r)&=1+\int_0^r\frac{1}{s^{2n+1}}\int_0^s \tau^{2n+1} \nu_\Omega(\tau) y_+(\tau)d\tau ds\\
\nonumber&=1+\int_0^r\int_s^r\frac{1}{\tau^{2n+1}}d\tau s^{2n+1} \nu_\Omega (s) y_+(s)ds\\
\nonumber&=1+\frac{1}{2n}\int_0^r\big[1-(\tfrac{s}{r})^{2n}] s \nu_\Omega (s) y_+(s) ds.
\end{align}
In particular, we find 
\begin{align*}
1\leqslant y_+(r)&\leqslant 1+\frac{1}{2n}\int_0^r s \nu_\Omega (s) y_+(s) ds.
\end{align*}
Now, applying Gronwall inequality we find
\begin{align*}
\forall \, r\in[0,1],\quad 1\leqslant y_+(r)\leqslant e^{\frac{1}{2n}\int_0^r s \nu_\Omega (s) ds}.
\end{align*}

\medskip\noindent
{\bf{(2)}} Since $\mu_\Omega^0\ge 0$ and  similarly to \eqref{Tippp1} we may write 
\begin{align*}
y_+(r)
&=1-\frac{1}{2n}\int_0^r\big[1-(\tfrac{s}{r})^{2n}] s \nu_\Omega (s) y_+(s) ds.
\end{align*}
Set
\begin{align}\label{z+}
y_+=1-z_+,
\end{align}
then 
\begin{align}\label{Tip001}
z_+(r)
&=\frac{1}{2n}\int_0^r\big[1-(\tfrac{s}{r})^{2n}] s \nu_{\Omega}(s)  ds-\frac{1}{2n}\int_0^r\big[1-(\tfrac{s}{r})^{2n}] s\nu_{\Omega} z_+(s) ds\\
\nonumber&:=f_n(r)-\mathcal{T}_n[z_+](r).
\end{align}

Applying Lemma \ref{prop-InvertT1} we infer
\begin{align*} 0\leqslant  f_n(r)
&\leqslant \frac{C_0}{(\Omega-\kappa_2)^\theta}\frac{1}{2n}\int_0^r\frac{ds}{(1-s)^{1-\theta}}\leqslant \frac{C_0}{2\theta n(\Omega-\kappa_2)^\theta}\cdot
\end{align*}
Similarly we get
\begin{align*} |\mathcal{T}_n[z_+](r)|
&\leqslant \frac{C_0}{2n\theta(\Omega-\kappa_2)^\theta}\|z_+\|_{L^\infty}\leqslant \frac{C_0}{2\theta n(\Omega-\kappa_2)^\theta}\|z_+\|_{L^\infty}.
\end{align*}
Therefore, if we assume 
$$
\frac{C_0}{2\theta n(\Omega-\kappa_2)^\theta}\leqslant \frac14,
$$
then $\mathcal{T}_n:L^\infty([0,1])\to L^\infty([0,1])$ is a contraction  and therefore we have only one solution $z_+$ to the equation  \eqref{Tip001} satisfying
$$
\forall r\in [0,1],\quad |z_+(r)|\leqslant \frac{2C_0}{3\theta n(\Omega-\kappa_2)^\theta}\leqslant \frac13.
$$
Coming back to \eqref{z+} we deduce that
\begin{align}\label{RS1}
\frac23\leqslant y_+(r)\leqslant \frac43.
\end{align}
Now, similarly to \eqref{Tip1} we have
\begin{align*}
y_+^\prime(r)=-\frac{1}{r^{2n+1}}\int_0^r s^{2n+1}\nu_{\Omega}(s)(s) y_+(s) ds.
\end{align*}
Hence by \eqref{RS1} we get that $y_+$  is decreasing on $[0,1]$,  and since  $y_+(0)=1,$ one has 
\begin{align*}
\forall\, r\in[0,1],\quad \frac23\leqslant y_+(r)\leqslant 1.
\end{align*}
This completes the proof of the lemma.
\end{proof}

\subsubsection{Dispersion relation}
Having introduced in the previous section the free {\it focusing}-{\it defocusing} Sturm-Liouville equation \eqref{diff-FnML1},   our goal here is  to build a  bridge with the  dispersion equation that will characterize analytically the kernel equation at any  level $n$ as  introduced in \eqref{kernel-eq}. We are able to unify with the same formalism the {\it defocusing} and {\it focusing} regimes, corresponding to $\Omega<\kappa_1$ and $\Omega>\kappa_2$, respectively.\\
Let $\Omega\not\in[\kappa_1,\kappa_2]$ and  $F_{n,\Omega}$ be the unique  solution of class $\mathscr{C}^2$ on $[0,1]$  to the normalized free Sturm-Liouville equation 
introduced in \eqref{diff-FnML1}, that is, 
\begin{align}\label{diff-Fn}
F^{\prime\prime}_{n,\Omega}+\tfrac{2n+1}{r}F^\prime_{n,\Omega}+\sigma_\Omega \nu_{\Omega} F_{n,\Omega}=0,\quad F_{n,\Omega}(0)=1.
\end{align}
where the potential $\nu_\Omega$ was previously defined  in \eqref{nu-mu}. 
The existence and uniqueness  is guaranteed by Lemma \ref{lem-pot1} where $y_+=F_{n,\Omega}$.
Now, define the following function subject to the strong effect of the profile $f_0$
\begin{align}\label{zeta-eq}
\zeta_n(\Omega):=F_{n,\Omega}(1)\left(\Omega-\frac{n}{n+1}\int_0^1 sf_0(s)ds\right)+\int_0^1F_{n,\Omega}(s) s^{2n+1}\big(f_0(s)-2\Omega\big)ds.
\end{align}
The relationship  between the  zeroes of this function (dispersion equation) and the kernel equation \eqref{kernel-eq} is elucidated  in the following crucial  result.
\begin{lem}\label{Lem-disper}
Let $n\geqslant 1$ and  assume that $\Omega\notin[\kappa_1,\kappa_2],$ supplemented with the condition \eqref{pouisa} in the {\it focusing} case.
Then 
the equation \eqref{kernel-eq} admits a nontrivial continuous  solution on $[0,1]$  if and only if $\Omega$ satisfies the dispersion equation
$$\zeta_n(\Omega)=0.$$ 
\end{lem}

\begin{proof}
By the kernel equation \eqref{kernel-eq}-\eqref{nu-mu}, we find the following equivalent condition
\begin{equation}\label{kernel-eqSS}
h_n(r)=-\frac{\sigma_\Omega H_n[h_n](1)}{2nG_n(1)}(\textnormal{Id}-\sigma_\Omega \mathcal{L}^\Omega_n)^{-1}[\nu_\Omega rG_n],
\end{equation} 
Observe that the existence of $(\textnormal{Id}-\sigma_\Omega \mathcal{L}^\Omega_n)^{-1}$ follows from Proposition \eqref{prop-Invert1}.
It follows that the kernel is a vectorial subspace with  at most one dimension, and it is of dimension one if and only if the equation  \eqref{kernel-eqSS} admits at least one solution with $H_n[h_n](1)\neq0$. Now, under this assumption and according  to the definition of $H_n[h](1)$, that follows from \eqref{Hn-def},  we get the  \begin{equation*}
\left[1+\frac{\sigma_\Omega}{2n G_n(1)}\int_0^1 s^{n+1} (\textnormal{Id}-\sigma_\Omega \mathcal{L}^\Omega_n)^{-1}[\nu_\Omega rG_n](s) ds\right]H_n[h_n](1)=0,
\end{equation*}
which implies that
\begin{equation}\label{spectral-condition-1}
1+\frac{\sigma_\Omega}{2n G_n(1)}\int_0^1 s^{n+1} (\textnormal{Id}-\sigma_\Omega \mathcal{L}^\Omega_n)^{-1}[\nu_\Omega rG_n](s) ds=0.
\end{equation}
Consequently, the kernel equation generates a subspace of dimension one, with a generator $\mathtt{h}_n:=(\textnormal{Id}-\sigma_\Omega \mathcal{L}^\Omega_n)^{-1}[\nu_\Omega rG_n]$, if and only if \eqref{spectral-condition-1} is satisfied together with the condition
$$
\int_0^1 s^{n+1}\mathtt{h}_n(s)ds\neq0.
$$
But this latter condition is automatically satisfied in view of \eqref{spectral-condition-1}.
Let us introduce the following real-valued function
\begin{equation}\label{TnOmega}
T(n,\Omega):=\frac{-\sigma_\Omega}{2n G_n(1)}\int_0^1 s^{n+1} (\textnormal{Id}-\sigma_\Omega \mathcal{L}^\Omega_n)^{-1}[\nu_\Omega rG_n](s) ds.
\end{equation}
Using that $\mathcal{L}_{n}^{\Omega}$ is self-adjoint as stated in Proposition \ref{pro-positivity}, we can write $T({n,\Omega})$ as follows
\begin{align*}
T(n,\Omega)
=&\frac{-\sigma_\Omega}{2n G_n^{\Omega}(1)}\int_0^1\big([(\textnormal{Id}-\sigma_\Omega\mathcal{L}^\Omega_n)^{-1}\big)[s^{n}\nu_{\Omega}](s) G_n^\Omega(s) s^{2}ds.
\end{align*}
Therefore, by setting  
\begin{align}\label{Eq-t1}
h:=-\sigma_\Omega(\textnormal{Id}-\sigma_\Omega\mathcal{L}^\Omega_n)^{-1}[s^{n}\nu_{\Omega}],\quad\hbox{and}\quad F:=\frac{h}{r^n\nu_{\Omega}},
\end{align}
we infer
\begin{align}\label{Tn-tout}
T(n,\Omega)
=&\frac{1}{2n G_n^{\Omega}(1)}\int_0^1 h(s) G_n^\Omega(s) s^{2} ds.
\end{align}
By \eqref{Gn-def} and \eqref{Op-L} we get
\begin{align}\label{Eq-t2}
-\sigma_\Omega\frac{h}{r^n \nu_{\Omega}}+\frac{1}{2n}\int_r^1\frac{1}{s^{n-1}}h(s)ds+\frac{1}{2n r^{2n}}\int_0^rs^{n+1}h(s)ds=1,
\end{align}
leading to
\begin{align}\label{Eq-tm2}
-\sigma_\Omega F+\frac{1}{2n}\int_r^1F(s)\,s\nu_\Omega(s)ds+\frac{1}{2n r^{2n}}\int_0^rF(s)\,s^{2n+1}\nu_\Omega(s)ds=1.
\end{align}
Differentiating twice this equation yields to  the Sturm-Liouville equation \eqref{diff-FnML1}
$$
F''+\tfrac{2n+1}{r}F'+\sigma_\Omega\nu_{\Omega} F=0.
$$
By virtue of  Lemma \ref{lem-pot1}, we know that  all the solutions  take  the form
$$
F(r)=\alpha r^{-2n}y_{-}(r)+\beta  y_{+}(r), \quad \alpha,\beta\in\R,
$$
with $y_{\pm}$ of class $\mathscr{C}^2$ in $[0,1]$ and $y_{\pm}(0)=1$. Therefore, we deduce according to the second point of \eqref{Eq-t1},
$$
h(r)=\nu_\Omega(r)\Big(\alpha r^{-n}y_{-}(r)+\beta  r^{n}y_{+}(r)\Big).
$$
From this general structure, we see easily  that the function associated to $\alpha $ is  singular at zero and does not belong to the space $L^2_\Omega$ related to the norm \eqref{Norm-L2},   implying that necessarily 
\begin{equation}\label{h-exp}
h(r)=\beta \nu_\Omega(r)\ r^{n}y_{+}(r),
\end{equation}
giving in view of \eqref{Eq-t1}
$$
F(r)= \beta y_{+}(r),
$$
with  $y_{+}(0)=1$. Notice that the function $y_+$ coincides with $F_{n,\Omega}$ introduced in \eqref{diff-Fn}.  It follows from this latter fact, together with  \eqref{Tn-tout} and \eqref{Eq-t1}, that
\begin{align}\label{Tipp3}
T(n,\Omega)
=&\frac{\beta}{2n G_n^\Omega(1)}\int_0^1\big[G_n^\Omega(r)r^{-n+1}\big]\big[ F_{n,\Omega}(r)   \nu_\Omega(r) r^{2n+1}\big] dr.
\end{align}
In addition, coming  back to \eqref{Eq-t2} we may write
\begin{align}\label{lower-1}
\forall r\in[0,1],\quad -\sigma_\Omega\beta F_{n,\Omega}(r)+\frac{\beta}{2n}\int_r^1F_{n,\Omega}(s) s\nu_\Omega(s)ds+\frac{\beta}{2n r^{2n}}\int_0^rF_{n,\Omega}(s)s^{2n+1}\nu_\Omega(s)ds=1.
\end{align}
Now we can write the equation \eqref{diff-FnML1} in the form 
\begin{align}\label{diff-Fn1}
\left(r^{2n+1}F^{\prime}_{n,\Omega}\right)^{\prime}=-\sigma_\Omega F_{n,\Omega}(r)   \nu_\Omega(r) r^{2n+1}.
\end{align}
Then, two integration  by parts yield
\begin{align}\label{TippN3}
-\sigma_\Omega T(n,\Omega)
=&\frac{\beta}{2n G_n^\Omega(1)}\Big(G_n^\Omega(1)F_{n,\Omega}^\prime(1)-F_{n,\Omega}(1)\mathcal{H}_n^\prime(1)\Big)\\
\nonumber&+\frac{\beta}{2n G_n^\Omega(1)}\int_0^1\big[r^{2n+1}\mathcal{H}_n^\prime(r)\big]^\prime  F_{n,\Omega}(r)  dr,
\end{align}
with
$$
\mathcal{H}_n(r):=G_n^\Omega(r)r^{-n+1}.
$$
Applying  \eqref{Gn-def} implies 
\begin{align*}
 \mathcal{H}_n(r)&= n\Omega r^{2}+\int_0^1sf_0(s)ds-(n+1)\int_0^rsf_0(s)ds
+\frac{n+1}{r^{2n}}\int_0^rs^{2n+1}f_0(s)ds.
\end{align*}
Differentiating this identity gives after a cancellation 
\begin{align}\label{Gn-defX1}
 \mathcal{H}^\prime_n(r)&= 2n\Omega r
-2n\frac{n+1}{r^{2n+1}}\int_0^rs^{2n+1}f_0(s)ds,
\end{align}
leading in turn to
\begin{align}\label{Gn-defX2}
\mathcal{H}^\prime_n(1)= 2n\left(\Omega 
-(n+1)\int_0^1r^{2n+1}f_0(r)dr\right),
\end{align}
and
\begin{align}\label{Gn-defX3}
\left(r^{2n+1} \mathcal{H}^\prime_n(r)\right)^\prime&= 2n(n+1)r^{2n+1}\big(2\Omega 
-f_0(r)\big).
\end{align}
On the other hand, integrating \eqref{diff-Fn1} we find
\begin{align*}
F^{\prime}_{n,\Omega}(1)=-\sigma_\Omega \int_0^1F_{n,\Omega}(r)   \nu_\Omega(r) r^{2n+1}dr.
\end{align*}
Thus, combining this identity with  \eqref{lower-1} tested with $r=1$  we infer
\begin{align}\label{lower-R1}
\tfrac{\beta}{2n} F^{\prime}_{n,\Omega}(1)=-\sigma_\Omega-\beta F_{n,\Omega}(1).
\end{align}
Putting together \eqref{TippN3}, \eqref{Gn-defX3} and \eqref{lower-R1}, {we find that $T(n,\Omega)=1$ is equivalent to}
$$
F_{n,\Omega}(1)\big(G_n(1)+\tfrac{1}{2n}\mathcal{H}_n^\prime(1)\big)+(n+1)\int_0^1 F_{n,\Omega}(r)r^{2n+1}\big(f_0(r)-2\Omega\big)dr=0.
$$
Moreover, using \eqref{Gn1-1} and  \eqref{Gn-defX2}, we get
$$
G_n(1)+\tfrac{1}{2n}\mathcal{H}_n^\prime(1)=(n+1)\Omega-n\int_0^1 rf_0(r)dr.
$$
Putting together the last two identities we infer
$$
F_{n,\Omega}(1)\left(\Omega-\frac{n}{n+1}\int_0^1 rf_0(r)dr\right)+\int_0^1 F_{n,\Omega}(r)r^{2n+1}\big(f_0(r)-2\Omega\big)dr=0.
$$
This achieves the proof of the lemma.
\end{proof}

\subsubsection{ Scarcity  of the zeroes (defocusing case)}
In this subsection we shall assume that $f_0$ is a positive function satisfying \eqref{H2}-\eqref{H1}. We shall see in the Proposition \ref{prop-omegan-low} below that the dispersion equation stated in Lemma \ref{Lem-disper} admits for lower symmetry $n$ real solutions  $\Omega$ that belong to the interval $(-\infty,\kappa_1)$. The scarcity of the eigenvalues can be formally interpreted as follows. When $f_0>0$, the sequence  $(\widehat\Omega_n)_n$ defined in \eqref{Interv1} is strictly increasing and converges to $\kappa_2$.  Therefore, this sequence will intersect the domain $\R\backslash[\kappa_1,\kappa_2]$ at a finite set embedded in $(-\infty,\kappa_1)$, and the solutions that we are able to construct are paired with  these finite  singular points. More precisely, we get  the following result.
\begin{pro}\label{prop-omegan-low}
	 Let $f_0>0$ 
 satisfy \eqref{H2}-\eqref{H1} and  consider $n,m\geqslant 1$. Then, the following assertions hold true.
\begin{enumerate}
\item The sequence $(\widehat{\Omega}_n)_{n\geqslant1}$ defined in \eqref{Interv1} is strictly increasing and converges to $\kappa_2.$
\item The function  $\zeta_n$ satisfies
$$
\forall\, \Omega\leqslant\min(\kappa_1,\widehat{\Omega}_n),\quad  \zeta_n\big(\Omega\big)<0.
$$

\item If  $n\leqslant\frac{\kappa_1}{\kappa_2-\kappa_1}$, then $\widehat{\Omega}_n<\tfrac{n\kappa_2}{n+1}\leqslant \kappa_1$ and 
$$
 \zeta_n\big(\tfrac{n\kappa_2}{n+1}\big)>0. 
$$
\item  If  $3\leqslant m\leqslant{\frac{f_0(0)}{f_0(1)-f_0(0)}}$ then $\zeta_m(\Omega)=0$ admits  at least one  solution $\Omega_m\in(\widehat{\Omega}_m, \tfrac{m\kappa_2}{m+1}).$ In addition,
$$
\forall\,n\geqslant2,\quad \zeta_{nm}(\Omega_m)<0.
$$

\end{enumerate}
\end{pro}
\begin{proof}
\medskip \noindent{\bf (1)}  According to \eqref{Interv1} and the definition \eqref{kappa2} we get
$$
\widehat{\Omega}_n=\kappa_2-\frac{n+1}{n}\int_0^1s^{2n+1}f_0(s)ds.
$$
Then the monotonicity of this sequence follows from the positivity of the profile $f_0.$  The convergence to $\kappa_2$ is easy to check.

\medskip \noindent{\bf (2)} Since $\Omega\leqslant \kappa_1=\frac{f_0(0)}{2}$ then we get by the monotonicity of $f$
\begin{equation}\label{bound-f0-omega}
f_0(r)-2\Omega\geqslant f_0(r)-f_0(0)\geqslant0.
\end{equation}
Combined with the fact that  $F_{n,\Omega}$ is   increasing, which  follows from Lemma \ref{lem-pot1}-(1), it yields
\begin{align*}
\int_0^1 F_{n,\Omega}(r)r^{2n+1}\big(f_0(r)-2\Omega\big)dr&< F_{n,\Omega}(1)\int_0^1r^{2n+1}\big(f_0(r)-2\Omega\big)dr\\
&=  F_{n,\Omega}(1)\left(\int_0^1r^{2n+1}f_0(r)dr-\frac{\Omega}{n+1}\right).
\end{align*}
Thus, we get from \eqref{zeta-eq} and  the definition of $\widehat{\Omega}_n$ in \eqref{Interv1} the following inequality
\begin{align*}
\zeta_n(\Omega)&< F_{n,\Omega}(1)\left(\frac{n}{n+1}\Omega-\frac{n}{n+1}\int_0^1 rf_0(r)dr+\int_0^1 r^{2n+1}f_0(r)dr\right)\\
&= F_{n,\Omega}(1)\frac{n}{n+1}\left(\Omega-\widehat\Omega_n\right).
\end{align*}
This implies in particular that 
$$
\forall \Omega\leqslant\widehat\Omega_n, \quad \zeta_n(\Omega)<0.
$$

\medskip\noindent {\bf (3)} First, note that the assumption $n\leqslant\frac{\kappa_1}{\kappa_2-\kappa_1}$ is equivalent to $\tfrac{n\kappa_2}{m+1}=\frac{n}{n+1}\int_0^1 sf_0(s)ds\leqslant\kappa_1$. It  implies from \eqref{zeta-eq} that in the particular case  $\Omega=\tfrac{n\kappa_2}{n+1}$
$$
 \zeta_n\big(\tfrac{n\kappa_2}{n+1}\big)=\int_0^1 F_{n,\Omega}(r)r^{2n+1}\big(f_0(r)-2\tfrac{n\kappa_2}{n+1}\big)dr>0,
$$
by using \eqref{bound-f0-omega}.  To prove $
\widehat\Omega_n<\tfrac{n\kappa_2}{n+1},
$ it can be done  through direct computations based on the definition of $\widehat{\Omega}_n$ in \eqref{Interv1}, or by simply evoking  the second point in Proposition \ref{prop-omegan-low} together with the estimate $\zeta_n\big(\tfrac{n\kappa_2}{n+1}\big)>0$

\medskip\noindent {\bf (4)} Take  $3\leqslant m\leqslant{\frac{f_0(0)}{f_0(1)-f_0(0)}}$, which implies by the monotonicity  that $m\leqslant{\frac{f_0(0)}{f_0(1)-f_0(0)}}\leqslant\frac{\kappa_1}{\kappa_2-\kappa_1}.$ Then combining the points $(2)$ and $(3)$ with the intermediate value theorem, we deduce the existence of a real number  $\Omega_m\in (\widehat\Omega_m,\tfrac{m\kappa_2}{m+1})$ such that
$$
\zeta_m(\Omega_m)=0.
$$
 Since  $m\leqslant\frac{\kappa_1}{\kappa_2-\kappa_1}$, implying that   $\frac{m}{m+1}\int_0^1 sf_0(s)ds\leqslant\kappa_1$, then 
 \begin{align*}\widehat{\Omega}_{2m}&=\int_0^1sf_0(s)ds-\frac{2m+1}{2m}\int_0^1s^{4m+1}f_0(s)ds\\
 &\geqslant \int_0^1sf_0(s)ds-\frac{1}{4m}f_0(1)\\
 &\geqslant \frac{m}{m+1}\int_0^1 sf_0(s)ds.
 \end{align*}
 Thus to get the inequality
 \begin{align*}\int_0^1sf_0(s)ds-\frac{1}{4m}f_0(1)
 &\geqslant \frac{m}{m+1}\int_0^1 sf_0(s)ds\Longleftrightarrow \int_0^1sf_0(s)ds\geqslant \frac{m+1}{4m}f_0(1). 
 \end{align*}
This latter inequality is satisfied provided that
\begin{align}\label{pop-R1}
f_0(0)\geqslant \frac{m+1}{2m}f_0(1).
\end{align}
From our assumption on $m$ ($3\leqslant m\leqslant{\frac{f_0(0)}{f_0(1)-f_0(0)}}$) and the positivity of $f_0$ we get
\begin{align*}
f_0(0)&\geqslant \frac{m}{m+1}f_0(1)\geqslant \frac{m+1}{2m}f_0(1),
\end{align*}
 which implies \eqref{pop-R1}. Consequently,
 \begin{align*}\widehat{\Omega}_{2m}
 &\geqslant \frac{m}{m+1}\int_0^1 sf_0(s)ds>\Omega_m,
 \end{align*}
 leading in view of the point $(2)$ to
 $$
 \zeta_{2m}(\Omega_m)<0.
 $$
By the monotonicity of the sequence $(\widehat{\Omega}_{n})_{n\geqslant1}$ established in the first point $(1)$, we find
$$
\forall\, n\geqslant 2,\quad \widehat{\Omega}_{nm}>\widehat{\Omega}_{2m}>\Omega_m.
$$
Thus by the second point $(2)$ we deduce that
$$
\forall\, n\geqslant 2,\quad \zeta_{nm}(\Omega_m)<0,
$$
and this concludes  the proof.
 \end{proof}
 
\subsubsection{ Abundance of the zeroes  (focusing case)}
In this  subsection, the profile   $f_0$ is negative and satisfies \eqref{H2}-\eqref{H1}. We intend to prove that  in this setting the dispersion equation admits  infinitely many solutions for large symmetry and the eigenvalues are located in the region $(\kappa_2,\infty),$ corresponding to the {\it defocusing} regime.
 As in the scarcity case analyzed before, the abundance of the  eigenvalues can be formally interpreted as follows. When $f_0<0$, the singular sequence  $(\widehat\Omega_n)_n$ defined in \eqref{Interv1} is strictly decreasing and converges to $\kappa_2$.  Therefore, this sequence will intersect the domain $\R\backslash[\kappa_1,\kappa_2]$ at a finite set embedded in $(\kappa_2,\infty)$, and the solutions that we are able to construct are paired with  the  singular points located in $(\kappa_2,\infty)$. Actually, we prove  the following result.

\begin{pro}\label{prop-omegam-asymp}
Let $m\in\N^*, \alpha\in(1,2)$ and $f_0<0$  satisfying \eqref{H2}-\eqref{H1}.
There exists $m_0$ large enough depending on $f_0$ such that  the following assertions hold true. For any $m\geqslant m_0$, 
\begin{enumerate}
\item There exists  $\Omega_m\in(\widehat\Omega_m-{m^{-\alpha}},\widehat\Omega_m)$ such that
$$
\zeta_m(\Omega_m)=0.
$$
\item For any $n\geqslant1$ we have
$$
\Omega_m\neq \widehat\Omega_{nm}.
$$
\item For any  $n\geqslant2$
$$
\zeta_{nm}(\Omega_m)\neq 0.
$$
\end{enumerate}
\end{pro}
\begin{proof}
\medskip\noindent
{\bf (1)} Let us decompose $F_{n,\Omega}$, the solution to \eqref{diff-Fn},  as follows
$$
F_{n,\Omega}(r)=:F_{n,\Omega}(1)+\rho_{n,\Omega}(r).
$$
Then, straightforward computations based on  \eqref{zeta-eq} yield
\begin{align}\label{zeta-eqMS1}
\nonumber\zeta_n(\Omega)&=\frac{n}{n+1}F_{n,\Omega}(1)\left(\Omega-\widehat\Omega_n\right)+\int_0^1\rho_{n,\Omega}(s) s^{2n+1}\big(f_0(s)-2\Omega\big)ds\\
&=:\zeta_{n,1}(\Omega)+\zeta_{n,2}(\Omega),
\end{align}
where $\zeta_{n,2}$ denotes the last integral term. The first step is to show that
\begin{align}\label{zeta-eqMSS1}
\zeta_n(\widehat\Omega_n)>0.
\end{align}
One has easily from \eqref{zeta-eqMS1} that
$$
\zeta_n(\widehat\Omega_n)=\int_0^1\rho_{n,\widehat\Omega_n}(s) s^{2n+1}\big(f_0(s)-2\widehat\Omega_n\big)ds.
$$
Next,we shall establish the following bounds: there exists $C_0>0$ such that for any $\theta\in(0,1)$
\begin{align}\label{L-U-bound}
{ \frac{C_0^{-1}(1-r)}{n(\Omega-\kappa_1)}\leqslant \rho_{n,\Omega}(r)\leqslant \frac{C_0(1-r)^\theta}{\theta\, n(\Omega-\kappa_2)^\theta}\cdot}
\end{align}
For that, we implement first Taylor formula with $\rho_{n,\Omega}(1)=0$ 
\begin{align*}
\rho_{n,\Omega}(r)&=-\int_{r}^1\rho_{n,\Omega}^\prime(s)ds=-\int_{r}^1F_{n,\Omega}^\prime(s)ds.
\end{align*}
Similarly to \eqref{Tip1} we find
\begin{align*}
F_{n,\Omega}^\prime(r)=-\frac{1}{r^{2n+1}}\int_0^r s^{2n+1}\nu_{\Omega} (s)F_{n,\Omega} (s) ds,
\end{align*}
amounting to
\begin{align*}
\rho_{n,\Omega}(r)=\int_{r}^1\frac{1}{s^{2n+1}}\int_0^s \tau^{2n+1}\nu_{\Omega} (\tau)F_{n,\Omega} (\tau) d\tau.
\end{align*}
Combined with Lemma \ref{prop-InvertT1} and Lemma \ref{lem-pot1}-(2), we find
  that for any $\theta\in (0,1]$
 \begin{align*}
0\leqslant \rho_{n,\Omega}(r)&\leqslant \frac{C_0}{(\Omega-\kappa_2)^\theta}\int_{r}^1\frac{1}{s^{2n+1}}\int_0^s \frac{\tau^{2n+1}}{(1-\tau)^{1-\theta}}d\tau\\
&\leqslant \frac{C_0}{(\Omega-\kappa_2)^\theta}\int_{r}^1\frac{1}{s^{2n+1}}\int_0^s \frac{\tau^{2n+1}}{(1-s)^{1-\theta}}d\tau\\
&\leqslant \frac{C_0(1-r)^\theta}{\theta\, n(\Omega-\kappa_2)^\theta}\cdot
\end{align*}
For the lower bound we use once again Lemma \ref{prop-InvertT1} and Lemma \ref{lem-pot1} together with the estimate below, which follows from  \eqref{H2}-\eqref{H1},
$$
\forall r\in(0,1],\quad \frac{f_0^\prime(r)}{r}\geqslant C_1>0
$$ 
for some $C_1>0$.  Therefore,  we obtain
\begin{align*}
 \rho_{n,\Omega}(r)&\geqslant \frac{1}{2(\Omega-\kappa_1)}\int_{r}^1\frac{1}{s^{2n+1}}\int_0^s {\tau^{2n+1}}\frac{f_0^\prime(\tau)}{\tau}d\tau ds\\
&\geqslant \frac{C_1(1-r)}{8(n+1)(\Omega-\kappa_1)}
\end{align*}
Consequently, by taking $C_0$ large enough we get 
$$
\frac{C_1}{8(n+1)}\geqslant {C_0^{-1}}{n},
$$
which achieves the proof of \eqref{L-U-bound}. Now from this inequality we deduce that
\begin{align}\label{L-U-bound1}
\frac{C_0^{-1}(1-r)}{(\widehat\Omega_n-\kappa_1)n}\leqslant \rho_{n,\widehat\Omega_n}(r)\leqslant \frac{C_0(1-r)^\theta}{\theta\, n(\widehat\Omega_n-\kappa_2)^\theta}\cdot
\end{align}
By \eqref{Interv1}, since $f_0(1)<0$ one gets the asymptotic
\begin{align}\label{omega-a}
\widehat\Omega_n=\kappa_2-\frac{f_0(1)}{2n}+O(n^{-2}),
\end{align}
and we find by \eqref{L-U-bound1} together with the fact that $\widehat\Omega_n-\kappa_1>0,\widehat\Omega_n-\kappa_2>0$,  that   by taking $C_0$ large enough, we  get
\begin{align}\label{L-U-bound2}
\forall \, n\geqslant 1, \forall r\in(0,1),\quad \frac{C_0^{-1}(1-r)}{n}\leqslant \rho_{n,\widehat\Omega_n}(r)\leqslant \frac{C_0(1-r)^\theta}{\theta\, n^{1-\theta}},
\end{align}
and 
\begin{align*}
f_0(s)-2\widehat\Omega_n=f_0(s)-2\kappa_2+O(n^{-1}):=f_1(s)+O(n^{-1}).
\end{align*}
Remark that by the monotonicity of $f_0$
$$
f_1(0)<0, \quad\textnormal{and}\quad f_1(1)>0.
$$
Then by continuity  and the uniform convergence, we can find $\delta,\varepsilon \in(0,1)$ independent of $n$ such that 
\begin{align}\label{BB-11}
\forall \, s\in[\delta,1],\quad f_0(s)-2\widehat\Omega_n\geqslant\varepsilon.
\end{align}
Making the splitting
\begin{align}\label{Arta1}
\nonumber \zeta_n(\widehat\Omega_n)&=\int_0^\delta\rho_{n,\widehat\Omega_n}(s) s^{2n+1}\big(f_0(s)-2\widehat\Omega_n\big)ds+\int_\delta^1\rho_{n,\widehat\Omega_n}(s) s^{2n+1}\big(f_0(s)-2\widehat\Omega_n\big)ds\\
&:=I_{n,1}+I_{n,2}.
\end{align}
Applying  \eqref{L-U-bound2} with $\theta=1$
\begin{align}\label{Arta2}
|I_{n,1}|&\leqslant { {C_0 \delta^{2n+2}}.}
\end{align}
For $I_{n,2}$ we use \eqref{L-U-bound2} and \eqref{BB-11}  leading for large $n$ to 
\begin{align}\label{Arta3}
 \nonumber I_{n,2}&\geqslant  \frac{C_0^{-1} \varepsilon}{n} \int_\delta^1(1-s){s^{2n+1}}ds\\
&\geqslant \frac{C_0^{-1} \varepsilon}{n}\left( \int_0^1(1-s){s^{2n+1}}ds+O(\delta^{2n+1})\right).
\end{align}
On the other hand,  we find
\begin{align*}
 \int_0^1(1-s){s^{2n+1}}ds&=\frac{1}{(2n+2)(2n+3)}\cdot
\end{align*}
 Inserting this inequality into \eqref{Arta3} yields 
\begin{align}\label{Arta4}
 I_{n,2}&\geqslant  \frac{C_0^{-1} \varepsilon}{n^{3}}\left(1+O(\delta^{2n+1})\right).
\end{align}
Putting together \eqref{Arta1}, \eqref{Arta2} and  \eqref{Arta4}   yields for large $n$
\begin{align*}
\zeta_n(\widehat\Omega_n)&\geqslant \frac{C_0^{-1}\varepsilon }{n^{3}}- {{C_0 \delta^{2n+1}}}\geqslant \frac{C_1(\varepsilon,\delta) }{n^{3}},
\end{align*}
for some constant $C_1(\varepsilon,\delta)>0.$ This ensures \eqref{zeta-eqMSS1}. \\
Coming back to \eqref{zeta-eqMS1} and using \eqref{L-U-bound} with $\theta=1$ allows to get
\begin{align*}
|\zeta_{n,2}(\Omega)|&\leqslant \frac{C_0}{n(\Omega-\kappa_2)}\int_0^1(1-s) s^{2n+1}ds\leqslant \frac{C_0}{ n^3(\Omega-\kappa_2)}.
\end{align*}
Let $\alpha>1$ and take $\Omega\geqslant \widehat\Omega_n-n^{-\alpha}$, then  using  \eqref{omega-a} we deduce  for large $n$
\begin{align}\label{zeta-eqMP4}
|\zeta_{n,2}(\Omega)|&\leqslant \frac{C_0}{n^2}\cdot
\end{align}
On the other hand, by virtue of \eqref{zeta-eqMS1} and Lemma \ref{lem-pot1}, we get for large $n$
\begin{align}\label{zeta-eqMDD}
\zeta_{n,1}( \widehat\Omega_n-n^{-\alpha})&\leq -\tfrac12 n^{-\alpha}\frac{n}{n+1} \leqslant -\tfrac13 n^{-\alpha},
\end{align}
and putting together \eqref{zeta-eqMP4} and \eqref{zeta-eqMDD} we obtain that for $\alpha\in(1,2)$ and $n$ large enough
\begin{align*}
\zeta_{n}( \widehat\Omega_n-n^{-\alpha})&\leqslant -\tfrac14 n^{-\alpha}.
\end{align*}
It follows that for $n$ large enough
$$
\zeta_{n}( \widehat\Omega_n-n^{-\alpha})<0,
$$
which gives the statement by applying the intermediate value theorem showing the existence of at least one solution 
\begin{align}\label{encadre}
\Omega_m\in(\widehat\Omega_m-m^{-\alpha},\widehat\Omega_m), \quad  \zeta_m(\Omega_m)=0.
\end{align}

\medskip\noindent
{\bf(2)} From \eqref{Interv1} and the fact that $f_0$ is negative we infer that $m\mapsto\widehat\Omega_m$ is strictly decreasing. Thus by using \eqref{omega-a} we deduce that
$$
\forall n\geqslant 2,\quad |\widehat\Omega_{nm}-\widehat\Omega_m|\geqslant |\widehat\Omega_{2m}-\widehat\Omega_m|\geqslant c\, m^{-1},
$$
for some constant $c>0$. On the other hand, using \eqref{encadre}, with $\alpha=\frac32$ we get
$$
|\Omega_{m}-\widehat\Omega_m|\leqslant c m^{-\frac32}.
$$
Therefore for large $m$, we deduce that
\begin{align}\label{lowe-SS}
\forall n\geqslant2,\quad |\widehat\Omega_{nm}- \Omega_{m}|\geqslant cm^{-1}.
\end{align}

\medskip\noindent
{\bf(3)} From \eqref{zeta-eqMS1} we find
 \begin{align}\label{zeta-eqMSL1}
\nonumber\zeta_{nm}(\Omega_m)&=\frac{nm}{nm+1}F_{nm,\Omega}(1)\left(\Omega_m-\widehat\Omega_{nm}\right)+\int_0^1\rho_{nm,\Omega_m}(s) s^{2nm+1}\big(f_0(s)-2\Omega_m\big)ds\\
&=\zeta_{nm,1}(\Omega_m)+\zeta_{nm,2}(\Omega_m).
\end{align}
Then, we write by \eqref{lowe-SS} and Lemma \ref{lem-pot1}
$$
\zeta_{nm,1}(\Omega_m)\geqslant c m^{-1}.
$$
For the second part, we use \eqref{zeta-eqMP4} and $\Omega_m\geqslant \widehat\Omega_{nm}-(nm)^{-\alpha}$ to get
\begin{align*}
|\zeta_{nm,2}(\Omega_m)|&\leqslant \frac{C_0}{n^2 m^2}\leqslant C_0 m^{-2}.
\end{align*}
Combining the preceding estimates with \eqref{zeta-eqMSL1} we deduce that for large $m$ and $n\geqslant 2$
$$
\zeta_{nm,2}(\Omega_m)\geqslant cm^{-1}.
$$
This shows in particular that $\zeta_{nm}$ is not vanishing at $\Omega_m$ as stated. 
\end{proof}

\subsection{Kernel generators}\label{sec-kernel-gen}
In this section, we will give the generators of the kernel of the linear operator $D_g \widehat{G}(\Omega_m,0):\mathscr{C}^{1,\alpha}_{s,m}(\D)\to \mathscr{C}^{1,\alpha}_{s,m}(\D)$ with  $\Omega_m$ being the eigenvalues  constructed along  the previous two sections. We set 
\begin{equation}\label{AOmega}
\mathcal{A}_{\Omega,m}:=\Big\{n\in m\N^\star, \quad \zeta_n(\Omega)=0\Big\}.
\end{equation}
then we deduce from \eqref{lin-op} and Proposition \ref{radialfunctions} that for $\Omega\notin ([\kappa_1,\kappa_2]\cup\mathbb{S}\cup\mathcal{S}^m_{\textnormal{sing}})$
\begin{align}\label{dim-kern-M}
\mbox{dim Ker}D_g \widehat{G}(\Omega,0)=\mbox{Card }\mathcal{A}_{\Omega,m}. 
\end{align}
From Proposition \ref{prop-omegan-low} and Proposition \ref{prop-omegam-asymp}, we can make the following summary. \\
 $\bullet$ {\it Case $1$: $f_0>0$ $($defocusing case$)$.} For $3\leqslant m\leqslant{\frac{f_0(0)}{f_0(1)-f_0(0)}}$, we  have at least one  $\Omega_m\in(-\infty,\kappa_1)\backslash\mathcal{S}^m_{\textnormal{sing}}$, with $\zeta_m(\Omega_m)=0$ and  $\zeta_{nm}(\Omega_m)\neq0$, for any $n\geqslant2.$ This implies that
$$
\mbox{dim Ker}D_g \widehat{G}(\Omega_m,0)=1.
$$
$\bullet$ {\it Case $2$:  $f_0<0$ $($focusing case$)$. } There exists $m_0\geqslant1$ large enough such that, for any $m\geqslant m_0$ there exists $\Omega_m\in(\widehat\Omega_m-{m^{-\alpha}},\widehat\Omega_m)$ with $\alpha\in(0,1)$ such that  $\zeta_m(\Omega_m)=0$ and $\zeta_{nm}(\Omega_m)\neq0$, for any $n\geqslant2.$ It follows that
$$
\mbox{dim Ker}D_g \widehat{G}(\Omega_m,0)=1.
$$ 
Notice that the assumption \eqref{pouisa} required at many steps to get  the dispersion equation in the defocusing case   is satisfied with our choice of $\Omega_m$ according the previous bounds together with \eqref{omega-a}, which lead   for any given $\theta\in(0,1)$ to 
$$
\lim_{m\to\infty}\frac{C_0}{m \theta (\Omega_m-\kappa_2)^\theta}=0<1,
$$
Another point to underline concerns the set $\mathbb{S}$ introduced in Proposition \ref{radialfunctions} which is a finite set embedded in $(\kappa_2,\infty).$ As the sequence $(\Omega_m)_{m\geqslant m_0}$ is convergent to $\kappa_2$, then it does not meet the set $\mathbb{S}$ for large $m_0$
\\
In what follows, we shall check that the elements of the $\hbox{Ker}D_g \widehat{G}(\Omega_m,0)$ are actually smooth enough, and this property is implicitly  used  to get \eqref{dim-kern-M}. 
\begin{pro}\label{prop-element-kernel}
Let $ f_0$ satisfy \eqref{H2}--\eqref{H1}, with $\beta\in(0,1).$ Let  $ \alpha\in(0,\beta),$ $m\geqslant1,$ and $\Omega_m$  as in the cases $1$ and $2$ discussed above. Then, the kernel of $D_g \widehat{G}(\Omega_m,0)$ is one-dimensional and generated by 
\begin{equation}\label{kernel-generator}
z=re^{i\theta}\in\mathbb{D}\mapsto h^\star(z)=h^\star_m(r)\cos(m\theta)\in \mathscr{C}_{s,m}^{1,\alpha}(\D),
\end{equation} with  
$$
h^\star_m(r)=r^m F_{n,\Omega_m}(r)\mu_{\Omega_m}^0(r)\int_1^r\frac{1}{F_{m,\Omega_m}^2(s)s^{2m+1}}\int_0^s F_{m,\Omega_m}(\tau)\tau^{2m+1}\left(\Omega_m-\tfrac{f_0(\tau)}{2}\right)d\tau ds,
$$
and $-\tfrac{H_m[h_m^\star](1)}{G_m(1)}=\tfrac{1}{2(m+1)}\cdot$

\end{pro}
\begin{proof}
From \eqref{kernel-eq} we have that  $h$ is an element of the kernel if
$$
h_m(r)-\frac{\mu_{\Omega_m}^0(r)}{2mr^m}H_m[h_m](r)=\tilde{V}(r),\quad\hbox{with}\quad 
\tilde{V}(r):=-\frac{H_m[h_m](1)}{2nG_n(1)}\mu_{\Omega_m}^0(r) rG_n(r).
$$
According to the discussion made above, the kernel is one dimensional in both cases provided that the elements of the kernel belong to the space $\mathscr{C}_{s,m}^{1,\alpha}(\D).$ By dilation we can assume that 
\begin{align}\label{Norm-P1}
-\tfrac{H_m[h_m](1)}{G_m(1)}=\tfrac{1}{2(m+1)}
\end{align} and the kernel equation becomes 
$$
h_m(r)-\frac{\mu_{\Omega_m}^0(r)}{2mr^m}H_m[h_m](r)=\tilde{V}(r),\quad\hbox{with}\quad 
\tilde{V}(r)=\tfrac{\mu_{\Omega_m}^0(r)}{4m(m+1)} rG_m(r).
$$
For $r=1$ we should get from the normalization 
\begin{align}\label{hn-11}
h_m(1)=0.
\end{align}
By the definition \eqref{Hn-def} and introducing 
$
\displaystyle{F_m(r):=\frac{h_m(r)}{r^m\mu_{\Omega_m}^0(r)},}
$
we find that
$$
F_m(r)-\frac{1}{2m}\int_r^1 s\mu_{\Omega_m}^0(s) F_m(s)ds-\frac{1}{2mr^{2m}}\int_0^r s^{2m+1}\mu_\Omega^0(s)F_m(s)ds=\tfrac{G_m(r)}{4m(m+1)r^{m-1}},
$$
with $F_m(1)=0.$ Hence
$$
F_m''(r)+\tfrac{2m+1}{r}F_m'(r)+\mu_{\Omega_m}^0(r)F_m(r)=\tfrac{1}{4m(m+1)}r^{-2m-1}\left(r^{2m+1}\left(\tfrac{G_m(r)}{r^{m-1}}\right)'\right)'.
$$
Straightforward computations using  \eqref{Gn-def} yield after some cancellations to 
$$
F_m''(r)+\tfrac{2m+1}{r}F_m'(r)+\mu_{\Omega_m}^0(r)F_m(r)=\Omega_m-\tfrac{f_0(r)}{2}.
$$
 By applying the variation of constants method with \eqref{diff-Fn} we find two constants $c_1,c_2$ such that
\begin{align*}
F_m(r)=&c_1 F_{m,\Omega_m}(r)+c_2 F_{m,\Omega_m}(r)\int_1^r \frac{ds}{F_{m,\Omega_m}^2(s)s^{2m+1}}\\
&+F_{m,\Omega_m}(r)\int_1^r\frac{1}{F_{m,\Omega_m}^2(s)s^{2m+1}}\int_0^s F_{m,\Omega_m}(\tau)\tau^{2m+1}\left(\Omega_m-\tfrac{f_0(\tau)}{2}\right)d\tau ds.
\end{align*}
As $F_m(1)=0$ then $c_1=0$ and therefore 
\begin{align*}
F_m(r)=&F_{m,\Omega_m}(r)\int_1^r\frac{1}{F_{m,\Omega_m}^2(s)s^{2m+1}}\int_0^s F_{m,\Omega_m}(\tau)\tau^{2m+1}\left(\Omega_m-\tfrac{f_0(\tau)}{2}\right)d\tau ds\\
&+c_2 F_{m,\Omega_m}(r)\int_1^r \frac{ds}{F_{m,\Omega_m}^2(s)s^{2m+1}}\cdot
\end{align*}
Coming back to $h_m$, the integral term associated to $c_2$ is singular at zero, and therefore the continuous solution takes the form 
$$
h_m^\star(r)= r^m F_{m,\Omega_m}(r)\mu_{\Omega_m}^0(r)\int_1^r\frac{1}{F_{m,\Omega_m}^2(s)s^{2m+1}}\int_0^s F_{m,\Omega_m}(\tau)\tau^{2m+1}\left(\Omega_m-\tfrac{f_0(\tau)}{2}\right)d\tau ds.
$$
The renormalization \eqref{Norm-P1} can be written in the form 
\begin{align}\label{Ren-Hm}
-\tfrac{H_m[h_m^\star](1)}{G_m(1)}=\tfrac{1}{2(m+1)}\cdot
\end{align}

Now, we shall prove that the generator of the kernel  $h^\star(z):=h_m^\star(r)\cos(m\theta)$ belongs to $\mathscr{C}^{1,\alpha}(\D)$. Set   $F_{m,\Omega_m}(r)=\tilde{F}_{m,\Omega_m}(r^2)$ and $\mu_\Omega^0(r)=\tilde{\mu}_\Omega^0(r^2)$, and define
$$
\mathcal{G}_m(x):=\frac14\tilde{F}_{m,\Omega}(x)\tilde{\mu}^0_{\Omega_m}(x)\int_1^{x}\frac{1}{\tilde{F}_{m,\Omega_m}^2(s)s^{m+1}}\int_0^s \tilde{F}_{m,\Omega_m}(\tau)\tau^{m}\left(\Omega_m-\tfrac{\tilde{f}_0(\tau)}{2}\right)d\tau ds.
$$
We observe that $h^\star$ can be written as
$$
h^\star(z)=\textnormal{Re}\left[\mathcal{G}_m(|z|^2)z^m\right].
$$
If $\tilde{F}_{m,\Omega_m}$ is $\mathscr{C}^2$ then we find that $\mathcal{G}_m$ is $\mathscr{C}^{1,\alpha}([0,1])$ leading to  $h^\star\in\mathscr{C}^{1,\alpha}(\mathbb{D})$. 
Let us check that $\tilde{F}_{m,\Omega_m}$ is $\mathscr{C}^2$. Note that since $F_{m,\Omega_m}$ is $\mathscr{C}^2$, the only problematic point for the regularity of   $\tilde{F}_{m,\Omega_m}$ is 0. First, notice that $\tilde{F}_{m,\Omega_m}(0)=1$. Second,
{$$
\lim_{x\rightarrow 0}\tilde{F}_{m,\Omega_m}'(x)=\tfrac12 \lim_{r\rightarrow 0}\tfrac{F_{m,\Omega_m}(r)}{r}=-\tfrac{\sigma_{\Omega_m}}{2(2m+2)}\nu_{\Omega_m}(0),
$$ }
where we have used for the last inequality the ODE  \eqref{Tip1}. Third, let us compute $\lim_{x\rightarrow 0}\tilde{F}_{m,\Omega_m}''(x)$. To do so, we notice that by  standard computations  the function   $\tilde{F}_{m,\Omega_m}$ solves the  differential equation
$$
x\tilde{F}_{m,\Omega_m}''(x)+\tilde{F}_{m,\Omega_m}'(x)(m+1)+\frac{\sigma_{\Omega_m}}{4}\tilde{\nu}_{\Omega_m}(x)\tilde{F}_{m,\Omega_m}(x)=0,
$$
and then
\begin{align*}
\lim_{x\rightarrow 0}\tilde{F}_{m,\Omega_m}''(x)=&-\lim_{x\rightarrow 0}\frac{\tilde{F}_{m,\Omega_m}'(x)(m+1)+\frac{\sigma_{\Omega_m}}{4}\tilde{\nu}_{\Omega_m}(x)\tilde{F}_{m,\Omega_m}(x)}{x}\\
=-&\lim_{x\rightarrow 0}\left\{\tilde{F}_{m,\Omega_m}''(x)(m+1)+\frac{\sigma_{\Omega_m}}{4}[\tilde{\nu}_{\Omega_m}(x)\tilde{F}_{m,\Omega_m}(x)]'\right\},
\end{align*}
which implies
$$
\lim_{x\rightarrow 0}\tilde{F}_{m,\Omega_m}''(x)=\frac{-\sigma_{\Omega_m}}{4(m+2)}\left[\tilde{\nu}_{\Omega_m}'(0)-\frac{\tilde{\nu}_{\Omega_m}(0)^2\sigma_{\Omega_m}}{2(2m+1)}\right].
$$
This implies that $\tilde{F}_{m,\Omega_m}$ is of class $\mathscr{C}^2 ([0,1])$. This achieves the proof of the desired result.
\end{proof}
\section{Transversality}\label{Sec-transv}
In  Proposition \ref{prop-element-kernel} we found, according the monotonicity and the sign of the profile  $f_0$, suitable values of  some $m$ and $\Omega_m$ such that the kernel of $D_g\widehat{G}(\Omega_m,0)$ is one dimensional. Moreover, by Proposition \ref{compactop} we get that this linear operator is  Fredholm with zero index, implying in particular that the   co-dimension of its range is  $1$. In order to apply the {Crandall-Rabinowitz theorem \cite{rabinowitz:simple},} we need to check  the transversal condition, which reads 
$$
D_\Omega D_g \widehat{G}(\Omega_m,0)h^\star\notin \textnormal{Range } D_g\widehat{G}(\Omega_m,0),
$$
where $h^\star$ is the generator of the kernel defined by \eqref{kernel-generator}.
To proceed, we shall first characterize the range and later  prove that the transversal condition for both  cases $f_0>0$ and $f_0<0$.\\
Given $\alpha\in(0,1)$ and $ m\geqslant 1$, we have seen in Proposition \ref{Gwelldefined} that  the functional  $\widehat{G}$ defined \mbox{by  \eqref{densityEq} }satisfies 
$$
\widehat{G}:\R\times B_{\mathscr{C}_{s,m}^{1,\alpha}(\D)}\rightarrow \mathscr{C}^{1,\alpha}_{s,m}(\D).
$$
The first main result characterizes the range  of $D_g \widehat{G}(\Omega_m,0)$ as the kernel of a linear form.
\begin{pro}\label{Prop-Range22}
Let $m\geqslant 1,\beta\in(0,1),  \alpha\in(0,\beta)$ and $f_0$ satisfy \eqref{H2}-\eqref{H1}. Let  $\Omega_m$ as in Proposition \ref {prop-element-kernel}. Then 
$$
\textnormal{Range } D_g\widehat{G}(\Omega_m,0)=\left\{d\in\mathscr{C}^{1,\alpha}_{s,m}(\D), \quad \int_{\D} d(z)\mathscr{K}_{m}(z)dz=0\right\},
$$
where
$$
\mathscr{K}_{m}(z):=\textnormal{Re}\left[\nu_{\Omega_m}(|z|) F_{m,\Omega_m}(|z|)z^{m}\right].
$$
\end{pro}
\begin{proof}
In order to obtain the range of the linearized operator, we need to study the equation
$$
D_g \widehat{G}(\Omega,0)h=d, \quad h(re^{i\theta})=\sum_{n\in m\N}h_{n}(r)\cos(n\theta), \quad d(re^{i\theta})=\sum_{n\in m\N }d_{n}(r)\cos(n\theta).
$$
Due to \eqref{lin-op}, that equals to
\begin{align*}
\frac{h_{n}(r)}{\mu_\Omega^0(r)}-\frac{r}{n}\left(A_{n}[h_{n}]G_{n}(r)+\frac{1}{2r^{n+1}}H_{n}[h_{n}](r)\right)=d_{n}(r), \quad& n\in m\N^\star,\\
\frac{h_0(r)}{\mu_\Omega^0(r)}-\int_r^1\frac{1}{\tau}\int_0^\tau sh_0(s)dsd\tau=d_0(r).
\end{align*}
For the equation for $n=0$, note that $\mathbb{L}_0^\Omega$ is a compact perturbation of an isomorphism and thus it is Fredholm of zero index. Moreover, from Proposition \ref{radialfunctions}, since $\Omega_m\notin\mathbb{S}$  then the kernel is trivial and thus  the equation for $n=0$ admits a unique solution. Let us focus on the equation for $n\geqslant  1$ and work as in the preceding study for the kernel. Then similarly to \eqref{kernel-eqbis} we write
$$
(\textnormal{Id}-\sigma_\Omega \mathcal{L}_{n}^{\Omega_m})[h_{n}](r)=-\tfrac{H_{n}[h_{n}](1)}{2nG_{n}(1)}\mu_{\Omega_m}^0(r)rG_{n}(r)+\mu_{\Omega_m}^0(r)d_{n}(r).
$$
 By Proposition \ref{prop-Invert1} and the assumptions on $\Omega_m$, we have that $(\textnormal{Id}-\sigma_{\Omega_m} \mathcal{L}_{n}^{\Omega_m})$ is invertible for any $n\geqslant 1$ (note that here we need to use that $n\in m\N^\star$) leading to
$$
h_n(r)=-\tfrac{H_n[h_n](1)}{2nG_n(1)}(\textnormal{Id}-\sigma_{\Omega_m} \mathcal{L}_n^{\Omega_m})^{-1}[\mu_{\Omega_m}^0(r)rG_n(r)]+(\textnormal{Id}-\sigma_{\Omega_m} \mathcal{L}_n^{\Omega_m})^{-1}[\mu_{\Omega_m}^0 d_n](r).
$$
However, to be sure that $h_n$ is a solution we need to check the compatibility condition related to $H_n[h_n](1)$. It can be done by   multiplying the preceding identity  by $r^{n+1}$ and \mbox{integrating in $[0,1]$,}
\begin{align*}
H_n[h_n](1)&=-\tfrac{H_n[h_n](1)}{2nG_n(1)}\int_0^1 r^{n+1} (\textnormal{Id}-\sigma_{\Omega_m} \mathcal{L}_n^{\Omega_m})^{-1}[\mu_{\Omega_m}^0(r)rG_n(r)]dr\\
&\qquad +\int_0^1 r^{n+1}(\textnormal{Id}-\sigma_{\Omega_m} \mathcal{L}_n^{\Omega_m})^{-1}[\mu_{\Omega_m}^0 d_n](r)dr.
\end{align*}
Using the function $T(n,\Omega)$ defined in \eqref{TnOmega},  from Lemma \ref{Lem-disper}, we get
$$
\big(1-T(n,\Omega_m)\big)H_n[h_n](1)=\int_0^1 r^{n+1}(\textnormal{Id}-\sigma_{\Omega} \mathcal{L}_n^\Omega)^{-1}[d_n\mu_\Omega^0](r)dr.
$$
Since the kernel is one-dimensional then
$$
n\in m\N^\star, \quad 1-T(n,\Omega_m)=0\Longleftrightarrow n=m
$$
Therefore, the compatibility condition implies that
$$
\int_0^1 r^{m+1}(\textnormal{Id}-\sigma_{\Omega_m} \mathcal{L}_m^\Omega)^{-1}[\mu_{\Omega_m}^0 d_m](r)dr=0,
$$
Then using that $\mathcal{L}_{m}^{\Omega_m}$ is self-adjoint as stated in Proposition \ref{pro-positivity},  the previous identity  agrees with
$$
\int_0^1 \big((\textnormal{Id}-\sigma_{\Omega_m} \mathcal{L}_m^{\Omega_m})^{-1}[\nu_{\Omega_m} r^{m}]\big) d_m(r)r dr=0.
$$
Set $h=(\textnormal{Id}-\sigma_{\Omega_m} \mathcal{L}_m^{\Omega_m})^{-1}[\nu_{\Omega_m} r^{m}]$, then by virtue of \eqref{Eq-t1} and  \eqref{h-exp}   we find a constant $c\neq 0$ such that
$$
h(r)=c \nu_{\Omega_m}(r) r^m F_{m,\Omega_m}(r),
$$
and hence we should get  
$$
\int_0^1  \nu_{\Omega_m}(r) F_{m,\Omega_m}(r) r^{m} d_m(r) rdr=0.
$$
Introduce the function
$$
z\in\mathbb{D},\quad \mathscr{K}_{m}(z)=\textnormal{Re}\left[\nu_{\Omega_m}(|z|) F_{m,\Omega_m}(|z|)z^{m}\right].
$$
Then, straightforward calculus based on  polar coordinates and  Fourier expansion allows to write   the compatibility condition in the form
$$
\int_{\D} d(z)\mathscr{K}_{m}(z)dz=0.
$$
It follows that
$$
\textnormal{Range } D_g\widehat{G}(\Omega_m,0)\subset \left\{d\in\mathscr{C}^{1,\alpha}_{s,m}(\D), \quad \int_{\D} d(z)\mathscr{K}_{m}(z)dz=0\right\}:=\mathcal{E}.
$$
Since $\mathscr{K}_{m}\in \mathscr{C}^{1,\alpha}_{s,m}(\D)$ (a weak regularity is  enough for the result), then the mapping 
$$d\in\mathscr{C}^{1,\alpha}_{s,m}(\D)\mapsto \int_{\D} d(z)\mathscr{K}_{m}(z)dz\in\R
$$
is continuous and therefore its kernel $\mathcal{E}$ defines a hyperplane (a closed space of co-dimension one). Combining this with the fact that $\textnormal{Range } D_g\widehat{G}(\Omega_m,0)$ is a closed subspace  of $ \mathscr{C}^{1,\alpha}_{s,m}(\D)$ of co-dimension one, we get the equality between the two subspaces. 
 This concludes the proof.
\end{proof}
The next goal is to write down an analytical equation for the transversality assumption in Crandall-Rabinowitz theorem \cite{rabinowitz:simple}. For this aim, we need to introduce the following function
$$
d^\star_m(r):=h_m^\star(r)\tfrac{r}{f_0'(r)}+\tfrac{A_m[h_m^\star]}{G_m(1)}rG_m(r)-A_m[h_m^\star] r^{m+2},
$$
where $h^\star_m$ is defined in Proposition \ref{prop-element-kernel} and $A_m[h]$ is defined by \eqref{An-HHH}. Therefore, we get by \eqref{Ren-Hm},
\begin{align}\label{dn-simple}
d^\star_m(r)=h_m^\star(r)\tfrac{r}{f_0'(r)}+\tfrac{1}{4(m+1)G_m(1)}rG_m(r)-\tfrac{1}{4(m+1)} r^{m+2},
\end{align}
We intend to prove the following result.
\begin{pro}
Let $m\geqslant 1,\beta\in(0,1),  \alpha\in(0,\beta)$ and $f_0$ satisfy \eqref{H2}-\eqref{H1}. Let  $\Omega_m$ as in Proposition \ref {prop-element-kernel}. Then, the transversal condition agrees with
$$
\int_0^1 \nu_{\Omega_m}(r) r^m F_{m,\Omega_m}(r) d_m^\star(r)r dr\neq 0.
$$
\end{pro}
\begin{proof}
Note that from the expression of $D_g \widehat{G}(\Omega,0)$ described by \eqref{lin-op}, one finds
\begin{align*}
D_\Omega D_g\widehat{G}(\Omega,0)h(r,\theta)&=\sum_{n\in m\N^\star}\cos(n\theta)\left\{h_{n}(r)\tfrac{r}{f_0'(r)}+\tfrac{A_{n}[h_{n}]}{G_{n}(1)}rG_{n}(r)-A_{n}[h_n] r^{n+2}\right\}\\
&+{h_{0}(r)\tfrac{r}{f_0'(r)}}\cdot
\end{align*}
Then, we deduce in view of the definition of $d_m^\star$, whose simplified form is given by  \eqref{dn-simple},
$$
D_\Omega D_g\widehat{G}(\Omega_m,0)h^\star(r,\theta)=d_m^\star(r)\cos(m\theta).
$$
To check the transversal condition we need to prove that
$$
d^\star_m(r)\cos(m\theta)\notin \textnormal{Range}( D_g\widehat{G}(\Omega,0)).
$$
Applying Proposition \ref{Prop-Range22}, it agrees with 
\begin{align}\label{Trans-V1}
 \mathbb{I}_m:=\int_0^1 \nu_{\Omega_m}(r) r^m F_{m,\Omega_m}(r) d_m^\star(r)r dr\neq 0,
\end{align}
and this achieves  the proof of the desired result.
\end{proof}
Coming back to \eqref{dn-simple} and using the expression of  $h^\star_m$ in Proposition \ref{prop-element-kernel}
\begin{align*}
4(m+1)d_m^\star(r)=&r^m\Big[\tfrac{4(m+1)r}{f_0'(r)} F_{m,\Omega_m}(r)\mu_{\Omega_m}^0(r)\\
&\times\int_1^r\tfrac{1}{F_{m,\Omega_m}^2(s)s^{2m+1}}\int_0^s F_{m,\Omega_m}(\tau)\tau^{2n+1}\left\{\Omega_m-\tfrac{f_0(\tau)}{2}\right\}d\tau ds\\
&-r^2+\tfrac{r^{1-m}G_m(r)}{G_m(1)}\Big]\\
=:&r^m\Big[\mathcal{H}_{m,1}(r)+\mathcal{H}_{m,2}(r)+\mathcal{H}_{m,3}(r)\Big],
\end{align*}
with
\begin{align}\label{Hm-Tra1}
\nonumber \mathcal{H}_{m,1}(r)=&-\tfrac{4(m+1)r\mu^0_{\Omega_m}(r)}{f_0'(r)} F_{m,\Omega_m}(r)\int_r^1\tfrac{1}{F_{m,\Omega_m}^2(s)s^{2m+1}}\int_0^s F_{m,\Omega_m}(\tau)\tau^{2n+1}\left\{\Omega_m-\tfrac{f_0(\tau)}{2}\right\}d\tau ds,\\
\mathcal{H}_{m,2}(r)=&-r^2,\quad \mathcal{H}_{m,3}(r)=\tfrac{r^{1-m}G_m(r)}{G_m(1)}.
\end{align}
Then \eqref{Trans-V1} is equivalent to 
\begin{align}\label{Trans-V2}
\nonumber \mathbb{I}_m&= \int_0^1 \nu_{\Omega_m}(r) r^{2m+1} F_{m,\Omega_m}(r) \big(\mathcal{H}_{m,1}(r)+\mathcal{H}_{m,2}(r)+\mathcal{H}_{m,3}(r)\big) dr\\
&=: \mathbb{I}_{m,1}+\mathbb{I}_{m,2}+\mathbb{I}_{m,3}\neq 0.
\end{align}
Let us start with the case $f_0>0$ associated with  lower values of  $m$.
\begin{pro}\label{prop-trans-low}
Let $f_0>0$ satisfy \eqref{H2}--\eqref{H1}. Fix  $m$ such that
$$
3\leqslant m\leqslant \tfrac{1}{10}\tfrac{f_0(0)}{f_0(1)-f_0(0)},
$$ 
and $\Omega_m$  as in Proposition $\ref{prop-omegan-low},$ then \, $\mathbb{I}_m\neq 0$.
\end{pro}
\begin{proof}
Since ${\Omega_m}\in(-\infty,\kappa_1)$ and by \eqref{nu-mu} we get that ${\Omega_m}-\frac{f_0(\tau)}{2}<0$ and $\mu^0_{\Omega_m}(r)\leqslant 0$ obtaining that $\mathcal{H}_{m,1}\leqslant 0$. Let us check the sign of $\mathcal{H}_{m,2}+\mathcal{H}_{m,3}$, which takes the form in view of \eqref{Gn1-1}
\begin{align*}
\mathcal{H}_{m,2}(r)+\mathcal{H}_{m,3}(r)=&\tfrac{r^{1-m}G_m(r)-r^2G_m(1)}{G_m(1)}\\
=&\tfrac{r^{1-m}G_m(r)-r^2G_m(1)}{m({\Omega_m}-\widehat{{\Omega}}_m)}\cdot
\end{align*}
By the definition \eqref{Gn-def} we infer that
\begin{align*}
r^{1-m}G_m(r)-r^2G_m(1)=&(1-r^2)\int_0^1sf_0(s)ds-(m+1)\int_0^r sf_0(s)ds+\frac{m+1}{r^{2m}}\int_0^r s^{2m+1}f_0(s)ds\\
&+(m+1)r^2\int_0^1 sf_0(s)ds-(m+1)r^2\int_0^1s^{2m+1}f_0(s)ds.
\end{align*}
Thus, making the change of variable $s=r \tau$ allows to get
\begin{align*}
r^{1-m}G_m(r)-&r^2G_m(1)=(1-r^2)\int_0^1sf_0(s)ds\\
&+(m+1)\left\{-r^2\int_0^1 sf_0(sr)ds+r^2\int_0^1 s^{2m+1}f_0(sr)ds+r^2\int_0^1 sf_0(s)ds\right.\\
&\left.-r^2\int_0^1s^{2m+1}f_0(s)ds\right\}\\
&=(1-r^2)\int_0^1sf_0(s)ds+(m+1)r^2\int_0^1 s(1-s^{2m})(f_0(s)-f_0(sr))ds.
\end{align*}
As $f_0$ increases, we find
\begin{align*}
\forall r\in[0,1),\quad r^{1-m}G_m(r)-r^2G_m(1)
&> 0,
\end{align*}
Applying  Proposition \ref{prop-omegan-low}-$(4)$ we find that $\Omega_m-\widehat{{\Omega}}_m> 0$ and then
$$
\forall r\in[0,1),\quad \mathcal{H}_{m,2}(r)+\mathcal{H}_{m,3}(r)> 0.
$$
This shows  competition of sign between $\mathcal{H}_{m,1}$ and $\mathcal{H}_{m,2}+\mathcal{H}_{m,3}$ and we will check that their sum is positive. Indeed, we can write
\begin{align*}
&\mathcal{H}_{m,1}+\mathcal{H}_{m,2}+\mathcal{H}_{m,3}=\frac{(1-r^2)\int_0^1sf_0(s)ds+(m+1)r^2\int_0^1 s(1-s^{2m})(f_0(s)-f_0(sr))ds}{m(\Omega_m-\widehat{\Omega}_m)}\\
&+\tfrac{2(m+1)r\mu_\Omega^0(r)}{f_0'(r)} F_{m,\Omega_m}(r)\int_r^1\tfrac{1}{F_{m,\Omega_m}^2(s)s^{2m+1}}\int_0^s F_{m,\Omega_m}(\tau)\tau^{2m+1}\left\{{f_0(\tau)}-2\Omega_m\right\}d\tau ds.
\end{align*}
According to  Lemma \ref{lem-pot1}-(1) and \eqref{diff-Fn} we deduce that $F_{m,\Omega_m}$ is increasing. Then
combined with $\frac{f_0(\tau)}{2}-\Omega_m>0$, we obtain by the monotonicity of $f_0$ 
\begin{align*}
&\int_r^1\tfrac{1}{F_{m,\Omega_m}^2(s)s^{2m+1}}\int_0^s F_{m,\Omega_m}(\tau)\tau^{2m+1}\big(f_0(\tau)-2\Omega_m\big)d\tau ds\\
&\leqslant  \int_r^1\frac{1}{s^{2m+1}}\int_0^s \tau^{2m+1}\left\{f_0(\tau)-2\Omega_m\right\}d\tau ds\\
&\leqslant \frac{1-r^2}{4(m+1)}\big(f_0(1)-2\Omega_m\big).
\end{align*}
Using  $\mu^0_{\Omega_m}(r)\leqslant 0$ together with
\begin{align*}
\frac{r\mu_\Omega^0(r)}{f_0'(r)}=\frac{1}{\Omega-\int_0^1sf_0(rs)ds}\geqslant\frac{1}{\Omega-\kappa_1},\quad \forall\,  r\in  (0,1],
\end{align*}
 give
\begin{align*}
&\tfrac{2(m+1)r\mu_\Omega^0(r)}{f_0'(r)} F_{m,\Omega_m}(r)\int_r^1\tfrac{1}{F_{m,\Omega_m}^2(s)s^{2m+1}}\int_0^s F_{m,\Omega_m}(\tau)\tau^{2m+1}\big({f_0(\tau)}-2\Omega_m\big)d\tau ds\\
&\geqslant \frac{1-r^2}{2(\Omega-\kappa_1)}\big(f_0(1)-2\Omega_m\big).
\end{align*}
Consequently,  we get for any $r\in(0,1)$
\begin{align*}
\mathcal{H}_{m,1}+\mathcal{H}_{m,2}+\mathcal{H}_{m,3}
>& \frac{(1-r^2)\int_0^1sf_0(s)ds}{m(\Omega_m-\widehat{\Omega}_m)}-\frac{1-r^2}{2(\kappa_1-\Omega)}\big(f_0(1)-2\Omega_m\big)\\
>&(1-r^2)\left(\frac{\kappa_2}{m(\Omega_m-\widehat{\Omega}_m)}-\frac{f_0(1)-2\Omega_m}{2(\kappa_1-\Omega_m)} \right)=:(1-r^2)D_m.
\end{align*}
It follows that
\begin{align}\label{coupe0}
D_m>0\Longleftrightarrow 2(\kappa_1-\Omega_m)\kappa_2>m(\Omega_m-\widehat{\Omega}_m)\big(f_0(1)-2\Omega_m\big).
\end{align}
From \eqref{Interv1}  and the monotonicity of $f_0$ we infer that
$$
\widehat{\Omega}_m\geqslant \kappa_2-\frac{f_0(1)}{2m},
$$
which implies using  Proposition \ref{prop-omegan-low}-(4) 
\begin{align*}
\Omega_m-\widehat{\Omega}_m&\leqslant \frac{m\kappa_2}{m+1}-\kappa_2+\frac{f_0(1)}{2m}\\
&\leqslant \frac{f_0(1)}{2m}-\frac{\kappa_2}{m+1}
\end{align*}
and
\begin{align*}
0\leqslant f_0(1)-2\Omega_m&\leqslant f_0(1)-2\widehat\Omega_m\\
&\leqslant \tfrac{m+1}{m}f_0(1)- 2\kappa_2.
\end{align*}
Thus
\begin{align}\label{coupe1}
m(\Omega_m-\widehat{\Omega}_m)\big(f_0(1)-2\Omega_m\big)&\leqslant   \tfrac{m+1}{2m}\big(f_0(1)-\tfrac{2m}{m+1} \kappa_2\big)^2.
\end{align}
On the other hand
\begin{align*}
(\kappa_1-\Omega_m)\geqslant \kappa_1-\tfrac{m}{m+1}\kappa_2,
\end{align*}
and therefore
\begin{align}\label{coupe2}
2(\kappa_1-\Omega_m)\kappa_2\geqslant 2\big(\kappa_1-\tfrac{m}{m+1}\kappa_2\big)\kappa_2.
\end{align} 
Moreover, under the assumption 
\begin{align}\label{coupe6}2\leqslant 2m\leqslant \frac{\kappa_1}{\kappa_2-\kappa_1},
\end{align} we find
\begin{align*}
\big(\kappa_1-\tfrac{m}{m+1}\kappa_2\big)&\geqslant \big(\tfrac{2m}{2m+1}-\tfrac{m}{m+1}\big)\kappa_2{\geqslant}\frac{m\kappa_2}{(m+1)(2m+1)}.
\end{align*}
Inserting this inequality into \eqref{coupe2} we get
\begin{align}\label{coupe3}
2(\kappa_1-\Omega_m)\kappa_2&\geqslant \tfrac{2m}{(m+1)(2m+1)}\kappa_2^2.
\end{align} 
Hence, putting together \eqref{coupe1} and \eqref{coupe2} we deduce that  \eqref{coupe0} holds true if 
\begin{align*}
\tfrac{2m}{(m+1)(2m+1)}\kappa_2^2> \tfrac{m+1}{2m}\big(f_0(1)-\tfrac{2m}{m+1} \kappa_2\big)^2.
\end{align*} 
Now, define $\delta:=f_0(1)-2\kappa_2>0$ then the preceding inequality is equivalent to
\begin{align*}
\kappa_2^2> \tfrac{(m+1)^2(2m+1)}{{4m^2}}\big(\delta+\tfrac{2}{m+1} \kappa_2\big)^2,
\end{align*} 
or
\begin{align}\label{coupe5}
\big(1-\tfrac{\sqrt{2m+1}}{m}\big)\kappa_2> \tfrac{(m+1)\sqrt{2m+1}}{2m}\delta.
\end{align} 
Since
\begin{align*}
1-\tfrac{\sqrt{2m+1}}{m}=\tfrac{m^2-2m-1}{m(m+\sqrt{2m+1})}
\end{align*} 
and  the function $x\in(0,\infty)\mapsto \tfrac{x^2-2x-1}{x(x+\sqrt{2x+1})}$ is strictly increasing, then we deduce that
$$
\forall \, m\geqslant 3,\quad \tfrac{m^2-2m-1}{m(m+\sqrt{2m+1})}\geqslant \tfrac{2}{3(3+\sqrt{7})}.
$$
Similarly the function  $x\in(0,\infty)\mapsto  \tfrac{(x+1)\sqrt{2x+1}}{2x^2}$ is strictly decreasing and therefore
$$
\forall \, m\geqslant 3,\quad \tfrac{(m+1)\sqrt{2m+1}}{2m^2}\leqslant \tfrac{2\sqrt{7}}{9}.
$$
It follows that \eqref{coupe5} holds true if
$$
m\delta<\tfrac{3}{\sqrt{7}(3+\sqrt{7})} \kappa_2,
$$
and since $\tfrac{3}{\sqrt{7}(3+\sqrt{7})}>\frac15$ then the previous inequality holds true if
$$
3\leqslant m\leqslant \frac{\kappa_2}{5\delta}.
$$
Coming back to \eqref{coupe6} we get \eqref{coupe0} provided that
$$
3\leqslant m\leqslant \min\left(\frac{\kappa_2}{5\delta},\frac{\kappa_1}{2(\kappa_2-\kappa_1)}\right).
$$
As $\frac{f_0(0)}{2}=\kappa_1\leqslant\kappa_2,$ then the previous inequality occurs if
\begin{align}\label{coupe7}
3\leqslant m\leqslant \tfrac15\kappa_1\min\left(\tfrac{1}{f_0(1)-2\kappa_2},\tfrac{1}{2\kappa_2-f_0(0)}\right).
\end{align}
Remark that by the monotonicity of $f_0$ we get 
\begin{align*}
f_0(1)-2\kappa_2=2\int_0^1s(f_0(1)-f_0(s))ds\leqslant f_0(1)-f_0(0),
\end{align*}
and
\begin{align*}
{2\kappa_2-f_0(0)}=2\int_0^1s(f_0(s)-f_0(0))ds\leqslant f_0(1)-f_0(0).
\end{align*}
Thus,  to get \eqref{coupe7}, it is enough to impose
\begin{align*}
3\leqslant m\leqslant \frac{1}{10}\frac{f_0(0)}{f_0(1)-f_0(0)}\cdot
\end{align*}
This ends the proof of Proposition \ref{prop-trans-low}.
\end{proof}
Let us now move  to the transversality condition when $f_0<0$.
\begin{pro}
Let $f_0$ satisfy \eqref{H2}--\eqref{H1} with $f_0<0$, and take $m$ and $\Omega_m$ as  in Proposition $\ref{prop-omegam-asymp}. $
There exists $m_0\geqslant 1$ such that for any $m\geqslant m_0$ we get \, $\mathbb{I}_m\neq 0.$
\end{pro}
\begin{proof}
Applying Proposition \ref{prop-omegam-asymp}, with $\alpha=\frac32$,  we get that
\begin{align}\label{asiyq}
\Omega_m=\widehat{\Omega}_m+O\left(m^{-\frac32}\right),
\end{align}
and
\begin{align*}
\widehat{\Omega}_m&=\int_0^1 sf_0(s)ds-\frac{m+1}{m}\int_0^{1}s^{2m+1}f_0(s)ds\\
&=\int_0^1sf_0(s)ds-\frac{1}{2m}f_0(1)+\frac{m+1}{m}\int_0^{1}s^{2m+1}\big[f_0(1)-f_0(s)\big]ds\\
&=\int_0^1sf_0(s)ds-\frac{1}{2m}f_0(1)+O(m^{-2}).
\end{align*}
It follows that
\begin{align}\label{asym-diff}
 \Omega_m-\kappa_2=-\frac{1}{2m}f_0(1)+O(m^{-\frac32}).
 \end{align}
We start with an  estimate the first term $\mathbb{I}_{m,1}$ in \eqref{Trans-V2}  which is based on the control of  $\mathcal{H}_{m,1}$. From \eqref{Hm-Tra1}, \eqref{mutrivial} and  Lemma \ref{lem-pot1}-(2) we infer
\begin{align*}
|\mathcal{H}_{m,1}(r)|\leqslant \frac{Cm}{({\Omega_m-\int_0^1 s f_0(rs)ds})} \int_r^1\frac{1}{s^{2m+1}}\int_0^s \tau^{2m+1}d\tau ds.
\end{align*}
It follows from  \eqref{nu102} applied with $\theta=0$  that 
\begin{align*}
|\mathcal{H}_{m,1}(r)|&\leqslant \frac{m}{(1-r)}\int_1^r s^{-2m-1}\int_0^s \tau^{2m+1}d\tau dr\leqslant C.
\end{align*} Hence, we deduce from Lemma \ref{prop-InvertT1} applied with $\theta=1$ together with  \eqref{asym-diff}%
\begin{align*}
|\mathbb{I}_{m,1}|&\leqslant \left|\int_0^1 \nu_{\Omega_m}(r) r^{2m+1} |\mathcal{H}_{m,1}(r)|dr\right|\leqslant  \frac{C}{m(\Omega_m-\kappa_2)}\leqslant  C,
\end{align*}
and consequently
\begin{align}\label{Im-1}
{\limsup_{m\to\infty} } \,\mathbb{I}_{m,1}=C<+\infty.
\end{align}
Let us now move to the estimate of $\mathbb{I}_{m,2}$ which is quite similar to  $\mathbb{I}_{m,1}$. Then similar arguments based on Lemma \ref{prop-InvertT1}  yield to the estimate
\begin{align*}
| \mathbb{I}_{m,2}|&\leqslant \left|\int_0^1 \nu_{\Omega_m}(r) r^{2m+3} dr\right|\leqslant \frac{C}{(\Omega_m-\kappa_2)^\theta}\int_0^1 {\frac{r^{2m+3}}{(1-r)^{1-\theta}}}drdr.
\end{align*}
Let $\epsilon\in(0,1)$ then by splitting the integral we get
\begin{align*}
\int_0^1 {\frac{r^{2m+3}}{(1-r)^{1-\theta}}}dr&\leqslant \epsilon^{\theta-1}\int_0^{1-\epsilon} {{r^{2m+3}}}dr+\int_{1-\epsilon}^1 {(1-r)^{\theta-1}}dr\\
&\leqslant \frac{\epsilon^{\theta-1}}{2m+4}+\frac{\epsilon^\theta}{\theta}.
\end{align*}
Taking $\epsilon\approx m^{-1}$ yields 
\begin{align*}
\int_0^1 {\frac{r^{2m+3}}{(1-r)^{1-\theta}}}dr&\leqslant \frac{C}{\theta m^\theta}.
\end{align*}
Therefore we deduce for any $\theta\in(0,1)$ 
\begin{align*}
| \mathbb{I}_{m,2}|
&\leqslant \frac{C}{\theta m^\theta(\Omega_m-\kappa_2)^\theta}\cdot
\end{align*}
By fixing $\theta=\frac12$ and using \eqref{asym-diff} we find a constant $C>0$ such that 
\begin{align}\label{Im-2}
\forall m\in\N,\quad | \mathbb{I}_{m,2}|{\leqslant \frac{C}{m}\leqslant C.}
\end{align}
Let us now focus on the term  $\mathbb{I}_{m,3}$ which is more involved.  First we write from \eqref{Trans-V2}
\begin{align}\label{Mrad1}
\nonumber \mathbb{I}_{m,3}&=\int_0^1 \nu_{\Omega_m}(r) r^{2m+1} \mathcal{H}_{m,3}(r) dr+\int_0^1 \nu_{\Omega_m}(r) r^{2m+1}\big(F_{m,\Omega_m}(r)-1\big) \mathcal{H}_{m,3}(r) dr\\
 &=: \mathbb{I}_{m,3}^1+\mathbb{I}_{m,3}^2.
\end{align}
It is obvious from \eqref{Hm-Tra1} that
$$
\mathbb{I}_{m,3}^1= \frac{1}{G_m(1)}\int_0^1 \nu_{\Omega_m}(r)r^{m+2} G_m(r)dr.
$$
Next, we will write $G_m$ defined through \eqref{Gn-def} as follows
$$
G_m(r)=m r^{m-1}P_m(r),
$$
with
\begin{align*}
P_m(r)=&r^2\Omega_m+\frac{1}{m}\int_0^1 sf_0(s)ds-\frac{m+1}{m}\int_0^r sf_0(s)ds+\frac{m+1}{m r^{2m}}\int_0^r s^{2m+1}f_0(s)ds\\
=&r^2\left[\Omega_m-\int_0^1f_0(sr)ds\right]+\frac{1}{m}\int_r^1 sf_0(s)ds+\frac{m+1}{mr^{2m}}\int_0^r s^{2m+1}f_0(s)ds.
\end{align*}
Thus,  we get from \eqref{mutrivial} and \eqref{Gn1-1}
\begin{align}\label{Dec-pui}
\nonumber\mathbb{I}_{m,3}^1&=\frac{1}{\Omega_m-\widehat{\Omega}_m}\int_0^1 \frac{f_0'(r) r^{{2m}}P_m(r)}{(\Omega_m-\int_0^1 sf_0(sr)ds)}dr\\
\nonumber=&\frac{1}{\Omega_m-\widehat{\Omega}_m}\int_0^1 {f_0'(r)}r^{2m+2}dr+\frac{1}{(\Omega_m-\widehat{\Omega}_m)m}\int_0^1 \frac{f_0'(r) r^{2m}\int_r^1 sf_0(s)ds}{(\Omega_m-\int_0^1 sf_0(sr)ds)} dr\\
\nonumber&+\frac{m+1}{m(\Omega_m-\widehat{\Omega}_m)}\int_0^1 \frac{f_0'(r) \int_0^r s^{2m+1}f_0(s)ds}{(\Omega_m-\int_0^1 sf_0(sr)ds)} dr\\
=:&\frac{1}{m(\Omega_m-\widehat{\Omega}_m)}\left(\mathbb{J}_{m,1}+\mathbb{J}_{m,2}+\mathbb{J}_{m,3}\right).
\end{align}
Let us check that $\mathbb{J}_{m,2}$ decays to  $0$ for large $m$.
Applying \eqref{nu102} with $\theta=0$ implies 
\begin{align*}
|\mathbb{J}_{m,2}|\leqslant&C\int_0^1 \frac{r^{2m}}{(1-r)}(1-r)dr\leqslant \frac{C}{m}.
\end{align*}
 Hence
$$
\lim_{m\rightarrow \infty}\mathbb{J}_{m,2}=0.
$$
Let us now deal  with $\mathbb{J}_{m,1}$. First recall  the elementary result 
$$
\lim_{m\rightarrow \infty} m\int_0^1 f_0'(r) r^{m+2}dr=f_0'(1)\neq 0,
$$
then 
$$
\lim_{m\rightarrow \infty}\mathbb{J}_{m,1}=f_0'(1).
$$
To deal with the last term $ \mathbb{J}_{m,3},$ we use first
\begin{align*}
\lim_{m\rightarrow \infty} m r^{-2m-2} {\int_0^r f_0(s) s^{2m+1}ds}
=& \lim_{m\rightarrow \infty} m  {\int_0^1 f_0(rs) s^{2m+1}ds}  \\
=& \lim_{m\rightarrow \infty} \frac{m}{m+1}  {\int_0^1 f_0(rs^\frac{1}{m+1}) sds}  \\
=&\tfrac12f_0(r),
\end{align*}
leading after a change of variables $\tau=r^{2m+3}$  to 
\begin{align*}
\lim_{m\rightarrow \infty} \mathbb{J}_{m,3}=&\lim_{m\rightarrow +\infty}\tfrac12\int_0^1 \frac{f_0'(r)f_0(r) r^{2m+2}}{\Omega_m-\kappa_2+\int_0^1 s[f_0(s)-f_0(sr)]ds}dr\\
=&\lim_{m\rightarrow \infty}\frac12 \frac{1}{2m+3}\int_0^1 \frac{f_0'(\tau^{\frac{1}{2m+3}})f_0(\tau^{\frac{1}{2m+3}}) }{\Omega_m-\kappa_2+\int_0^1 s[f_0(s)-f_0(s\tau^{\frac{1}{2m+3}})]ds}d\tau.
\end{align*}
Using once again \eqref{asym-diff}, it yields  
\begin{align*}
\lim_{m\rightarrow \infty}(2m+3)[f_0(s)-f_0(s\tau^{\frac{1}{2m+3}})]=&\lim_{m\rightarrow +\infty}(2m+3)s(1-\tau^{1/(2m+3})\int_0^1 f_0'\big(s+\tau_1 s(\tau^{1/(2m+3)}-1)\big)d\tau_1\\
&=-\ln(r)sf_0'(s).
\end{align*}
Putting  together the preceding results  we find
\begin{align*}
\lim_{m\rightarrow \infty} \mathbb{J}_{m,3}=-\tfrac12f_0'(1)f_0(1)\int_0^1 \frac{dr}{\ln(r)\int_0^1 s^2f_0'(s)ds+{{f_0(1)}}}\cdot
\end{align*}
Then
$$
\lim_{m\rightarrow \infty} \mathbb{J}_{m,1}+ \mathbb{J}_{m,2}+ \mathbb{J}_{m,3}=f_0'(1)\left\{1-\tfrac12f_0(1)\int_0^1 \frac{dr}{\ln(r)\int_0^1 s^2f_0'(s)ds+{f_0(1)}}\right\}.
$$
Assume for a while that
\begin{equation}\label{assump-trans}
{\kappa:=}1-\tfrac12f_0(1)\int_0^1 \frac{dr}{\ln(r)\int_0^1 s^2f_0'(s)ds+{f_0(1)}}\neq 0,
\end{equation}
and plugging this into \eqref{Dec-pui} yields for large $m$
\begin{align}\label{asym-yu1}
 \mathbb{I}_{m,3}^1
\approx&\frac{\kappa}{m(\Omega_m-\widehat{\Omega}_m)}\cdot
\end{align}
Coming back to \eqref{Dec-pui} and \eqref{Mrad1} we get the splitting  
\begin{align*}
\mathbb{I}_{m,3}^2
=:&\frac{1}{m(\Omega_m-\widehat{\Omega}_m)}\left(\widehat{\mathbb{J}}_{m,1}+\widehat{\mathbb{J}}_{m,2}+\widehat{\mathbb{J}}_{m,3}\right),
\end{align*}
with
$$
\forall i\in\{1,2,3\},\quad |\widehat{\mathbb{J}}_{m,i}\leqslant |{\mathbb{J}}_{m,i}|\sup_{r\in[0,1]}|F_{m,\Omega_m}(r)-1|.
$$
By virtue of Lemma \ref{lem-pot1}-(2) we get
$$
\forall r\in [0,1],\quad |F_{m,\Omega_m}(r)-1|\leqslant \frac{C_0}{\theta m(\Omega_m-\kappa_2)^\theta}.
$$
Using once again \eqref{asym-diff} and  Proposition \ref{prop-omegam-asymp}-(1) allows to get
$$
\forall r\in [0,1],\quad |F_{m,\Omega_m}(r)-1|\leqslant \frac{C_0}{\theta m^{1-\theta}}\cdot
$$
It follows that
$$
\forall i\in\{1,2,3\},\quad |\widehat{\mathbb{J}}_{m,i}|\leqslant C_0 \frac{|{\mathbb{J}}_{m,i}|}{\theta m^{1-\theta}},
$$
implying that
$$
\lim_{m\to \infty}\widehat{\mathbb{J}}_{m,i}=0
$$
Therefore we obtain from  \eqref{Mrad1} and \eqref{asym-yu1}
\begin{align}\label{Im-3}
\mathbb{I}_{m,3}&\approx \mathbb{I}_{m,3}^1\approx\frac{\kappa}{m(\Omega_m-\widehat{\Omega}_m)}\cdot
\end{align}
This implies  according to \eqref{asiyq} that
$$
|\mathbb{I}_{m,3}|\geqslant C m^{\frac12}.
$$
Thus, plugging \eqref{Im-1}, \eqref{Im-2}, \eqref{Im-3} into \eqref{Trans-V2} leads  for large $m$ to
\begin{align*}
\nonumber\mathbb{I}_{m}&\approx\mathbb{I}_{m,3},
\end{align*}
implying in turn that
$$
\lim_{m\to\infty}|\mathbb{I}_{m}|=\infty.
$${ 
Finally, let us check \eqref{assump-trans}. Using the change of variables $-\ln x= \mu\,\tau$ we get  
$$
\int_0^1\frac{dx}{\mu-\ln x}=\int_0^\infty\frac{e^{-\mu \tau}}{1+\tau}d\tau,\quad \forall  \mu>0.
$$
Set
${\mu=\frac{-f_0(1)}{\int_0^1 s^2f_0'(s)ds}},$ then
\begin{align*}
1-\tfrac12f_0(1)\int_0^1 \frac{dx}{\ln(x)\int_0^1 s^2f_0'(s)ds+{f_0(1)}}&=1-\tfrac12\mu\int_0^1\frac{dx}{\mu-\ln x}=1-\tfrac12 \mu\int_0^\infty\frac{e^{-\mu \tau}}{1+\tau}d\tau.
\end{align*}
Integration by parts yields
$$
1-\mu\int_0^\infty\frac{e^{-\mu \tau}}{1+\tau}d\tau=\int_0^\infty\frac{e^{-\mu \tau}}{(1+\tau)^2}d\tau,
$$
and therefore
\begin{align*}
1-\tfrac12f_0(1)\int_0^1 \frac{dr}{\ln(r)\int_0^1 s^2f_0'(s)ds+{f_0(1)}}&=\frac12+\frac12\int_0^\infty\frac{e^{-\mu t}}{(1+\tau)^2}d\tau>0.
\end{align*}
This achieves the proof of \eqref{assump-trans} and the  proof of the transversality condition is complete.}
\end{proof}

\section{Proof of Theorem \ref{th-intro}}\label{Sec-main-result}

In this section, we shall put together all the preceding results to prove the main result of this work: Theorem \ref{th-intro}. The proof is based on  the Crandall-Rabinowitz theorem \cite{rabinowitz:simple} and  all  its required  assumptions are checked in the previous sections. Here, we shall summarize the steps implemented before.  Let us analyze just the defocusing case i), and the other case follows similarly.\\
First, from Proposition \ref{prop-omegan-low} there exists $m_1\in\N$ with $m_1\geq \frac{1}{10(A[f_0]-1)}$ such that for any $m\in[2,m_1]\cap\N$, there exists $\Omega_m$ solution of $\zeta_m(\Omega)=0$. From now on, fix $m$ and $\Omega_m$ according to such proposition.\\
Then, from Section \ref{Sec-formulation} we know that the existence of nontrivial rotating solutions to the 2D Euler equations around generic equilibria agrees with the existence of nontrivial roots of the functional 
$$\widehat{G}:I\times B_{\mathscr{C}_{s,m}^{1,\alpha}(\D)}(0,\E)\to \mathscr{C}_{s,m}^{1,\alpha}(\D),
$$ 
which is well-defined and $\mathscr{C}^1$ due to Proposition \ref{prop-G-welldef}. The proof then relies on finding the roots of $\widehat{G}$ and thus applying the Crandall-Rabinowitz theorem ({see  \cite{rabinowitz:simple,kielhofer})}. \\
For that goal, one needs to check the spectral properties of $D_g \widehat{G}(\Omega_m,0)$. This operator is described in Fourier series in \eqref{lin-op}:
\begin{align*}
D_g \widehat{G}(\Omega_m,0)[h](re^{i\theta})=&\mathbb{L}_0^\Omega [h_0](r)+\sum_{n\in m\N^\star} \cos(n\theta)\mathbb{L}_{n}^{\Omega_m} [h_{n}](r).
\end{align*}
By Proposition \ref{compactop} we know that $D_g \widehat{G}(\Omega_m,0)$ is Fredholm of zero index and thus
$$
\mbox{dim }\mbox{Ker } D_g\widehat{G}(\Omega_m,0)=\mbox{dim }\mbox{Rang }Y\backslash D_g\widehat{G}(\Omega_m,0),
$$
and then we just need to check that the kernel is one dimensional. To study the kernel, we can do it for each Fourier mode. Note that the case $n=0$ was studied in Proposition \ref{radialfunctions} obtaining the invertibility of the operator. Now, fix $n=m$ and study
$$
\mathbb{L}_m^{\Omega_m}[h_m](r)=0.
$$
Notice that $m\leq \frac{1}{10(A[f_0]-1)}$ implies that 
$$
m\leq \frac{\kappa_1}{\kappa_2-\kappa_1}.
$$ 
Since we have chosen $\Omega_m$ to have $\zeta_m(\Omega_m)=0$, then using Proposition \ref{prop-element-kernel} one gets one element in the kernel of $D_g \widehat{G}(\Omega_m,0)$. Note that here we are dealing with Case 1 described in Section \ref{sec-kernel-gen}. Finally, the transversal condition is satisfied thanks to Proposition \ref{prop-trans-low}. {Consequently  Crandall-Rabinowitz theorem \cite{rabinowitz:simple} can be applied and so our statement is achieved.}

\bibliographystyle{plain}
\bibliography{references}

\end{document}